%% file: SubwordComplexFans.tex
\documentclass[a4paper,10pt,twoside]{amsart}

\usepackage{enumerate, amsmath, amsfonts, amssymb, amsthm, blkarray,wasysym, graphics, graphicx, caption, subcaption, xcolor, colortbl, url, hyperref, hypcap, ifthen, xargs, pdflscape, caption, tikz, rotate, multicol, multirow, paralist, a4wide}
\hypersetup{colorlinks=true, citecolor=darkblue, linkcolor=darkblue, urlcolor=darkblue}

\newtheorem{theorem}{Theorem}[section]
\newtheorem{proposition}[theorem]{Proposition}
\newtheorem{corollary}[theorem]{Corollary}
\newtheorem{lemma}[theorem]{Lemma}

\theoremstyle{definition}
\newtheorem{definition}[theorem]{Definition}
\newtheorem{example}[theorem]{Example}

\newtheorem{remark}[theorem]{Remark}

\newcommand{\R}{\mathbb{R}} 
\newcommand{\A}{\mathcal{A}} 

\newcommand{\subwordComplex}[1][Q,\pi]{\mathcal{SC}(#1)} 
\newcommand{\wo}{w_\circ} 

\graphicspath{{figures/}}

\definecolor{darkblue}{rgb}{0,0,0.7} 
\newcommand{\darkblue}{\color{darkblue}} 
\newcommand{\defn}[1]{\emph{\darkblue #1}} 

\usepackage{todonotes}

\def\sign{\operatorname{sign}}
\def\det{\operatorname{Det}}

\def\link{\operatorname{Link}}

\newcommand{\1}{{\bf 1}} 
\newcommand{\2}{{\bf 2}} 
\newcommand{\3}{{\bf 3}} 

\newcommand{\DiamondVector}{v_{\text{diamond}}} 
\newcommand{\fan}{\mathcal F_{Q,M}} 
\newcommand{\fanp}[1]{\mathcal F_{Q,#1}} 
\newcommand{\dualCountingMatrix}{D_{c,m}} 
\newcommand{\dualCountingMatrixp}[2]{D_{#1,#2}} 
\newcommand{\dualMatrix}{D_\varphi} 
\newcommand{\primalMatrix}{M_\varphi} 
\newcommand{\Obs}{\operatorname{Obs}(A_3)}

\newcommand\T{\rule{0pt}{2ex}}       
\newcommand\B{\rule[-2ex]{0pt}{0pt}} 

\numberwithin{equation}{section}

\title[Fan realizations of subword complexes and multi-associahedra]{Fan realizations of subword complexes and multi-associahedra via Gale duality}

\author[Bergeron, Ceballos, Labb\'e]{Nantel Bergeron$^{1,3}$, Cesar Ceballos$^{1,3}$, Jean-Philippe Labb\'e$^{2}$}

\address[1]{Fields Institute\\ Toronto, ON, Canada}
\address[2]{Freie Universit\"at Berlin, Berlin, Germany}
\address[3]{York University\\ Toronto, ON, Canada}

\email{bergeron@yorku.ca}
\email{ceballos@mathstat.yorku.ca}
\email{labbe@math.fu-berlin.de}

\thanks{
The first author was partially  supported by NSERC.\\
The second author was supported by the government of Canada through an NSERC Banting Postdoctoral Fellowship. He was also supported by a York University research grant.\\
The third author was supported by a FQRNT Doctoral scholarship and SFB Transregio ``Discretization in Geometry and Dynamics'' (TRR 109).
}

\date{\today}

\begin{document}
\maketitle

\begin{abstract}
We present complete simplicial fan realizations of any spherical subword complex of type~$A_n$ for $n\leq 3$. This provides complete simplicial fan realizations of simplicial multi-associahedra~$\Delta_{2k+4,k}$, whose facets are in correspondence with $k$-triangulations of a convex $(2k+4)$-gon. This solves the first open case of the problem of finding fan realizations where polytopality is not known. 
The techniques presented in this paper work for all finite Coxeter groups and we hope that they will be useful to construct fans realizing subword complexes in general. In particular, we present fan realizations of two previously unknown cases of subword complexes of type $A_4$, namely the multi-associahedra $\Delta_{9,2}$ and $\Delta_{11,3}$.
\end{abstract}

%
%
\section{Introduction}

Subword complexes are simplicial complexes introduced by Knutson and Miller in~\cite{KnutsonMiller-subwordComplex} motivated from the study of Gr\"obner geometry of Schubert varieties~\cite{knutson_grobner_2005}.  They proved that any subword complex is either a topological ball or sphere~\cite[Corollary 3.8]{KnutsonMiller-subwordComplex}, and asked the question of whether every spherical subword complex can be realized as the boundary complex of a simplicial convex polytope~\cite[Question 6.4]{KnutsonMiller-subwordComplex}. The answer to this question has been verified to be positive only for a few cases, which include interesting families of polytopes such as all even dimensional cyclic polytopes~\cite{ceballos_subword_2013}, the duals of $c$-generalized associahedra~\cite{ceballos_subword_2013}, the pseudotriangulation polytope of any planar point set in general position~\cite{RoteSantosStreinu-pseudotriangulationPolytope} and the brick polytopes of ``root-independent subword complexes"~\cite{PilaudSantos-brickPolytope, pilaud_brick_2011}.  
Another family of closely related simplicial complexes is the family of multi-associahedra. Given any two positive integers $k\geq 1$ and $\ell \geq 2k+1$, the (simplicial) multi-associahedron $\Delta_{\ell,k}$ is a simplicial complex whose facets correspond to $k$-triangulations of a convex $\ell$-gon~\cite{jonsson_generalized_2005}. 
This simplicial complex is conjectured to be realizable as the boundary complex of a polytope~\cite[Section 1.2]{jonsson_generalized_2005}. 
The only cases for which this conjecture has been verified are summarized in Table~\ref{tab:polytopality_multi-associahedra}. We refer to~\cite{ceballos_associahedra} and the recent book~\cite{mueller_associahedra_2012} for background and history about the special case of the classical associahedron.

\begin{small}
\begin{table}[!htbp]
\begin{center}
\begin{tabular}{|c|l|}
\hline
&\\[-.13in]
$\Delta_{\ell,k}$ & Realizable as a\\[0.05in]
\hline
&\\[-.13in]
$k=1$ & \text{dual of a classical associahedron} \\
$\ell=2k+1$ & \text{single vertex}  \\
$\ell=2k+2$ & \text{simplex}  \\
$\ell=2k+3$ & \text{cyclic polytope} \cite{PilaudSantos-multitriangulations,ceballos_subword_2013} \\
$\ell=2k+4$ & \text{complete simplicial fan} (\defn{this paper})\\
$\Delta_{8,2}$ & \text{6-dimensional polytope} \cite{bokowski_symmetric_2009,ceballos_associahedra_2012} (\defn{this paper}) \\
$\Delta_{9,2}$ & \text{complete simplicial fan} (\defn{this paper})\\
$\Delta_{11,3}$ & \text{complete simplicial fan} (\defn{this paper})\\
\hline
\end{tabular}
\vspace{.05in}
\end{center}
\caption{Known results about polytopality of multi-associahedra.}
\label{tab:polytopality_multi-associahedra} 
\vspace{-15pt}
\end{table}
\end{small}

\noindent
Subword complexes and multi-associahedra turn out to be quite related: every multi-associahedron can be obtained as a well chosen subword complex of type $A$~\cite{PilaudPocchiola, Stump}, see also~\cite{serrano_maximal_2012}. Conversely, the family of multi-associahedra is universal, in the sense that every spherical subword complex of type $A$ can be obtained as the link of a face in a multi-associahedron~\cite[Proposition~5.6]{PilaudSantos-brickPolytope}, see also~\cite[Theorem 2.15]{ceballos_subword_2013}. 
The relation between these two families of simplicial complexes has been extended to arbitrary finite Coxeter groups in~\cite{ceballos_subword_2013}. 

In this paper, we find complete simplicial fan realizations for any spherical subword complex of type~$A_n$ for~$n\leq 3$.
This is particularly interesting for the case $n=3$, where the answer to the question of polytopality is not known. In particular, we obtain complete simplicial fan realizations of multi-associahedra~$\Delta_{\ell,k}$ for $\ell \leq 2k+4$, which solves the previously unknown case when $\ell=2k+4$. 
Our constructions also lead to a large connected component of the fan realization space in each case. We emphasize that even for the classical associahedron very little is known about the space of polytopal realizations~\cite{ceballos_realizing_2012}. Most of the known constructions of associahedra are very punctual and combinatorial, and it is not even known if the space of realizations is connected. 
We hope that the results and techniques in this paper will represent a significant advance in this direction.
It is also natural to ask whether the complete simplicial fans constructed here are normal fans of polytopes. Although it is often true for $k\leq 2$, we do not know if any of our fan realizations for~$\Delta_{10,3}$ is the normal fan of a polytope. The multi-associahedron~$\Delta_{10,3}$ is a 3-neighborly simplicial complex of dimension $8$ with $f$-vector $(1, 15, 105, 455, 1320, 2607, 3465, 2970, 1485, 330)$. We tested $144139$ different fans realizing it among the infinitely many produced by our construction, and none of them is the normal fan of a polytope, see Table~\ref{tab:polytopality_A3}. Nevertheless, the polytopality question of multi-associahedra remains open in general, see Section~\ref{sec:polytopality} for more details.

The construction of the fans in this paper can be applied for any subword complex of type~$A_n$. Although it does not produce the right fans for $n\geq 4$, it seems to be close to a valid construction, at least for small values of $k$ and $n$. We argue this fact in Section~\ref{sec:typeA4}, where we apply a slight modification to our construction to produce fan realizations of two previously unknown cases of subword complexes of type~$A_4$, namely the multi-associahedra~$\Delta_{9,2}$ and~$\Delta_{11,3}$. We remark that these two cases are far from trivial. In particular, they are intractable with computational methods previously used in the literature~\cite{bokowski_symmetric_2009,ceballos_associahedra_2012}. The corresponding~$f$-vectors are presented in Table~\ref{tab:f-vectorA4}. 
Interestingly, we show that the two fans we present for~$\Delta_{9,2}$ and~$\Delta_{11,3}$ can not be obtained as the normal fan of a polytope. 

\begin{footnotesize}
\begin{table}[!htdp]
\begin{tabular}{|c|c|c|l|}
\hline
&&\\[-.13in]
$\Delta_{\ell,k}$ & neighborliness & dim. & $f$-vector\\[0.05in]
\hline
$\Delta_{9,2}$ & 2-neighborly &7& $(1, 18, 153, 732, 2115, 3762, 4026, 2376, 594)$ \\
$\Delta_{11,3}$ & 3-neighborly &11& $(1, 22, 231, 1540, 7150, 23958, 58751, 105534, 137280, 125840, 77077, 28314, 4719)$ \\
\hline
\end{tabular}
\caption{$f$-vectors of two multi-associahedra of type $A_4$.}
\label{tab:f-vectorA4}
\end{table}%
\end{footnotesize}

The main ideas of this paper are based on results presented in Ceballos's doctoral thesis~\cite[Chapter~3]{ceballos_associahedra_2012}. The methods presented here are developed using Coxeter group theory. We expect that this new point of view will be useful to obtain complete simplicial fan realizations of spherical subword complexes in general.

\subsection*{Acknowledgements}
The authors are grateful to 
Bruno Benedetti, 
Frank Lutz,
Thomas McConville, and
Vic Reiner
for important conversations that influenced the results in this paper. They are specially grateful to Darij Grinberg for his important comments about Section~\ref{sec:sign_megabipartite}, and to Francisco Santos for his polytopal construction in Example~\ref{ex:Santos}. We are also grateful to Vincent Pilaud and Vic Reiner for their comments on previous versions of this paper.

%

%
%
\section{Definitions and main results}

Let~$(W,S)$ be a finite Coxeter system. Let~$Q=(q_1,\dots, q_r)$ be a word in the generators~$S$ and~$\pi\in W$ be an element of the group. 

\begin{definition}[Knutson--Miller \cite{KnutsonMiller-subwordComplex}]
The \defn{subword complex $\subwordComplex$} is the simplicial complex whose faces are subsets~$I\subset [r]$ such that the subword of~$Q$ with positions at~$[r] \smallsetminus I$ contains a reduced expression for~$\pi$.
\end{definition}

As mentioned above, a subword complex is either a topological ball or sphere. Moreover, it was proven in \cite[Theorem 3.7]{ceballos_subword_2013} that every spherical subword complex is isomorphic to a subword complex of the form~$\subwordComplex[Q,\wo]$, where the element $\pi=\wo$ is the longest element of the group. Therefore, we restrict our study to subword complexes of this form and write~$\subwordComplex[Q]$ for simplicity. 

In this paper we are interested in constructing complete simplicial fan realizations of spherical subword complexes. We denote by~$N:=\ell(\wo)$ the length of the longest element of the group. 

\begin{definition}\label{def:fan}
Given a word $Q$ and a matrix $M\in \R^{(r-N)\times r}$, we define a natural collection of cones~$\fan$ in~$\R^{r-N}$. Its rays are given by the column vectors of $M$ and its cones are spanned by the columns corresponding to faces of subword complex $\subwordComplex[Q]$. The notation $\fan$ is extensively used throughout the paper.
\end{definition}

Although the techniques developed in this paper work for arbitrary finite Coxeter groups, part of our main results are devoted to the particular case of Coxeter groups of type~$A$.
Let $W$ be a Coxeter group of type $A$ generated by the set $S=\{s_1,\dots ,s_n\}$ of simple transpositions $s_i =(i\ \ i+1)$. Moreover, let $\Delta$ be the corresponding set of simple roots and $\Phi=W\cdot\Delta=\Phi^+\sqcup\Phi^-$ be the set of roots partitioned into the positive and negative roots respectively. Let~$c=(s_1,\dots ,s_n)$ be a Coxeter element, and $P_m = c^m = (p_1,\dots , p_{\widetilde r})$ be a sufficiently long word that contains Q as a subword. 
The number $\widetilde r = mn$ denotes the number of letters in~$P_m$.
The main ingredient in our construction is a counting matrix $\dualCountingMatrix$, whose entries count the number of reduced expressions of~$c$ in~$P_m$ containing the letter $p_i$, in position $i$, after restricting to standard parabolic subgroups.

\begin{definition}\label{def:dual_counting_matrix}
The \defn{counting matrix $\dualCountingMatrix$} is a $(N \times \widetilde r)$-matrix whose rows correspond to positive roots and columns to the positions $1\le j\le \widetilde r$ of the letters of~$P_m=c^m$. Given $\alpha\in\Phi^+$ and $1\le j\le \widetilde r$, denote by~$S_\alpha\subset S$ the subset of generators whose corresponding simple roots are used in the unique decomposition of the root~$\alpha$ in the basis $\Delta$ and by~$c_\alpha$ the restriction of~$c$ to the generators in $S_\alpha$. The entry $d_{\alpha,j}$ of $\dualCountingMatrix$ is the number of reduced expressions of $c_\alpha$ in $P$ (copies of~$c_\alpha$ up to commutations) using the letter $p_j$ in position $j$. In particular, if $p_j\notin S_\alpha$, then $d_{\alpha,j}=0$.
\end{definition}

\begin{example}
Let $P_3=c^3=(s_1,s_2,s_1,s_2,s_1,s_2)$ be a word of type $A_2$. The counting matrix is 
\[
\dualCountingMatrixp{c}{3} 
=
\left(
\begin{array}{cccccc}
  1& 0  & 1 & 0  & 1 & 0      \\
 3 &  1 & 2 & 2  & 1 & 3     \\
  0 & 1  & 0 & 1  & 0 & 1     
\end{array}
\right)
\]
where the rows correspond to the positive roots \{$\alpha_1$, $\alpha_1+\alpha_2$,  $\alpha_2$\} in this order. For example, if~$\alpha=\alpha_1$ then~$P_\alpha=(s_1,s_1,s_1)$ and~$c_\alpha=s_1$. Therefore, $d_{\alpha_1,j}=1$ if $p_j=s_1$ and $d_{\alpha_1,j}=0$ if $p_j=s_2$. For $\alpha=\alpha_1+\alpha_2$, we have $P_\alpha=P_3$ and $c_\alpha=c$. Now $d_{\alpha,1}=3$ since we must use $s_1=p_1$ in position 1 and we have three choices of $s_2$ to the right of position $1$. More general formulas for counting matrices of type $A_n$ with $n\leq 3$ are presented in Appendix~\ref{app:matrices}.
\end{example}

For any embedding $\varphi: Q \rightarrow c^m$ of $Q$ into a sufficiently long word $P_m=c^m$ we construct, in Theorem~\ref{thm:complete_fan}, a complete simplicial fan realization of the spherical subword complex~$\subwordComplex[Q]$ . This embedding can be thought as a map $\varphi: [r] \rightarrow [\widetilde r]$ from positions in $Q$ to positions in $P_m$.

\begin{definition}
\label{def:dual_matrix}
The \defn{restricted matrix} $\dualMatrix$ is the restriction of $\dualCountingMatrix$ to the columns $\varphi(1),\dots ,\varphi(r)$ corresponding to the positions of the letters of $Q$ embedded in $c^m$.
\end{definition}

\begin{example}
Let $Q=(q_1,\dots,q_5)=(s_1,s_2,s_2,s_1,s_1)$ be a word of type $A_2$. Let~$c=(s_1,s_2)$ and~$P_4 =c^4= ({\bf s_1},{\bf s_2},s_1,{\bf s_2},{\bf s_1},s_2,{\bf s_1},s_2,)$. Here, the letters in bold correspond to the letters of~$Q$ embedded in $c^4$. The dual restricted matrix is 
\[
\dualMatrix=
\left(
\begin{array}{ccccc}
  1& 0  & 0 & 1  & 1       \\
 4 &  1 & 2 & 2  & 1      \\
  0 & 1  & 1 & 0  & 0      
\end{array}
\right)
\]
where the rows correspond to the positive roots \{$\alpha_1$, $\alpha_1+\alpha_2$,  $\alpha_2$\} in this order.
\end{example}

Given a full rank matrix $A$, we say that $B$ is a \defn{Gale dual matrix} of $A$ if the rows of $B$ form a basis for the kernel of $A$, see \cite[Definition~4.1.35]{de_loera_triangulations_2010}. This dual matrix is determined up to linear transformation of the rows. Let~$\primalMatrix$ be a Gale dual matrix of~$\dualMatrix$. This matrix is the key ingredient in the following main theorem, which is proven in Section~\ref{sec:mainproof}.

\begin{theorem}\label{thm:complete_fan}
Let~$\subwordComplex[Q]$ be a spherical subword complex of type~$A_n$ with $n\leq 3$, and $\varphi$ be an embedding of $Q$ into $c^m$. The fan $\fanp{\primalMatrix}$ is a complete simplicial fan realization of~$\subwordComplex[Q]$.
\end{theorem}

This result provides infinitely many complete simplicial fan realizations of any spherical subword complex of type~$A_n$ with $n\leq 3$. In the particular cases where the word $Q=c^m$ is naturally embedded into itself we get explicit realizations. In the corollaries below, the entries of the matrices depend on the functions $S(i)=i^2$ and $T(i)=i(i+1)/2$. 

\begin{corollary}\label{cor:fan_bipartite213}
Let~$c=(s_2,s_1,s_3)$ be a bipartite Coxeter element of type $A_3$ and $Q=c^m$, with $m\geq 3$. The fan~$\fanp{M_{213,m}}$ is a complete simplicial fan realization of~$\subwordComplex[Q]$ for the matrix $M_{213,m}$ below.
\end{corollary} 

\begin{minipage}{.4\textwidth}
\footnotesize
\[
	M_{213,m}^t=\left(\begin{array}{c}
	\\[1.5ex]
	-I_{3m-6} \\
	\\[1.5ex]\hline
	\begin{array}{c@{\hspace{0.25cm}}|@{\hspace{0.5cm}}c@{\hspace{0.5cm}}|@{\hspace{0.25cm}}c}
	&& \\[1ex]
	B_{213,m-2} & \cdots & B_{213,1}\\
	&& \\[1ex]
	\end{array}
	\end{array}\right)_{(3m)\times (3m-6)}
\]
\end{minipage}
\qquad
\begin{minipage}{.4\textwidth}
\tiny
\[
	\tiny
	B_{213,i}:=\left(\begin{array}{rrr}
	S(i+1) & -T(i) & -T(i) \\
	2T(i) & -T(i-1)+1 & -T(i) \\
	2T(i) & -T(i) & -T(i-1)+1 \\
	-S(i+1)+1 & T(i) & T(i) \\
	-2T(i) & T(i-1) & T(i) \\
	-2T(i) & T(i) & T(i-1) \\ 
	\end{array}\right)
\]
\end{minipage}

\medskip
\begin{corollary}\label{cor:fan_123}
Let~$c=(s_1,s_2,s_3)$ be a Coxeter element of type $A_3$ and $Q=c^m$, with $m\geq 3$. The fan~$\fanp{M_{123,m}}$ is a complete simplicial fan realization of~$\subwordComplex[Q]$ for the matrix $M_{123,m}$ below.
\end{corollary} 

\begin{minipage}{.4\textwidth}
\footnotesize
\[
	M_{123,m}^t=\left(\begin{array}{c}
	\\
	\begin{array}{@{\hspace{1.5cm}}c@{\hspace{1.5cm}}|@{\hspace{0.0cm}}r@{\hspace{-0.5cm}}}
	  -I_{3m-5}& \begin{array}{c} 0\\ \vdots \\ 0\end{array} \\
	  \end{array}
	\\[1.5ex]\hline
	\begin{array}{c@{\hspace{0.25cm}}|@{\hspace{0.5cm}}c@{\hspace{0.5cm}}|@{\hspace{0.25cm}}r}
	&& \\[1ex]
	B_{123,m-2} & \cdots & B_{123,1}\\
	&& \\[1ex]
	\end{array}\\\hline
	0 \hspace{1.75cm} \dots \hspace{1.5cm} 0\quad -1\\
	\end{array}\right)_{(3m)\times (3m-6)}
\]
\end{minipage}
\hspace{1cm}
\begin{minipage}{.4\textwidth}
\tiny
\[
	\tiny
	B_{123,i}:=\left(\begin{array}{rrr}
	T(i) & -2T(i) & T(i)^* \\
	T(i+1) & -2T(i) & T(i-1)^* \\
	T(i) & -S(i)+1 & T(i-1)^* \\
	-T(i) & 2T(i) & -T(i)+1^* \\
	-T(i+1)+1 & 2T(i) & -T(i-1)^* \\
	-T(i) & S(i) & -T(i-1)^* \\ 
	\end{array}\right),
\]
\end{minipage}
\vspace{10pt}

The star denotes the exception that for $i=1$ the last column is given by $(0,1,1,1,-1,-1)^t$.

\bigskip
\begin{example}[3-dimensional associahedron]
Let~$c=(s_2,s_1,s_3)$ be a bipartite Coxeter element of type $A_3$ and $Q=c^3$. The subword complex~$\subwordComplex[Q]$ is isomorphic to a 3-dimensional simplicial associahedron. It can be realized as a complete simplicial fan~$\fanp{M_{213,3}}$ where
\[
\small
	M_{213,3}=\left(\begin{array}{rrrrrrrrr}
	-1 & 0 & 0 & 4 & 2 & 2 & -3 & -2 & -2 \\	
	 0 & -1 & 0 & -1 & 1 & -1 & 1 & 0 & 1 \\	
	 0 & 0 & -1 & -1 & -1 & 1 & 1 & 1 & 0  
	\end{array}\right)
\]
The rays of this fan are given by the column vectors of the matrix, and the cones are spanned by column vectors corresponding to faces of the subword complex~$\subwordComplex[Q]$.
Alternatively, the column vectors correspond to the diagonals $\{14,15,24,25,26,35,36,46,13\}$ of an hexagon (with vertices labeled cyclically from 1 up to 6) and the cones of the fan to its subdivisions.

Theorem~\ref{thm:complete_fan} actually provides infinitely many fan realizations of the 3-dimensional simplicial associahedron, one for each appropriate word $Q$ and embedding $\varphi$. Computational experiments show that most of these fans are the normal fan of a polytope, see Table~\ref{tab:polytopality_A3}. In Figure~\ref{fig:associahedra_3d}, we illustrate eight different polytopal realizations and remark that they are not equivalent in general, say up to affine linear transformations. 
\end{example}

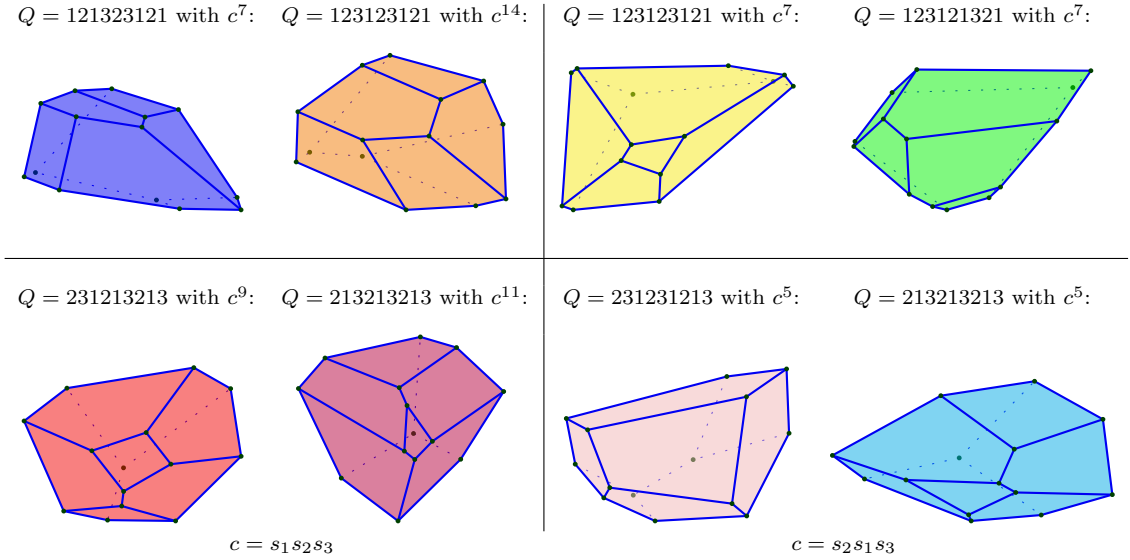
\begin{figure}[!htbp]
	\centering
	\footnotesize
	\begin{tabular}{cc@{\hspace{0.2cm}}|@{\hspace{0.2cm}}cc}
		$Q=121323121$ with $c^7$: & $Q=123123121$ with $c^{14}$: & $Q=123123121$ with $c^7$: & $Q=123121321$ with $c^{7}$: \\[1em]
		\input{figures/Associahedra/Asso1.tex} & \input{figures/Associahedra/Asso2.tex} & \input{figures/Associahedra/Asso5.tex} & \input{figures/Associahedra/Asso6.tex}\\[0.5cm]\hline
		&  &  & \\
		$Q=231213213$ with $c^9$: & $Q=213213213$ with $c^{11}$: & $Q=231231213$ with $c^5$: & $Q=213213213$ with $c^{5}$: \\[1em]
		\input{figures/Associahedra/Asso3.tex} & \input{figures/Associahedra/Asso4.tex} & \input{figures/Associahedra/Asso7.tex} & \input{figures/Associahedra/Asso8.tex}\\
		\multicolumn{2}{c}{$c=s_1s_2s_3$} & \multicolumn{2}{c}{$c=s_2s_1s_3$}
	\end{tabular}
	\caption{Eight different polytopal realizations of the 3-dimensional associahedron using embeddings of a word $Q$ into some $c^m$ for the Coxeter elements~$c=s_1s_2s_3$ and~$c=s_2s_1s_3$.}\label{fig:associahedra_3d}
\end{figure}

Corollary~\ref{cor:fan_bipartite213} can also be restated in terms of multi-associahedra. Before giving a precise statement let us recall some basic definitions. 
Let $k\geq 1$ and $\ell \geq 2k+1$ be two positive integers. We say that a set 
of $k+1$ diagonals of a convex $\ell$-gon forms a \defn{$(k+1)$-crossing} 
if all the diagonals in this set are pairwise crossing. A diagonal is called 
\defn{$k$-relevant} if it is contained in some $(k+1)$-crossing, that is, if 
there are at least $k$ vertices of the $\ell$-gon on each side of the diagonal.
The \defn{simplicial multi-associahedron} $\Delta_{\ell,k}$ is the simplicial
complex of $(k+1)$-crossing-free sets of $k$-relevant diagonals of a convex 
$\ell$-gon.
We are particularly interested in the case $\ell=2k+4$, where the polytopality conjecture is still open, see Table~\ref{tab:polytopality_multi-associahedra}. In this case, there are exactly $3k+6$ $k$-relevant diagonals, which are in correspondence with the columns of the matrix~$M_{213,k+2}$ as follows:
cyclically label the vertices of the $\ell$-gon from 1 up to $\ell$. The first $k$-relevant diagonal in lexicographic order corresponds to the last column of~$M_{213,k+2}$, the other $k$-relevant diagonals correspond to the other columns in the order they appear in lexicographic order, see Example~\ref{ex:multi_10_3}. Let $\mathcal F_{k}$ be the simplicial fan in $\R^{3k}$ whose rays are the column vectors of~$M_{213,k+2}$, and whose cones are spanned by the column vectors corresponding to faces of $\Delta_{2k+4,k}$. Using this terminology, Corollary~\ref{cor:fan_bipartite213} can be read as follows. 

\begin{corollary}
The fan $\mathcal F_{k}$ is a complete simplicial fan realization of the simplicial multi-associahedron $\Delta_{2k+4,k}$.
\end{corollary}

\begin{example}[Multi-associahedron $\Delta_{10,3}$]
\label{ex:multi_10_3}
The multi-associahedron $\Delta_{10,3}$ can be realized as the complete simplicial fan~$\mathcal F_{3}$. The rays are the column vectors of the matrix~$M_{213,5}$ below, and the cones are spanned by the column vectors corresponding to faces of $\Delta_{10,3}$. The pairs of numbers on top of the matrix are the 3-relevant diagonals of the 10-gon associated to each of the columns of the matrix. 

{\small
\[M_{213,5}=
\begin{blockarray}{rrrrrrrrrrrrrrr}
1,6 & 1,7 & 2,6 & 2,7 & 2,8 & 3,7 & 3,8 & 3,9 & 4,8 & 4,9 & 4,10 & 5,9 & 5,10 & 6,10  & 1,5\\
\begin{block}{(rrrrrrrrrrrrrrr)}
-1 &  0 &  0 &  0 &  0 &  0 &  0 &  0 &  0 & 16 & 12  & 12 & -15  & -12  & -12  \\
0 &  -1 &  0 &  0 &  0 &  0 &  0 &  0 &  0 & -6 & -2  & -6 & 6  & 3  & 6  \\
0 &  0 &  -1 &  0 &  0 &  0 &  0 &  0 &  0 & -6 & -6  & -2 & 6  & 6  & 3  \\
0 &  0 &  0 &  -1 &  0 &  0 &  0 &  0 &  0 &9  & 6  & 6 & -8  & -6  & -6  \\
0 &  0 &  0 &  0 &  -1 &  0 &  0 &  0 &  0 & -3 & 0  & -3 & 3  & 1  & 3  \\
0 &  0 &  0 &  0 &  0 &  -1 &  0 &  0 &  0 & -3 & -3  &0  & 3  & 3  & 1  \\
0 &  0 &  0 &  0 &  0 &  0 &  -1 &  0 &  0 & 4 & 2  & 2 & -3  & -2  & -2  \\
0 &  0 &  0 &  0 &  0 &  0 &  0 &  -1 &  0 & -1 & 1  & -1 & 1  & 0  & 1  \\
0 &  0 &  0 &  0 &  0 &  0 &  0 &  0 &  -1 & -1 & -1  & 1 & 1  & 1  & 0 \\
\end{block}
\end{blockarray}
 \]}
 


\noindent
Remarkably, this fan can not be obtained as the normal fan of polytope. Even more surprising, using different embeddings of an appropriate word into $c^m$, we tested more than a hundred thousand different fans realizing $\Delta_{10,3}$ and none of them turned out to be the normal fan of a polytope, see Table~\ref{tab:polytopality_A3}. We refer to Section~\ref{sec:polytopality} for more details.  
\end{example}

In Section~\ref{sec:realization_space}, we present real parameter generalizations of Corollaries~\ref{cor:fan_bipartite213} and~\ref{cor:fan_123}. These produce a large connected component of the fan realization space of spherical subword complexes of type~$A_3$ and, in particular, of multi-associahedra~$\Delta_{2k+4,k}$. The realization space is an important concept of great interest not only in the literature~\cite{richter-gebert_realization_1996}, but also in the particular case of the associahedron and its relatives~\cite{ceballos_realizing_2012}.

Using the folding technique presented in \cite[{Theorem~2.10 and Section~6.3}]{ceballos_subword_2013}, we obtain the following corollary.
This is an application of the principles to obtain the cyclohedron from the associahedron used by Hohlweg and Lange in \cite{HohlwegLange}.

\begin{corollary}\label{cor:fan_typeb_12}
Let~$c=(s_1,s_2)$ be a Coxeter element of type $B_2$ and $Q=c^m$. The fan~$\fanp{M_{12,m}}$ is a complete simplicial fan realization of~$\subwordComplex[Q]$ for the matrix $M_{12,m}$ below.
\end{corollary} 

\begin{minipage}{.4\textwidth}
\footnotesize
\[
	M_{12,m}^t=\left(\begin{array}{c}
	\\[1.5ex]
	-I_{2m-4} \\
	\\[1.5ex]\hline
	\begin{array}{c@{\hspace{0.25cm}}|@{\hspace{0.5cm}}c@{\hspace{0.5cm}}|@{\hspace{0.25cm}}c}
	&& \\[1ex]
	B_{12,m} & \cdots & B_{12,1}\\
	&& \\[1ex]
	\end{array}
	\end{array}\right)_{(3m)\times (3m-6)}
\]
\end{minipage}
\quad
\begin{minipage}{.4\textwidth}
\footnotesize
\[
	\small
	B_{12,i}:=\left(\begin{array}{rr}
	S(i+1) & -T(i) \\
	4T(i) & -S(i)+1 \\
	-S(i+1)+1 & T(i) \\
	-4T(i) & S(i) \\
	\end{array}\right).
\]
\end{minipage}

\begin{proof}
	Consider the matrix $B_{213,i}$ of type $A_3$, and replace the second and third row by the row given by their sum. Apply the same procedure to the fifth and sixth row. The matrix we obtain has dimension $4\times 3$, but the last two columns are equal to each other. Finally, remove the last column to obtain $B_{12,i}$. The type $B$ fan associated to this matrix describes exactly the intersection of $\fanp{M_{213,m}}$ with a hyperplane splitting symmetrically $\fanp{M_{213,m}}$ according to the coordinates of consecutive letters $s_1$ and $s_3$. 
\end{proof}
%

%
%
\section{The sign function and mega bipartite graphs}\label{sec:sign_megabipartite}

This section contains a systematic study of bipartite properties of the graph of reduced expressions of an element in a finite Coxeter group. It is independent of the rest of the paper and uses a sign function on the vertices of the graph as the main tool in our proofs. In the case of type~$A_n$, this sign function also appears in a connection between scattering amplitudes in physics and the positive Grassmannian~\cite{postnikov_sign_physics_2012}. Many of the ideas and notation are from~\cite{reiner_diameter_2013, humphreys_reflection_1992} and the references therein. 

\subsection{Graph of reduced expressions of $w$}
Given a finite Coxeter group~$W$ and an element $w\in W$, we consider the graph $G(w)$ 
of reduced expressions of $w$ connected by braid relations. More precisely,
we denote by~$s_i$ the generators of~$W$ such that~$s_i^2=Id$. They satisfy the relations ~$ (s_is_j)^{m_{ij}}=Id$ for some positive integers~$m_{ij}=m_{ji}$.
This can be rewritten as the braid relation
\begin{equation}\label{eq:braid}
 \underbrace{s_i s_j s_i \ldots}_{m_{ij}} =  \underbrace{s_j s_i s_j \ldots}_{m_{ij}} \,.
\end{equation}
The vertices of the graph~$G(w)$ are all the reduced expressions of $w$ in terms of the generators $s_i$. Two reduced expressions are connected by an edge if and only if they are related by a single braid relation~\eqref{eq:braid}. The left part in Figure~\ref{fig:Mega_bipartite_graph} illustrates an example of the graph~ $G(\wo)$ of reduced expressions of the longest element in the symmetric group~$W=S_4$. The 16 vertices are labeled by the subscript sequence of the 16 reduced words of~$\wo$, for example 123121 represents the reduced expression $s_1s_2s_3s_1s_2s_1$. The edges are labelled with the pair of indices $\{i,j\}$ of the corresponding braid relation $m_{ij}$.

It is well known by a theorem of Tits~\cite{tits_probleme_1969} (see also~\cite[Theorem 3.3.(ii)]{bjorner_combinatorics_2005}) that for any finite Coxeter group $(W, S)$ and any $w$ in $W$, the graph $G(w)$ is connected. The main result of this section shows that $G(w)$ is a \emph{mega bipartite graph}, in the sense that any graph obtained from it by contracting the edges corresponding to a specified set of braid relations is a bipartite graph. In particular, $G(w)$ is as well bipartite. This property is illustrated in Figure~\ref{fig:Mega_bipartite_graph} for the Coxeter group of type $A_3$. The graph obtained by contracting edges of $G(\wo)$ corresponding to commutations in type $A$ has been studied by several authors in the context of \emph{higher Bruhat order} $B(n,2)$~\cite{manin_arrangement_1989,ziegler_higher_1993,shapiro_connected_1997,felsner_theorem_2000}, and in various types in connection with rhombic tilings of polygons in~\cite{elnitsky_rhombic_1997}. 

In order to make this statement more precise, we need some definitions.
We say that two pairs of integers $\{i,j\},\{i',j'\}\in [n]\times [n]$ are \defn{conjugated}  if $s_{i'}=w^{-1}s_ iw$ and $s_{j'}=w^{-1}s_ jw$ for some~$w\in W$. We say that $\{i,j\}$ and $\{i',j'\}$ are in the same \defn{automorphism class} if $\{i',j'\}$ is the image of $\{i,j\}$ via an automorphism of $W$. Any outer automorphism of $W$ is conjugated, via an inner automorphism, to the automorphism graph of the Coxeter graph of $W$. In all cases we have that~$m_{ij}=m_{i'j'}$.

Let $Z=\big\{\{i_1,j_1\}, \dots , \{i_\ell,j_\ell\}\big\}$ be a subset of $\big\{ \{i,j\}: 1\le i< j\le n\big\}$. 
We say that $Z$ is \defn{stabled} if for any $\{i,j\}\in Z$ and any $\{i',j'\}$ image of $\{i,j\}$ via an automorphism of $W$, the pair~$\{i',j'\}$ is also in $Z$.
For any stabled subset $Z$, let $G^Z(w)$ be the graph obtained from $G(w)$ by contracting all the edges corresponding to braid relations $m_{ij}$ for $\{i,j\}\notin Z$.  Two interesting cases occur when~$Z$ consists of the pairs $\{i,j\}$ for which $m_{ij}$ is even or odd (these two cases are clearly stabled). We denote by $G^\text{even}(w)$  (respectively $G^\text{odd}(w)$) the graph obtained by contracting edges corresponding to non-even braid relations $m_{ij}$ (respectively non-odd braid relations $m_{ij}$).

\begin{theorem}\label{thm:mega_bipartite}
Let $W$ be a finite Coxeter group and $w$ be an element in $W$. Then, for any stabled set~$Z$ of specified braid relations, the contracted graph~$G^Z(w)$ is a bipartite graph. In particular, the graphs $G(w)$, $G^\text{even}(w)$ and $G^\text{odd}(w)$ are bipartite.
\end{theorem}

\begin{figure}[htbp]
\begin{center}
\includegraphics[width=0.95\textwidth]{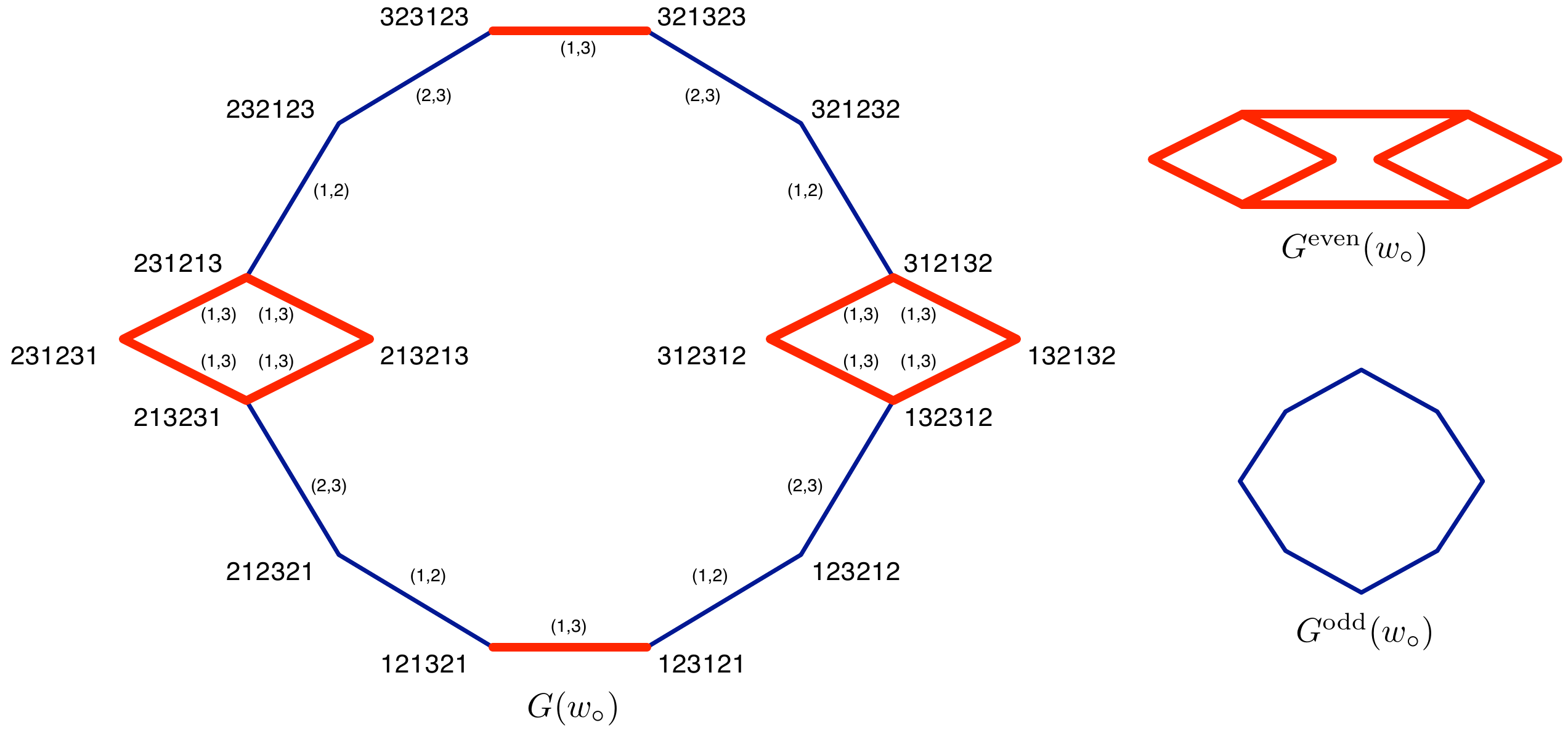}
\caption{The mega bipartite graph $G(\wo)$ of reduced expressions of the longest element in type $A_3$. The graphs obtained after contracting edges corresponding to any specified stabled set of braid moves are bipartite graphs.
Remark that for type $A$ the only non-trivial contracted graphs are $G^\text{even}(w)$ and $G^\text{odd}(w)$.}
\label{fig:Mega_bipartite_graph}
\end{center}
\end{figure}

\begin{corollary}\label{cor:evenedge}
Let $W$ be a finite Coxeter group and $w$ be an element in $W$. Then, any loop in the graph $G(w)$ contains an even number of edges labeled by pairs in the automorphism class of~$\{i,j\}$ (corresponding to a braid relation~$m_{ij}$) for any fixed pair $\{i,j\}$. In particular
\begin{enumerate}
\item Any loop in $G(w)$ contains an even number of edges.
\item Any loop in $G(w)$ contains an even number of edges corresponding to even braid relations.
\item Any loop in $G(w)$ contains an even number of edges corresponding to odd braid relations.
\end{enumerate}
\end{corollary}

\begin{remark}
For all finite types it is sufficient to consider only inner automorphism (conjugation) classes of pairs $\{i,j\}$. 
In type $A_n$, any pair $\{i,i+1\}$ is conjugated to any other pair~$\{j,j+1\}$ so that all odd braid relations are conjugated to each other. Similarly, any pair $\{i,j\}$ where~$j-i>1$ is conjugated to any other similar pair so all even braid relations are conjugated to each other. Hence, the only possible stabled sets $Z$ for $A_n$ are $\emptyset$, $Z^\text{odd}$, $Z^\text{even}$ and $Z^\text{odd}\cup Z^\text{even}$. 
In type~$B_n$ we have more possibilities. The pair $\{0,1\}$ with $m_{01}=4$ is an even braid relation in a single conjugacy class. For $m_{ij}=2$ we have two conjugacy classes: $\{ \{0,i\} : 1<i<n \}$ and $\{\{i,j\} : 0<i<j<n \text{ and } j-i>1\}$. Finally, all odd braid relations $m_{ij}$ are conjugated to each other. A stabled set $Z$ in type $B_n$ is any union of these classes. A study of all other irreducible cases (except $F_4$) shows that the conjugacy classes of pairs is the same as the automorphism classes.
In type $F_4$ the pair $\{1,2\}$ is not conjugated to $\{3,4\}$ but the first pair can be sent to the second via an outer automorphism.
An exhaustive computer search shows that all graphs $G^Z(w)$  are bipartite when $Z$ is stabled by conjugation instead of automorphism.
\end{remark}

In order to prove Theorem~\ref{thm:mega_bipartite}, it is convenient to represent reduced words of $w\in W$ using paths in the Coxeter arrangement associated to $W$. This will lead us to a new notion of a \emph{sign function} on the set of reduced expressions of $w$. One particular case which will be important from the subword complex perspective is the sign function for reduced expressions of the longest element.

\begin{remark}
Our proof of Theorem~\ref{thm:mega_bipartite} is purely topological and relies on the geometry of $W$ seen as a reflection group in a Euclidean space. It would be interesting to know whether this result extends to other Coxeter groups (not only finite ones).
\end{remark}

\subsection{The sign function on reduced expressions of $\wo$}
\label{sec:sign_function}

\begin{definition}\label{def:sign_function}
The \defn{sign function} on reduced expressions of~$\wo$ is a map 
\[
\begin{array}{cccc}
 \sign: & \{\text{reduced expressions of }\wo \}  & \rightarrow  & \{1,-1\}
\end{array}
\]
such that if $w,w'$ are two reduced expressions of $\wo$ connected by a braid move $m_{ij}$, then 
\begin{equation}\label{eq:sign_function1}
\sign(w') = (-1)^{m_{ij}-1} \cdot \sign(w).
\end{equation}
\end{definition}

Since the graph of reduced expressions of~$\wo$ is connected, this function is unique up to global multiplication by $-1$. Two reduced expressions connected by an odd braid move have the same sign, while two connected by an even braid move have opposite signs. A priori it is not clear whether the sign function exists and is well defined, but we will see below that it is a particular case of a more general family of sign functions on reduced expressions of an element~$w$ in $W$. In the case of type~$A_n$, this sign function already appeared in connection with scattering amplitudes in physics and the positive Grassmannian~\cite{postnikov_sign_physics_2012}. More precisely, the authors sign function arising from the case of the Grassmannian $G(2,n)$ coincides with the sign function of reduced expressions of~$\wo$ in type $A_n$.    
Figure~\ref{fig:sign_functionA3} illustrates the sign function on reduced expressions of~$\wo$ in type $A_3$.

\begin{figure}[htbp]
\begin{center}
\includegraphics[width=0.8\textwidth]{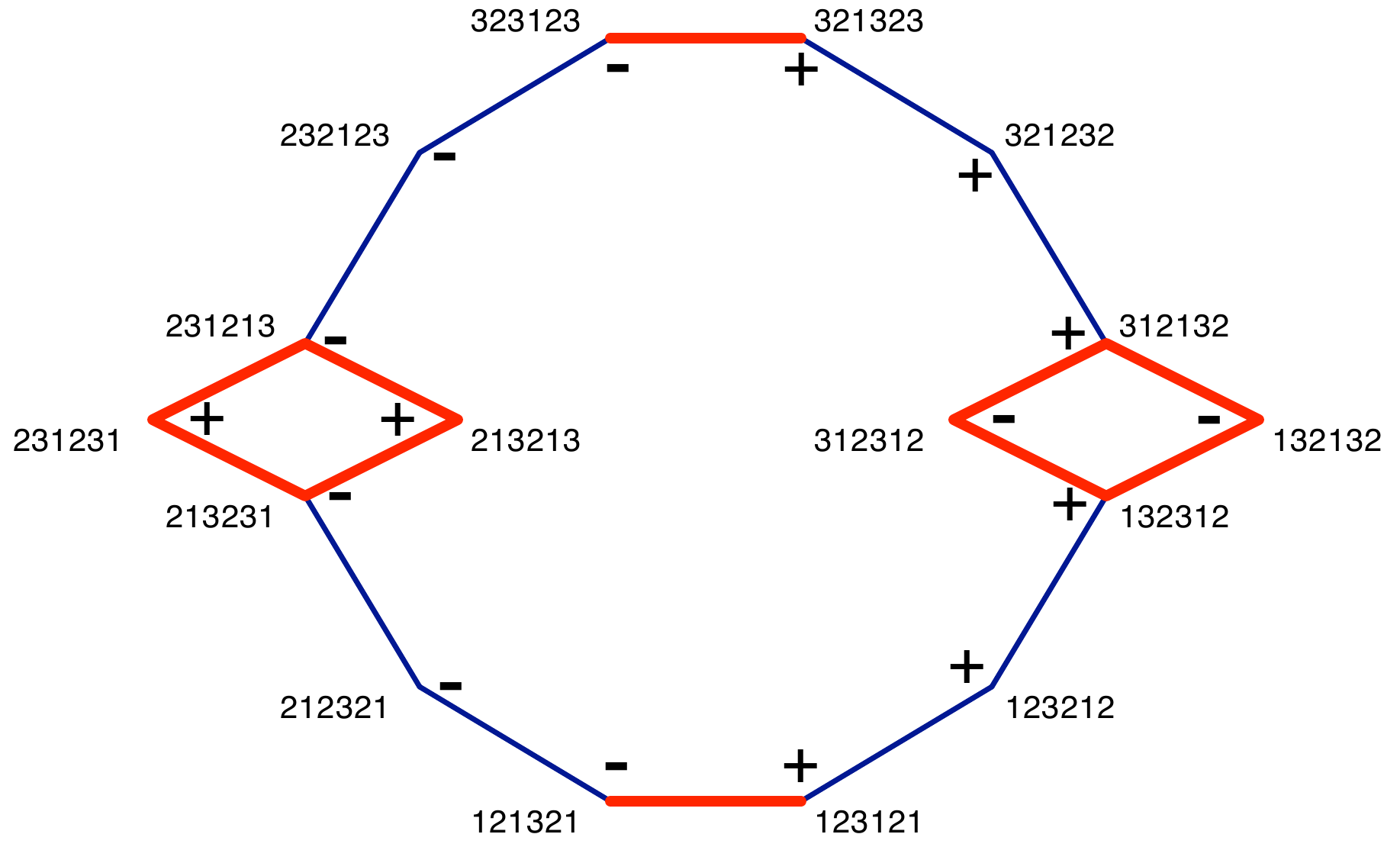}
\caption{The sign function on reduced expressions of~$\wo$ in type $A_3$. The red (thick) edges represent even braid relations (commutations in this case), and the blue (thin) edges represent odd braid relations.}
\label{fig:sign_functionA3}
\end{center}
\end{figure}

The following alternative description will be useful in Section~\ref{sec:reformulation}.
\begin{lemma}[Alternative description of the sign function]\label{lem:alternative_sign}
The sign function on reduced expressions of~$\wo$ is the unique map, up to multiplication by $-1$, such that if $w=w_1\dots w_N$ and $w'=w'_1\dots w'_N$ are two reduced expressions of $\wo$ connected by a flip, that is $w\setminus w_i=w'\setminus w'_j$, then 
\begin{equation}\label{eq:sign_function2}
\sign(w') = (-1)^{i-j} \cdot \sign(w).
\end{equation}
\end{lemma}

\begin{proof}
If $w$ and $w'$ are connected by a flip corresponding to a braid move $m_{ij}$, then equation~\eqref{eq:sign_function2} is clearly transformed into equation~\eqref{eq:sign_function1}. On the other hand, any flip between $w$ and $w'$ can be obtained by a sequence of flips associated to braid moves.  Equation~\eqref{eq:sign_function2} is obtained by applying equation~\eqref{eq:sign_function1} several times along the sequence.
\end{proof}

\begin{remark}[Sign function of type~$A_n$]
The sign function of type~$A_n$ can be easily described in terms of their inversions. Let $w=w_1\dots w_N$ be a reduced expression of~$\wo$ and $t_1, \dots , t_N$ be the corresponding inversions given by 
$t_k = w_1 \dots w_k \dots w_1.$
Each $t_k$ is a transposition of the form~$(i_k,j_k)$. If we replace each~$t_k$ by $\min\{i_k,j_k\}$, we obtain a multi-permutation of
\[
\underbrace{111\dots 1}_{n \text{ times }} \ \underbrace{22\dots 2}_{n-1 \text{ times }}  \   \dots \ n-1\ n-1\ n.
\]
The $\sign(w)$ is the sign of this multi-permutation. Recall that the sign of a multi-permutation $p=p_1\dots p_N$ is the equal to $-1$ to the number of inversions in $p$:
\[\sign(p) = (-1)^{ |\{ (i,j) :\ 1\leq i < j \leq N \text{ and } p_i > p_j \}| }. \]
For example, for the reduced expression $w=232123$ of $\wo$ of type $A_3$ the corresponding inversions and multi-permutation are given by 
\[
\begin{array}{rcccccc}
 w= & 2  & 3  & 2  & 1  & 2  & 3  \\
 \text{inversions}= & (2,3)  & (2,4)  & (3,4)  & (1,4)  & (1,3)  & (1,2)  \\
 p= &  2 & 2  & 3  & 1  & 1  & 1  \\
\end{array}
\]
This multi-permutation has 9 inversions. Therefore, $\sign(w)=\sign(p) = (-1)^9 = -1$.

According to personal communication with Alexander Postnikov, the sign of $w$ can be alternatively obtained as the product of the sign of the permutation of inversions of $w$ with $-1$ to the number of higher inversions of $w$. 
\end{remark}

\subsection{Coxeter complex; restriction and localization.}

We recall the standard construction of the Coxeter complex associated to a Coxeter group $W$. See \cite{humphreys_reflection_1992} for more details and proofs.
Let $\Phi\subset \R^n$ be a root system associated to a Coxeter group $(W,S)$, and let $\A$ be the hyperplane arrangement of all reflections induced by $\Phi$.
For each hyperplane $H\in\A$ there is a unique positive root $\alpha_H\in \Phi^+$.
We let $H^+=\{v\in\R^n : \langle v,\alpha_H\rangle>0 \}$ where $\langle -,-\rangle$ is the canonical scalar product on $\R^n$ induced by $\Phi$.
Similarly, let $H^-=-H^+$. The triples $H^-,H,H^+$ decompose $\R^n$ into two half spaces and a subspace on codimension 1.
The \emph{Coxeter complex} of $W$ is a cell decomposition of $\R^n$ obtained by considering all possible non-empty intersections 
  $\bigcap_{H\in\A} H^{\epsilon(H)}$
  where $\epsilon(H)$ is either $+, -$ or empty.
  The \emph{fundamental chamber} is the $n$-dimensional cell we obtain by choosing $\epsilon(H)=+$ for all $H\in A$. The chambers of the complex (the $n$-dimensional cells) are in natural bijection
  with the elements of $W$. More generally, as in Section~1.5~of~\cite{humphreys_reflection_1992}, every cell of the Coxeter complex is well labelled by the elements of the left cosets in the union
     $$\bigcup_{I\subseteq S} W \big/ W_I$$
 where $W_I$ is the parabolic subgroup of $W$ generated by the $s_i$ for $i\in I$. The labeling is compatible with the reflection action of $W$ on the Coxeter arrangement where the identity is identified with the fundamental chamber. The size of $I\subseteq S$ is the codimension of the cell it labels. See the examples in Figure~\ref{fig:TitsA2} and Figure~\ref{fig:TitsA2xA1}.
A reduced expression of~$w$ corresponds to a path from the fundamental chamber to the chamber $w$, crossing only through codimension~$\le 1$ cells, with a minimal number of codimension~1 cells.
The example in Figure~\ref{fig:TitsA2} corresponds to the Coxeter group $W=A_2$. The two paths from the identity to $\wo$ correspond to the two reduced expressions $\wo=s_1s_2s_1$ and $\wo=s_2s_1s_2$.
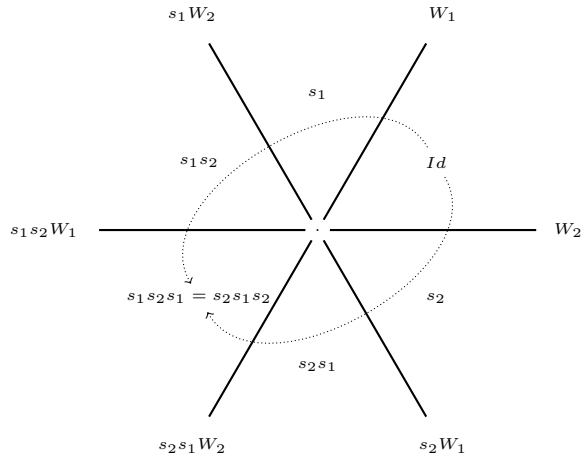
\begin{figure}[htbp]
\begin{center}
 \[
 	{\begin{tikzpicture}[scale=.3,font=\tiny]
		\node (0) at (0:0) {.};
		\node (1) at (0:10) {};
		\node (2) at (60:10){};
		\node (3) at (120:10){};
		\node (4) at (180:10) {};
		\node (5) at (240:10) {};
		\node (6) at (300:10) {};
		
		\draw[thick] (0)--(1);
		\draw[thick] (0)--(2);
		\draw[thick] (0)--(3);
		\draw[thick] (0)--(4);
		\draw[thick] (0)--(5);
		\draw[thick] (0)--(6);
		
		\node (7) at (30:6) {$\small Id$};
		\node  (8) at (90:6) {$\small s_1$};
		\node (9) at (150:6) {$\small s_1s_2$};
		\node (10) at (210:6) {$\small s_1s_2s_1=s_2s_1s_2$};
		\node (11) at (270:6) {$\small s_2s_1$};
		\node (12) at (330:6) {$\small s_2$};
		\node (13) at (0:11) {$\small W_{2}$};
		\node (14) at (60:11) {$\small W_{1}$};
		\node (15) at (120:11) {$\small s_1W_{2}$};
		\node (16) at (180:12) {$\small s_1s_2W_{1}$};
		\node (17) at (240:11) {$\small s_2s_1W_{2}$};
		\node (18) at (300:11) {$\small s_2W_{1}$};
		
		\draw [densely dotted,->] (7) to [out=120,in=120] (10);
		\draw [densely dotted,->] (7) to [out=-60,in=-60] (10) ;
	\end{tikzpicture}}
\]
\caption{Coxeter complex of type $A_2$. The chambers are labeled by elements of $W$, the walls are labeled by elements of the left cosets $W/W_1$ or $W/W_2$. The origin is labeled by the unique coset of $W/W_{12}$. A minimal path from the identity to $w\in W$ correspond to reduced expressions for $w$.}
\label{fig:TitsA2}
\end{center}
\end{figure}

Consider two distinct hyperplanes $H_1,H_2$ in $\A$ and the codimension 2 intersection $X=H_1\cap H_2$. The space $X$ is a union of cells of the original Coxeter complex of $W$.
The cells of maximal dimension in $X$ are indexed by left cosets $wW_I$ for some subsets $I=\{i,j\}$. The cell decomposition of $X$
is a Coxeter arrangement of rank $n-2$. The reflections within $X$ can always be seen as automorphisms of $W$ restricted to $X$.
Indeed any reflection of the system in $X$ is a reflection of $\R^n$ that also preserve the Coxeter system of $W$, this is an automorphism of $W$.
Example~\ref{exam:outer} allows us to visualize this fact. In particular, any two maximal cells $wW_{ij}$ and  $w'W_{i'j'}$ of $X$ are related
via an automorphism of $W$. This shows that $m_{ij}=m_{i'j'}$.  

\begin{definition} \label{def:X} 
For two distinct hyperplanes $H_1,H_2$ in $\A$ such that $X=H_1\cap H_2$ is of codimension~2, the \defn{restriction of the Coxeter complex}
to $X$ is the union of cells of the original Coxeter complex of $W$.
\end{definition}

\begin{definition} For $X$ as in Definition~\ref{def:X}, we define the \defn{localized Coxeter arrangement} $\A_X$ as the hyperplane arrangement in the (2-dimensional) quotient space $\R^n/X$
\[
\A_X=\big\{ H/X : H\in \A \text{ and } X\subseteq H \big\} \subseteq {\mathbb R}^n/X\, .
\]
The arrangement $\A_X$ is a Coxeter arrangement of lines in a two dimensional plane. 
Pick any cell of maximal dimension in $X$, it is indexed by a left coset $wW_I$ where $I=\{i,j\}$. The number of lines in $\A_X$ is equal to $m_{ij}$.
In view of the discussion before Definition~\ref{def:X}, this is independent of the choice of maximal cell in $X$ that we pick.
\end{definition}

\begin{example}\label{exam:outer}
In Figure~\ref{fig:TitsA2xA1}, there are four possible codimension~2 spaces $X$. Denote by $\bf 0$ the origin. The line $s_2s_1W_{24}, {\bf 0}, W_{14}$ is one possible $X$ given with its cell decomposition.
The reflection of $X$ that reflect $s_2s_1W_{24}$ into $W_{14}$ can be viewed in $\R^3$ as the reflection through the plane orthogonal to $X$. This is not a reflection of $W$ but it preserves the structure
of the Coxeter arrangement of~$W$, hence it is an outer automorphism of $W$.  This automorphism sends the pair $\{1,4\}$ to the pair $\{2,4\}$ which are conjugated.
If instead we take $X$ to be $s_4W_{12}, {\bf 0}, W_{12}$, then this time the reflection of $X$ is the reflection $s_4\in W$.
It is an inner automorphism. 
\end{example}

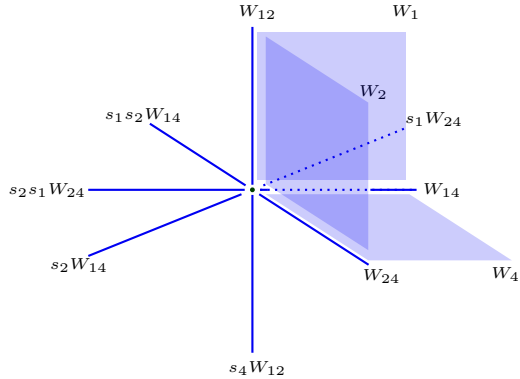
\begin{figure}[htbp]
\begin{center}
$$
\begin{tikzpicture}%
	[x={(1cm, .0cm)},
	y={(.0cm, 1 cm)},
	z={(-.4cm, -0.866cm)},
	scale=0.0675,
	font=\tiny,
	back/.style={loosely dotted, thin},
	edge/.style={color=blue!95!black, thick},
	facet/.style={fill=blue!95!black,fill opacity=0.200000},
	vertex/.style={inner sep=0pt,circle,draw=green!25!black,fill=green!75!black,thick,anchor=base}]
%
%
\node[vertex] at (0,0,0)     {};
\node at (37,0,0) {$\small W_{14}$};
\node at (30,35,0) {$\small W_{1}$};
\node at (1,35,0) {$\small W_{12}$};
\node at (1,-35,0) {$\small s_4W_{12}$};
\node at (33,0,19.5) {$\small W_{24}$};
\node at (57,0,19) {$\small W_{4}$};
\node at (30,33,16) {$\small W_{2}$};
\node at (29,0,-16) {$\small s_1W_{24}$};
\node at (-27,0,17) {$\small s_2W_{14}$};
\node at (-28,0,-17) {$\small s_1s_2W_{14}$};
\node at (-40,0,0) {$\small s_2s_1W_{24}$};

\draw[edge,dotted] (1.5,0,0) -- (32,0,0);
\draw[edge] (1.5,0,0) -- (3,0,0);
\draw[edge] (23,0,0) -- (32,0,0);
\draw[edge] (-1.5,0,0) -- (-32,0,0);
\draw[edge] (0,1,0) -- (0,32,0);
\draw[edge] (0,-1,0) -- (0,-32,0);
\draw[edge] (1.732,0,1) -- (29.4448,0,17);
\draw[edge] (-1.732,0,-1) -- (-25.9808,0,-15);
\draw[edge] (-1.732,0,1) -- (-25.9808,0,15);
\draw[edge,dotted] (1.732,0,-1) -- (24.2487,0,-14);

\fill[facet] (1,2,0) -- (30,2,0) -- (30,31,0) -- (1,31,0) -- (1,2,0) -- cycle {};
\fill[facet] (3,0,1) -- (31,0,1) -- (56.9808,0,16) -- (28.9808,0,16) -- (3,0,1) -- cycle {};
\fill[facet] (3,2,1) -- (3,31,1) -- (28.9808,31,16) -- (28.9808,2,16) -- (3,2,1) -- cycle {};

\end{tikzpicture}
$$
\caption{Coxeter complex of type $A_2\times A_1=W_{1,2,4}\subset A_4$. There are 12 chambers, 18 planar cells, 8 codimension~2 cells and the origin. We have labelled the planar cells around the fundamental chamber and all 8 codimension 2 cells.}
\label{fig:TitsA2xA1}
\end{center}
\end{figure}

\subsection{General sign functions and proof of Theorem~\ref{thm:mega_bipartite}}

Given a stabled set $Z$ of specified braid relations, we are interested in the localizations $\A_X$ when $X$  has a maximal cell $wW_{ij}$ for some $\{i,j\}\in Z$. 
Let
\[
Z(W)=\left\{ X: {X=H_1\cap H_2 \text{ for } H_1,H_2\in \A \text{  of codimension~2}
           \atop  \text{ $X$  has a maximal cell $wW_{ij}$ for some } \{i,j\}\in Z} \right\}\,.
\]
Remark that $Z(W)$ is a set with no multiplicity; if a codimension 2 space $X$ is generated in two different ways, then we count it only once in $Z(W)$. 
The stability of $Z$ guaranties that the choices of $H_1$ and $H_2$ for $X$ and the choices of maximal cell $wW_{ij}$ in $X$ do not matter. 

We want to define a sign function on the vertices of $G(w)$ depending on the set $Z$. The sign function introduced in Section~\ref{sec:sign_function} corresponds to the case where $Z$ is the set of even braid moves. 
Fix one reduced expression $r_0=s_{i_2}s_{i_2}\cdots s_{i_\ell}$ for $w$. Any reduced expression
$r$ for $w$ can be encoded with a path $P_r$ from the identity chamber to the chamber corresponding to $w$. 
For any $X\in Z(W)$, The paths $P_r$ and $P_{r_0}$ induce two paths $\overline{P_r}^X$ and  $\overline{P_{r_0}}^X$  in the quotient plane ${\mathbb R}^n/X$. 
Only two situations may happen: the two paths form a closed loop around $(0,0)$ or not. This is illustrated in the figure below for the case where the chamber of $w$ becomes the chamber of the longest element in the quotient plane.

 $${\begin{tikzpicture}[scale=.2,font=\tiny]
 \node (1) at (10,0) {};
 \node (2) at (7.2,7.2){};
 \node (3) at (0,10){};
 \node (4) at (-10,0) {};
 \node (5) at (-7.2,-7.2) {};
 \node (6) at (0,-10) {};
\draw[thick] (1)--(4);
\draw[thick] (2)--(5);
\draw[thick] (3)--(6);
\node (7) at (4.7,1.9) {$\small Id$};
\node (10) at (-4.7,-1.9) {$\small \wo$};
\node at (5,-5) {$\scriptstyle \bullet^{ \bullet{^\bullet}}$};
\node at (-5,5) {$\scriptstyle \bullet^{ \bullet{^\bullet}}$};
 \draw [densely dotted,->] (7) to [out=120,in=120] (10) ;
 \draw [densely dotted,->] (7) to [out=-60,in=-60] (10) ;
\end{tikzpicture}} 
\qquad\raise60pt\hbox{ or }\qquad
 {\begin{tikzpicture}[scale=.2,font=\tiny]
 \node (1) at (10,0) {};
 \node (2) at (7.2,7.2){};
 \node (3) at (0,10){};
 \node (4) at (-10,0) {};
 \node (5) at (-7.2,-7.2) {};
 \node (6) at (0,-10) {};
\draw[thick] (1)--(4);
\draw[thick] (2)--(5);
\draw[thick] (3)--(6);
\node at (5,-5) {$\scriptstyle \bullet^{ \bullet{^\bullet}}$};
\node at (-5,5) {$\scriptstyle \bullet^{ \bullet{^\bullet}}$};
\node (7) at (4.7,1.9) {$\small Id$};
\node (10) at (-4.7,-1.9) {$\small \wo$};
 \draw [densely dotted,->] (7) to [out=120,in=80] (10) ;
 \draw [densely dotted,->] (7) to [out=90,in=110] (10) ;
\end{tikzpicture}} $$

\noindent
In the first case we let $z(\overline{P_r}^X,\overline{P_{r_0}}^X)=1$ and in the second case we let $z(\overline{P_r}^X,\overline{P_{r_0}}^X)=0$.

\begin{definition}
The~\defn{$Z$-sign function} on reduced expressions of $w$ as the map
\[
	\sign^Z(r) = \prod_{X\in Z(W)} (-1)^{z(\overline{P_r}^X,\overline{P_{r_0}}^X)}.
\]
This function is well defined and depends only on the choice of $r_0$. 
\end{definition}

\begin{lemma} If $r$ and $r'$ differ by a single braid relation $m_{ij}$ then
   $$\sign^Z(r) = \left\{ \begin{array}{rl} -\sign^Z(r') &\text{if $(i,j)\in Z$,} \cr \cr \sign^Z(r') &\text{otherwise.} \end{array} \right. $$
\end{lemma}

\begin{proof}
Assume the two reduced expressions $r$ and $r'$ differ by a single braid relation $m_{ij}$. So, they factor as~$r=uxv$ and $r'=uyv$ where $x$ and $y$ is exactly one braid relation. The paths $P_r$ and $P_{r'}$ will be the same in the first $u$ steps and the same in the last $v$ steps. Let $C$ be the chamber we are in after the first $u$ steps of $P_r$ or $P_{r'}$. Let $H_1, H_2\in\A$ be the unique hyperplanes containing the codimension~1 cell  that we cross in the first step of $x$ and $y$ respectively. We have distinct $H_1$ and $H_2$ since $x$ and $y$ start with distinct generators.
This defines a unique $X=H_1\cap H_2\in Z(W)$. The closure $\overline{C}$ of $C$ determines a maximal cell $wW_{ij}$  in $X$. The pair $\{i,j\}$ corresponds to the braid relation $m_{ij}$ between $r$ and $r'$. 
Since $x$ and $y$ is a full braid relation, the paths
$\overline{P_r}^X$ and  $\overline{P_{r'}}^X$ describe a loop around the origin. For all the other $X'$ of codimension~2 in the Coxeter complex, the two paths 
$\overline{P_r}^{X'}$ and 
$\overline{P_{r'}}^{X'}$ will remain on the same side. 

Now, if $X\in Z(W)$, the paths $\overline{P_r}^{X'}$ and $\overline{P_{r'}}^{X'}$ describe a loop for $X'=X\in Z(W)$ and not for all $X\ne X'\in Z(W)$. 
If $X\not\in Z(W)$, then the paths $\overline{P_r}^{X'}$ and $\overline{P_{r'}}^{X'}$ do not describe a loop for any $X'\in Z(W)$. 
This shows that $\sign^Z(r) = -\sign^Z(r') $ exactly in the case $X\in Z(W)$; exactly when $\{i,j\}\in Z$.
\end{proof}

 In the case where $Z$ consists of the pairs $\{i,j\}$ for which $m_{ij}$ is even, we get back the definition of the sign function in Definition~\ref{def:sign_function}. In particular, this shows that the sign function exists and is well defined.
We are now ready to prove Theorem~\ref{thm:mega_bipartite}.

\begin{proof}[Proof of Theorem~\ref{thm:mega_bipartite}]
Let $G(w)$ be the graph of reduced expressions of $w$ in $W$ and $G^Z(w)$ be the graph obtained by contracting the edges of $G(w)$ labeled by braid moves $\{i,j\}\notin Z$. Two vertices of $G(w)$ connected by a braid move not in $Z$ have the same sign, while two connected by a braid move in $Z$ have opposite signs. Therefore, if we contract all edges of $G(w)$ corresponding to braid moves not in $Z$ we obtain a graph where every pair of adjacent vertices are labeled with different signs. Thus, the resulting graph $G^Z(w)$ is a bipartite graph.
\end{proof}

\begin{remark} Theorem~\ref{thm:mega_bipartite} is not true if we remove the word {\sl stabled} from its statement. In particular, Corollary~\ref{cor:evenedge} does not hold if we consider only pairs $\{i,j\}$ and not automorphism classes of them.
The following example
	was pointed out to us by Darij Grinberg. As in Figure~\ref{fig:TitsA2xA1}, consider the reduced expression $s_1s_2s_1s_4$ and say you consider the pair $\{1,4\}$ but {\bf not} its conjugate~$\{2,4\}$.
	The sequence of braid moves
	 $$s_1s_2s_1s_4\  {\buildrel  {\tiny 14} \over \longrightarrow} \ s_1s_2s_4s_1\  {\buildrel  {\tiny 24} \over \longrightarrow} \ s_1s_4s_2s_1\  {\buildrel  {\tiny 14} \over \longrightarrow}\ s_4s_1s_2s_1\  {\buildrel  {\tiny 12} \over \longrightarrow} \ s_4s_2s_1s_2 $$ 
	 $$ {\buildrel  {\tiny 24} \over \longrightarrow}\  s_2s_4s_1s_2\  {\buildrel  {\tiny 14} \over \longrightarrow}\  s_2s_1s_4s_2\ {\buildrel  {\tiny 24} \over \longrightarrow} \ s_2s_1s_2 s_4\  {\buildrel  {\tiny 12} \over \longrightarrow}\ s_1s_2s_1s_4$$
	 is a loop in the graph of $G(w)$ that contains three braid relations $\{1,4\}$, not an even number.  But if we consider both $\{1,4\}$ and $\{2,4\}$, then we get six braid relations of that type, an even number.
\end{remark}
%

%
%
\section{Coxeter signature matrices}

Let $Q=(q_1,\dots, q_r)$ be a word in $S$ containing at least one reduced expression of $\wo$, and let~$N=\ell(\wo)$ be the length of the longest element in $W$.
A Coxeter signature matrix is a concept that plays a fundamental role to obtain fan realizations of subword complexes. Indeed, we will see in Theorem~\ref{thm:fan_reformulation} that finding a complete simplicial fan realization of~$\subwordComplex[Q]$ is almost equivalent to finding a Coxeter signature matrix for the pair~$(Q,\wo)$.

\begin{definition}[Coxeter signature matrix]
A matrix $M \in \R^{N\times r}$ is a \defn{signature matrix} of type~$W$ for the pair $(Q,\wo)$ if for every reduced expression $w\subset Q$ of $\wo$, 
\begin{equation}\label{cond:signature}
\sign(w) \cdot \det(w) > 0, 
\end{equation}
where $\det(w)$ is the determinant of the matrix $M$ restricted to the columns corresponding to $w$, and $\sign(w)$ is the sign function of $w$ according to Definition~\ref{def:sign_function}.
\end{definition}

\begin{proposition}\label{prop:part1}
Let $W$ be a Coxeter group of type $A_n$ with $n\leq 3$, let $c$ be a Coxeter element and $Q=c^m$. The counting matrix $\dualCountingMatrix$ is a signature matrix for the pair~$(Q,\wo)$.
\end{proposition}

\begin{proof}
This result is proven by inspection in each case. Explicit formulas for the counting matrices can be found in Appendix~\ref{app:matrices}.\\
\emph{Type $A_1$:} there is exactly one reduced expression of $\wo$ in type $A_1$, whose sign function is positive. Since the determinant of any reduced expression of $\wo$ in $Q$ is equal to 1, the result follows.\\
\emph{Type $A_2$:} there are exactly two reduced expressions of $\wo$ in type $A_2$: $s_1s_2s_1$ and $s_2s_1s_2$. The sign function for both expressions is positive. It is also straight forward to check that the determinant of any reduced expression of $\wo$ in $Q$ is also positive. Therefore the result follows.\\
\emph{Type $A_3$:}
let $w\subset Q$ be a reduced expression of $\wo$. In type $A_3$, there are 16 different reduced expressions of $\wo$. For each of them, the submatrices corresponding to $w$ always have a fixed form with 6 parameters $m\geq a\geq b\geq c\geq d\geq e\geq f$ corresponding to the copy of $c$ in which the columns are taken in the power $c^m$ (counted from right to left). The determinants only depend on these parameters. Table~\ref{table:determinantsA3} presents the formulas for the 16 determinants for~$c=(s_2,s_1,s_3)$ and Table~\ref{table:determinantsA3_2} for~$c=(s_1,s_2,s_3)$. By inspection, the determinants satisfy the necessary condition~\eqref{cond:signature} of a Coxeter signature matrix, compare with Figure~\ref{fig:sign_functionA3}. The other Coxeter elements are obtained using the symmetry of the counting matrix.

\begin{table}[!htbp]
\small
\begin{center}
\begin{tabular}{c|r|c}
Expression & Determinant & Sign \T\B\\\hline
123121 & $\frac{1}{2}(a-d)(a-f)(b-e)(d-f)$ & $+$\T\B\\
121321 & $-\frac{1}{2}(a-c)(a-f)(b-e)(c-f)$ & $-$\T\B\\\hline
231231 & $\frac{1}{2}(a-d)(b-e)(c-f)(2(a+d)-b-c-e-f+2)$ & $+$\T\B\\
231213 & $-\frac{1}{2}(a-d)(b-f)(c-e)(2(a+d)-b-c-e-f+2)$ & $-$\T\B\\
213231 & $-\frac{1}{2}(a-d)(c-e)(b-f)(2(a+d)-b-c-e-f+2)$ & $-$\T\B\\
213213 & $\frac{1}{2}(a-d)(c-f)(b-e)(2(a+d)-b-c-e-f+2)$ & $+$\T\B\\\hline
123212 & $(a-e)(b-d)(b-f)(d-f)$ & $+$ \T\B\\\hline
212321 & $-(a-c)(a-e)(b-f)(c-e)$ & $-$ \T\B\\\hline
321323 & $\frac{1}{2}(a-d)(a-f)(b-e)(d-f)$ & $+$ \T\B\\
323123 & $-\frac{1}{2}(a-c)(a-f)(b-e)(c-f)$ & $-$ \T\B\\\hline
132132 & $-\frac{1}{2}(a-d)(b-e)(c-f)(a+b+d+e-2(c+f)-2)$ & $-$ \T\B\\
132312 & $\frac{1}{2}(a-e)(b-d)(c-f)(a+b+d+e-2(c+f)-2)$ & $+$ \T\B\\
312132 & $\frac{1}{2}(a-e)(b-d)(c-f)(a+b+d+e-2(c+f)-2)$ & $+$ \T\B\\
312312 & $-\frac{1}{2}(a-d)(b-e)(c-f)(a+b+d+e-2(c+f)-2)$ & $-$ \T\B\\\hline
232123 & $-(a-c)(a-e)(b-f)(c-e)$ & $-$ \T\B\\\hline
321232 & $(a-e)(b-d)(b-f)(d-f)$ & $+$\T\B
\end{tabular}
\end{center}
\caption{Fomulas for the determinants of submatrices corresponding to reduced expressions of $\wo$ in $Q=c^m$ for $c=213$.}
\label{table:determinantsA3}
\end{table}

\begin{table}[!htbp]
\small
\begin{center}
\begin{tabular}{c|r|c}
Expression & Determinant & Sign \T\B\\\hline
123121 & $\frac{1}{2}(a - d)(a - f)(b - e)(d - f)$ & $+$\T\B\\
121321 & $-\frac{1}{2}(a - c)(a - f)(b - e)(c - f)$ & $-$\T\B\\\hline
231231 & $\frac{1}{2}(2(a + d) - b - c - e - f)(a - d)(b - e)(c - f)$ & $+$\T\B\\
231213 & $-\frac{1}{2}(2(a + d) - b - c - e - f)(a - d)(b - f)(c - e)$ & $-$\T\B\\
213231 & $-\frac{1}{2}(2(a + d) - b - c - e - f)(a - d)(b - f)(c - e)$ & $-$\T\B\\
213213 & $\frac{1}{2}(2(a + d) - b - c - e - f)(a - d)(b - e)(c - f)$ & $+$\T\B\\\hline
123212 & $(a - e)(b - d)(b - f)(d - f)$ & $+$ \T\B\\\hline
212321 & $-(a - c)(a - e)(b - f)(c - e)$ & $-$ \T\B\\\hline
321323 & $\frac{1}{2}(a - d)(a - f)(b - e)(d - f)$ & $+$ \T\B\\
323123 & $-\frac{1}{2}(a - c)(a - f)(b - e)(c - f)$ & $-$ \T\B\\\hline

132132 & $-\frac{1}{2}(a + b + d + e - 2(c + f))(a - d)(b - e)(c - f)$ & $-$ \T\B\\
132312 & $\frac{1}{2}(a + b + d + e - 2(c + f))(a - e)(b - d)(c - f)$ & $+$ \T\B\\
312132 & $\frac{1}{2}(a + b + d + e - 2(c + f))(a - e)(b - d)(c - f)$ & $+$ \T\B\\
312312 & $-\frac{1}{2}(a + b + d + e - 2(c + f))(a - d)(b - e)(c - f)$ & $-$ \T\B\\\hline

232123 & $-(a - c)(a - e)(b - f)(c - e)$ & $-$ \T\B\\\hline
321232 & $(a - e)(b - d)(b - f)(d - f)$ & $+$\T\B
\end{tabular}
\end{center}
\caption{\label{tab:red_expr} Fomulas for the determinants of submatrices corresponding to reduced expressions of $\wo$ in $Q=c^m$ for $c=123$.}
\label{table:determinantsA3_2}
\end{table}

In order to illustrate how these determinants are computed, we present a specific example for~$w=123121$ and $c=(s_2,s_1,s_3)$. In this case, the matrix $\dualCountingMatrix$ restricted to the columns corresponding to $w$ has the form
\[
\tiny
\left(\begin{array}{rrrrrr}
-\frac{1}{2} \, {\left(a + 1\right)} a & {\left(b + 1\right)}^{2} & -\frac{1}{2} \, {\left(c + 1\right)} c & -\frac{1}{2} \, {\left(d + 1\right)} d & {\left(e + 1\right)}^{2} & -\frac{1}{2} \, {\left(f + 1\right)} f \\
-\frac{1}{2} \, {\left(a - 1\right)} a + 1 & {\left(b + 1\right)} b & -\frac{1}{2} \, {\left(c + 1\right)} c & -\frac{1}{2} \, {\left(d - 1\right)} d + 1 & {\left(e + 1\right)} e & -\frac{1}{2} \, {\left(f - 1\right)} f + 1 \\
-\frac{1}{2} \, {\left(a + 1\right)} a & {\left(b + 1\right)} b & -\frac{1}{2} \, {\left(c - 1\right)} c + 1 & -\frac{1}{2} \, {\left(d + 1\right)} d & {\left(e + 1\right)} e & -\frac{1}{2} \, {\left(f + 1\right)} f \\
\frac{1}{2} \, {\left(a + 1\right)} a & -{\left(b + 1\right)}^{2} + 1 & \frac{1}{2} \, {\left(c + 1\right)} c & \frac{1}{2} \, {\left(d + 1\right)} d & -{\left(e + 1\right)}^{2} + 1 & \frac{1}{2} \, {\left(f + 1\right)} f \\
\frac{1}{2} \, {\left(a - 1\right)} a & -{\left(b + 1\right)} b & \frac{1}{2} \, {\left(c + 1\right)} c & \frac{1}{2} \, {\left(d - 1\right)} d & -{\left(e + 1\right)} e & \frac{1}{2} \, {\left(f - 1\right)} f \\
\frac{1}{2} \, {\left(a + 1\right)} a & -{\left(b + 1\right)} b & \frac{1}{2} \, {\left(c - 1\right)} c & \frac{1}{2} \, {\left(d + 1\right)} d & -{\left(e + 1\right)} e & \frac{1}{2} \, {\left(f + 1\right)} f
\end{array}\right)
\]
where $m\geq a > b\geq c > d > e\geq f$ correspond to the copy of $c$ (counted from right to left) in which the letters of $w$ appear in $Q=c^m$. The determinant of this matrix was computed using the computer software Sage~\cite{sage}, and is shown in Table~\ref{table:determinantsA3}. It is remarkable that the determinant in all cases has such a simple factorization. 
\end{proof}

We remark that Proposition~\ref{prop:part1} does not hold for $n\geq 4$, and address the problem of finding general Coxeter signature matrices as a main direction of future research.

%

%
%
\section{Fan realizations}
\label{sec:reformulation}

The main goal of this section is to present a reformulation of the problem of finding fan realizations of subword complexes in terms of Coxeter signature matrices. This will be used to prove our main result about fan realizations of spherical subword complexes of type~$A_3$ in Section~\ref{sec:mainproof}.

As before, we consider a word $Q=(q_1,\dots, q_r)$ in $S$ containing at least one reduced expression of $\wo$, and denote by $N=\ell(\wo)$ the length of the longest element in $W$. We also consider a full rank matrix~$M\in \R^{(r-N)\times r}$ and a Gale dual matrix~$M^G\in \R^{N\times r}$, as well as the associated fan~$\fan$ from Definition~\ref{def:fan}.

\begin{theorem}[{Ceballos \cite[Section~3.1 and Theorem~3.7]{ceballos_associahedra_2012}}]
\label{thm:fan_reformulation}
$\fan$ is a complete simplicial fan realization of the spherical subword complex~$\subwordComplex[Q]$ if and only if
\begin{enumerate}[1.]
\item[(S)] $M^G$ is a Coxeter signature matrix for the pair~$(Q,\wo)$ (Signature), and
\item[(I)] there is a facet of~$\subwordComplex[Q]$ for which the interior of its associated cone is not intersected by any other cone (Injectivity).
\end{enumerate}
\end{theorem}

The proof of this theorem follows directly from Lemmas~\ref{lem:complete_fan1} and~\ref{lem:complete_fan2} below. Lemma~\ref{lem:complete_fan1} is a common characterization of complete simplicial fans in the literature, see for example~\cite[Cor. 4.5.20]{de_loera_triangulations_2010}. Lemma~\ref{lem:complete_fan2} is restated from~\cite[Theorem~3.7]{ceballos_associahedra_2012} but is explicitly proven here for convenience of the reader.

\begin{lemma}\label{lem:complete_fan1}
$\fan$ is a complete simplicial fan if and only if the following conditions are satisfied:
\begin{enumerate}[1.]
\item[(B)] The vectors associated to a facet of~$\subwordComplex[Q]$ form a basis of~$\R^{r-N}$ (Basis).
\item[(F)] If $I$ and $J$ are two adjacent facets that differ by a flip, that is $I\setminus \{i\} = J \setminus \{j\}$, then the vectors associated to $i$ and $j$ lie in opposite sides of the hyperplane generated by the vectors associated to the intersection $I \cap J$ (Flip).
\item[(I)] There is a facet for which the interior of its associated cone is not intersected by any other cone (Injectivity).
\end{enumerate}

\end{lemma}

\begin{lemma}\label{lem:complete_fan2}
Conditions $(B)$ and $(F)$ of Lemma~\ref{lem:complete_fan1} are satisfied if and only if $M^G$ is a Coxeter signature matrix for the pair~$(Q,\wo)$.
\end{lemma}

\begin{proof}
By Gale duality, conditions $(B)$ and $(F)$ in Lemma~\ref{lem:complete_fan1} are satisfied for a matrix $M$ if and only if~$M^G$ satisfies the following two conditions:
\begin{enumerate}[1.]
\item The vectors associated to the complement of a facet of~$\subwordComplex[Q]$ form a basis of~$\R^N$.
\item If $I$ and $J$ are two adjacent facets that differ by a flip, that is $I\setminus \{i\} = J \setminus \{j\}$. Then the vectors associated to $i$ and $j$ lie in the same side of the hyperplane generated by the vectors associated to the complement of $I \cup J$.
\end{enumerate}

Condition 1 implies that for every reduced expression $w\subset Q$ of $\wo$ the determinant $\det(w)$ is different from zero. 
Moreover, using the alternative description of the sign function in Lemma~\ref{lem:alternative_sign},
if we set the sign and the determinant of $w_1\dots w_N \subset Q$ to be positive, then condition 2 implies that the sign of the determinant of $w$ is determined by
\[
\sign(w)\cdot \det(w)>0.
\]
Conversely, these inequalities imply both condition 1 and condition 2.  
\end{proof}

\begin{remark}
Although the injectivity condition $(I)$ in Theorem~\ref{thm:fan_reformulation} is a difficult property to prove in general, we suggest that the Signature condition $(S)$ is the most important part of the theorem. This is supported by our results in Section~\ref{sec:realization_space}, which show a continuous space of matrices satisfying the signature condition $(S)$ that automatically satisfy the injectivity condition $(I)$.
\end{remark}
%

%
%
\section{Proof of Theorem~\ref{thm:complete_fan}}
\label{sec:mainproof}

Let~$\subwordComplex[Q]$ be a spherical subword complex of type~$A_n$ with $n\leq 3$, and $\varphi$ be an embedding of~$Q$ into~$c^m$. This section contains the proof of Theorem~\ref{thm:complete_fan}, which asserts that the fan $\fanp{\primalMatrix}$ is complete. 
The fact that it realizes the subword complex~$\subwordComplex[Q]$ follows from the definition.
Recall that the matrix~$\primalMatrix$ is a Gale dual matrix of the matrix~$\dualMatrix$ in Definition~\ref{def:dual_matrix}. We prove this result in two steps. First, we prove that it is sufficient to consider the case where $Q=c^m$ and the embedding~$\varphi$ is the trivial embedding of $c^m$ into itself~(Section~\ref{sec:cm_enough}). The second step contains the proof of that explicit case (Section~\ref{sec:cm_part}).   

\subsection{Sufficient to prove the case $c^m$} \label{sec:cm_enough}

Our proof is based on the following lemma, which is left to the reader.

\begin{lemma}\label{lem:projection}
Let $I$ be a cone (of any dimension) in a complete simplicial fan $\mathcal F$. The projection of the link $\link(I,\mathcal F)$, to the orthogonal space of $I$, is a complete simplicial fan realizing~$\link(I,\mathcal F)$.
\end{lemma}

\begin{lemma}
Let $Q=c^m$ and~$\varphi$ be the trivial embedding of $c^m$ into itself. 
If the fan~$\fanp{\primalMatrix}$ is complete then the fan $\mathcal F_{Q',M_{\varphi'}}$ is complete for any embedding $\varphi'$ of a word $Q'$ into $c^m$.
\end{lemma}

\begin{proof}
The idea of the proof is to obtain the fan associated to $Q'$ as a projection of the fan associated to $Q$. This is done for appropriate choices of~$\primalMatrix$ and~$M_{\varphi'}$. Since different choices of Gale dual matrices only effect the fans by linear transformations, and particularly do not affect their completeness, the result will follow. Throughout the proof we denote by~$\widetilde r$ the length of the word~$Q=c^m$. We also assume that $Q'$ contains at least one reduced expression of~$\wo$, otherwise its associated fan would be empty and there would be nothing to prove.  

Let~$I\subset [\widetilde r]$ be the face of~$\subwordComplex[c^m]$ containing the positions in~$c^m$ which are not in $\varphi'(Q')$. Then, there is a natural isomorphism
\begin{equation}\label{eq:link}
\subwordComplex[{Q'}] \cong \link(I,\subwordComplex[c^m]) \cong \link(I, \fanp{\primalMatrix})
\end{equation}

Let~$I'\subset [\widetilde r]$ be a facet of~$\subwordComplex[c^m]$ containing~$I$. 
We can assume that the matrix~$\primalMatrix$ restricted to the columns with indices in~$I'$ is the identity matrix. If it is not, we could multiply its inverse with~$\primalMatrix$ from the left to obtain a new matrix with that desired property. The result would still be a Gale dual matrix~$\primalMatrix$ of~$\dualMatrix$. 
Since $I\subset I'$, the columns of~$\primalMatrix$ with indices in $I$ are certain canonical basis vectors $\{e_{a_1},\dots,e_{a_k}\}$, where $k=|I|$. Let $M'$ be the matrix obtained from~$\primalMatrix$ by removing the columns with indices in $I$ and the rows with indices in~$\{a_1,\dots,a_k\}$. The main point of the proof is to observe that $M'$ is a Gale dual matrix~$M_{\varphi'}$ of $D_{\varphi'}$. This is easily deduced from the fact that $D_{\varphi'}$ is obtained from~$\dualMatrix$ by removing the columns with indices in $I$, and that~$\primalMatrix$ has a zero entry in every position which is in a column in $I$ and in a row not in~$\{a_1,\dots,a_k\}$.
Taking~$M_{\varphi'}=M'$, we deduce that the fan~$\mathcal F_{Q',M_{\varphi'}}$ is the projection of $\link(I, \fanp{\primalMatrix})$ to the orthogonal space of the cone corresponding to $I$. By Lemma~\ref{lem:projection} and Equation~\eqref{eq:link}, we obtain that~$\mathcal F_{Q',M_{\varphi'}}$ is a complete fan realizing~$\subwordComplex[Q']$.
\end{proof}

\subsection{Proof for the case $c^m$} \label{sec:cm_part}
Let $Q=c^m$ and~$\varphi$ be the trivial embedding of $c^m$ into itself. In order to prove that the fan $\fanp{\primalMatrix}$ is a complete simplicial fan realization of~$\subwordComplex[Q]$ we follow the two steps in Theorem~\ref{thm:fan_reformulation}. 
The Signature condition $(S)$ is equivalent to the statement in Proposition~\ref{prop:part1}: since~$\varphi$ is the trivial embedding then~$\primalMatrix^G=\dualCountingMatrix$, which is a Coxeter signature matrix for the pair~$(Q,\wo)$ as desired. For the Injectivity part $(I)$, we need to prove that there is a cone whose interior is not intersected by any other cone. This is done by inspection in each case. Explicit formulas for the Gale dual matrices~$\primalMatrix=M_{c,m}=\dualCountingMatrix^G$ can be found in Appendix~\ref{ap:Gale_dual}.

\subsubsection{Type $A_1$} 
This is the trivial case. The fan has $m$ rays given by the $m-1$ negative basis basis vectors in $\R^{m-1}$ together with the vector with all entries equal to one. The maximal cones correspond to subsets of $m-1$ rays and the subword complex is isomorphic to the boundary of an $m-1$ dimensional simplex.

\subsubsection{Type $A_2$} 
In this case we fix a cone~$C^*$ corresponding to the negative orthant and check that its interior is not intersected by any other cone. Let~$C$ be a cone corresponding to a subword of~$c^m$ whose complement is a reduced expression of~$\wo$. There are three possibilities:

\medskip
\paragraph{\it The reduced expression uses 1 negative basis vector}
This case follows from the fact that~$\primalMatrix^G$ is a Coxeter signature matrix, which implies that any two adjacent cones lie in opposite sides of the hyperplane spanned by their interesection, see Lemma~\ref{lem:complete_fan2}. Alternatively, one can check by inspection that flipping $-e_i$ in the negative orthant gives a vector whose $i$th coordinate is positive. 

\medskip
\paragraph{\it The reduced expression uses 2 negative basis vectors}
In this case, the cone $C$ uses all negative basis vectors except for two, and two of the last three columns of $M_{12,m}$. Denote by $v_1,v_2\in \R^2$ the restrictions of these two columns to the coordinates corresponding to the negative basis vectors that are not used in $C$. Proving that $C$ does not intersect the negative orthant is equivalent to show that the cone spanned by $v_1$ and $v_2$ does not intersect the negative orthant in $\R^2$. Figure~\ref{fig:A2_2negativevectors} shows all three possible cases. 
In each case, we provide a vector $v\in \R^2$ with non-negative entries whose inner product with $v_1$ and $v_2$ is non-negative. The hyperplane (in this case a line) orthogonal to~$v$ separates the negative orthant and the cone spanned by $v_1$ and $v_2$.

\begin{figure}[htbp]
\centering
\begin{subfigure}{0.32\textwidth} \centering
\includegraphics[height=3.8cm]{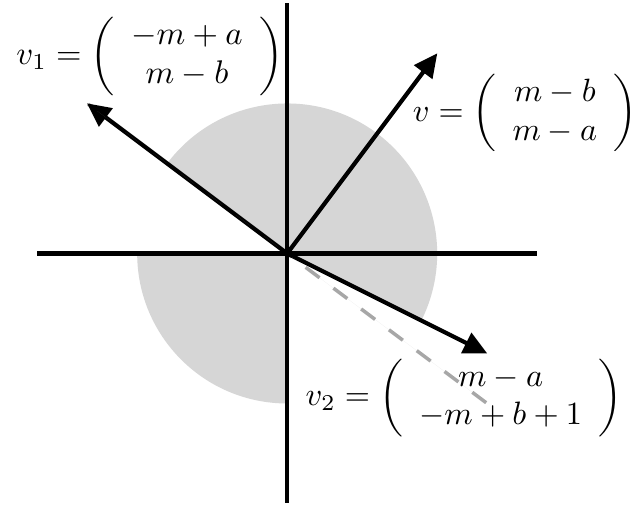}  
\caption{12-1 with embedding 2\12.}
\label{fig:A2_2negativevectors_a}
\end{subfigure} 
\begin{subfigure}{0.32\textwidth} \centering
\includegraphics[height=3.8cm]{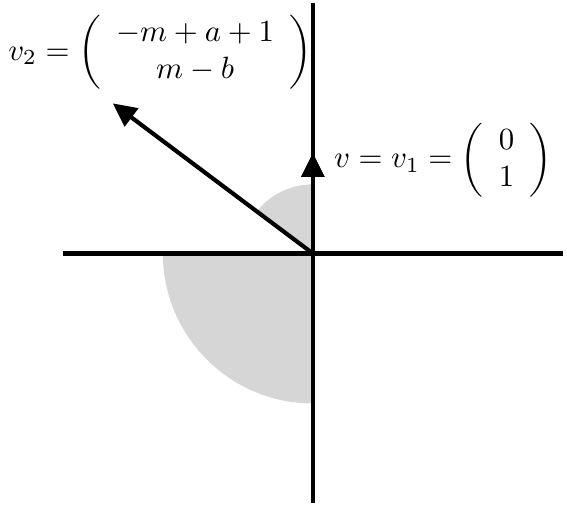}  
\caption{21-2 with embedding \212.}
\label{fig:A2_2negativevectors_b}
\end{subfigure} 
\begin{subfigure}{0.32\textwidth} \centering
\includegraphics[height=3.8cm]{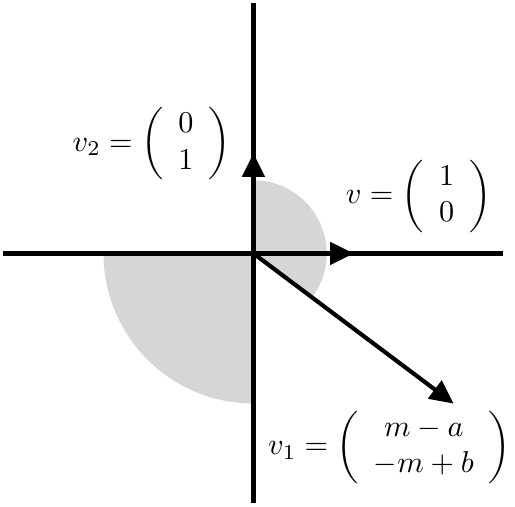}  
\caption{21-2 with embedding 21\2.}
\label{fig:A2_2negativevectors_c}
\end{subfigure} 
\caption{The three cases in type $A_2$ using 2 negative basis vectors.}
\label{fig:A2_2negativevectors}
\end{figure}

\noindent
Figure~\ref{fig:A2_2negativevectors_a} corresponds to the case where the reduced expression is~$121$. The parameters $a\leq b \leq m$ denote the copy of $c$ in which the negative basis vectors that are not used in $C$ are taken in the power $c^m$ (counted from left to right). The bold number in $2\12$ is the letter in the reduced expression that is in the last three letters of $c^m$, and the non-bold numbers correspond to the vectors $v_1$ and $v_2$. Similarly, Figures~\ref{fig:A2_2negativevectors_b} and~\ref{fig:A2_2negativevectors_c} illustrate the two possible cases for the reduced expression $212$.

\medskip
\paragraph{\it The reduced expression uses 3 negative basis vectors}

Let $v_1,v_2,v_3\in \R^3$ be the restriction of the last three columns of $M_{12,m}$ to the three negative basis vectors that are not used in the cone~$C$. As before, 
we prove that the cone spanned by $v_1,v_2,v_3$ does not intersect the negative orthant in~$\R^3$. There are two possible cases corresponding to the reduced expressions $121$ and $212$: 

\smallskip
\begin{minipage}{.4\textwidth}
\footnotesize
\[
\left(
\begin{array}{ccc}
 -m+a & 1  & m-a  \\
 m-b & 0  & -m+b+1  \\
 -m+c & 1  & m-c  
\end{array}
\right)
\]
\end{minipage}
\qquad
\begin{minipage}{.4\textwidth}
\footnotesize
\[
\left(
\begin{array}{ccc}
 m-a & 0  & -m+a+1  \\
 m+b & 1  & m-b  \\
 m-c & 0  & -m+c+1  
\end{array}
\right)
\]
\end{minipage}
\smallskip

\noindent
The columns are the given by the vectors~$v_1,v_2,v_3$. The parameters~$a\leq b \leq c \leq m$ denote the copy of~$c$ in which the negative basis vectors that are not used in~$C$ are taken in the power~$c^m$ (counted from left to right). In both cases, the hyperplane orthogonal to the vector~$\small v=(0,m-c,m-b)$ separates the negative orthant and the cone spanned by~$v_1,v_2,v_3$. This follows from the fact that~$v$ has non-negative entries and that the inner product of~$v$ with $v_1,v_2$ and $v_3$ is non negative. 

\subsubsection{Type $A_3$} 
This case is explicitly shown here for the bipartite Coxeter element $c=(s_2,s_1,s_3)$ using the matrix $M_{213,m}$ in Corollary~\ref{cor:fan_bipartite213}. The case $c=(s_1,s_2,s_3)$ is similar and the remaining cases are obtained by symmetry. 
As in the previous cases, we fix a cone~$C^*$ corresponding to the negative orthant and check that its interior is not intersected by any other cone~$C$.  In contrast to the type $A_2$ case, we work with the transpose matrix $M_{213,m}^t$ for convenience, and use row vectors instead of columns. We also consider parameters corresponding to the copy of~$c$ in which letters appear counted from right to left, so that the parameters coincide with the subindices of the matrices $B_{213,i}$.

There are six possible cases depending on the number~$k$ of negative basis vectors that are used in the reduced expression corresponding to~$C$. These are studied case by case below. We denote by~$v_1,v_2,\dots ,v_k\in \R^k$ the restrictions of the last six rows of $M_{213,m}^t$ to the~$k$ negative basis vectors that are not used in~$C$. As before, proving that $C$ does not intersect the negative orthant is equivalent to show that the cone spanned by $v_1,\dots ,v_k$ does not intersect the negative orthant in $\R^k$. This is done by providing a vector $v\in \R^k$ with non-negative entries whose inner product with $v_1,\dots,v_k$ is non-negative. The hyperplane orthogonal to $v$ separates the two cones in consideration. The vector $v$ is given in terms of some parameters~$m-2\geq a\geq b \geq \dots$, which denote the copy of~$c$ in which the negative basis vectors that are not used in~$C$ are taken in the power~$c^{m-2}$ (counted from right to left). The vectors $v_1,\dots ,v_k$ also have a fixed form with these parameters. We only provide the vector $v$ for space convenience.
The case $k=5$ is proven slightly different and details are given below. Recall the two functions $S(i)=i^2$ and $T(i)=i(i+1)/2$.

\medskip
\paragraph{\it The reduced expression uses 1 negative basis vector}
As before, this case follows from the fact that~$\primalMatrix^G$ is a Coxeter signature matrix, which implies that any two adjacent cones lie in opposite sides of the hyperplane spanned by their interesection, see Lemma~\ref{lem:complete_fan2}. One can also check by inspection that flipping $-e_i$ in the negative orthant gives a vector whose $i$th coordinate is positive. 

\medskip
\paragraph{\it The reduced expression uses 2 negative basis vectors}\label{sssec:2letters}
Given the symmetry of the problem we need to consider only the reduced expressions 123121, 213213, 132132, 123212, 212321 up to commutation of letters. The reason is that for the bipartite case, $\dualCountingMatrix$ presents a nice symmetry with respect to interchanging consecutive letters 1 and 3 in $c^m$. In fact, two cones $C$ and $C^*$ intersect if and only if the resulting cones after doing this operation intersect. The same reasoning applies for any value of $k$. Below is the list of all possible cases (up to symmetry) together with a possible choice for the non-negative vector $v$.  

\begin{description}
\item[123121] For this reduced expression there is one possible case
\begin{itemize}
\item 12-1321 with embedding $2\1\3\2\13$: $v=(2T(b+1),T(a+1))$ 
\end{itemize}
\item[213213] For this reduced expression there are two possible cases
\begin{itemize}
\item 21-3213 with embedding $21\3\2\1\3$: $v=(1,0)$ 
\item 23-1213 with embedding $2\13\2\1\3$: \\symmetric to 21-3213 with embedding $21\3\2\1\3$ above
\end{itemize}
\item[132132] For this reduced expression there is one possible case
\begin{itemize}
\item 13-2132 with embedding $\2\1\3\213$: $v=(1,0)$ 
\end{itemize}
\item[123212] For this reduced expression there are zero possible cases
\item[212321] For this reduced expression there is one possible case
\begin{itemize}
\item 21-2321 with embedding $\21\3\2\13$: $v=(T(b+1),2T(a+1))$ 
\end{itemize}
\end{description}

\medskip
\paragraph{\it The reduced expression uses 3 negative basis vectors}\label{sssec:3letters}

\begin{description}
\item[123121] For this reduced expression there are two possible cases
\begin{itemize}
\item 123-121 with embedding $2\13\2\13$: $v=(2T(b+1),T(a+1),0)$ 
\item 121-321 with embedding $21\3\2\13$: $v=(0,T(c+1),2T(b+1))$
\end{itemize}
\item[213213] For this reduced expression there are five possible cases
\begin{itemize}
\item 213-213 with embedding $\2\1\3213$: $v=(0,1,0)$ 
\item 213-213 with embedding $\2\1321\3$: $v=(T(c),0,2T(a+1))$
\item 213-213 with embedding $\2132\1\3$: \\$v=((b+1)(c+1)(b+c+2),2(a+1)(a+2)(c+1),2(a+1)(a+2)(b+1))$
\item 213-213 with embedding $213\2\1\3$: $v=(1,0,0)$
\item 213-231 with embedding $\21\32\13$: \\symmetric to 213-213 with embedding $\2\1321\3$ above
\end{itemize}
\item[132132] For this reduced expression there is one possible case
\begin{itemize}
\item 132-132 with embedding $2\1\3\213$: \\$v=(2(b+1)(c+1)(c+2),2(a+1)(c+1)(c+2),(a+1)(b+1)(a+b+2))$ 
\end{itemize}
\item[123212] For this reduced expression there is one possible case
\begin{itemize}
\item 123-212 with embedding $\2\13\213$: $v=(0,T(c),2T(b+1))$ 
\end{itemize}
\item[212321] For this reduced expression there is one possible case
\begin{itemize}
\item 212-321 with embedding $21\3\2\13$: $v=(0,2T(c+1),T(b+1))$ 
\end{itemize}
\end{description}

\medskip
\paragraph{\it The reduced expression uses 4 negative basis vectors}\label{sssec:4letters}

\begin{description}
\item[123121] For this reduced expression there are three possible cases
\begin{itemize}
\item 1231-21 with embedding $\2\13213$: $v=(0,T(c),2T(b+1),0)$ 
\item 1231-21 with embedding $\2132\13$: \\$v=(0,(c+d+2)(c+1)(d+1),2(b+1)(b+2)(d+1),2(b+1)(b+2)(c+1))$ 
\item 1231-21 with embedding $213\2\13$: $v=(0,T(d+1),0,2T(b+1))$  \\
\end{itemize}

\item[213213] For this reduced expression there are four possible cases. For all these cases, the negative basis vectors correspond to the letters of a \defn{diamond shape} 2132. We define the \defn{diamond vector} $\DiamondVector$ as the vector: 
{\scriptsize
\[
\left((b+1)(c+1)(d+2)(b+c-2d), 2(a+2)(c+1)(d+2)(a-d), 2(a+2)(b+1)(d+2)(a-d), (2a-b-c)(a+2)(b+1)(c+1) \right)
\]}
\begin{itemize}
\item 2132-13 with embedding $2\1\3213$: $v= \DiamondVector$
\item 2132-13 with embedding $2\1321\3$: $v= \DiamondVector$ 
\item 2132-13 with embedding $2132\1\3$: $v= \DiamondVector$ 
\item 2132-31 with embedding $21\32\13$: $v= \DiamondVector$
\end{itemize}

\item[132132] For this reduced expression there are two possible cases
\begin{itemize}
\item 1321-32 with embedding $21\3\213$: \\$v=(2(b+1)(c+1)(c+2),2(a+1)(c+1)(c+2),(a+1)(b+1)(a+b+2),0)$ 
\item 1323-12 with embedding $2\13\213$: \\symmetric to 1321-32 with embedding $21\3\213$ above
\end{itemize}

\item[123212] For this reduced expression there is one possible case
\begin{itemize}
\item 1232-12 with embedding $2\13\213$: \\$v=(2(c+1)(d+1)(d+2),0,2(a+1)(d+1)(d+2),(a+1)(c+1)(a+c+2))$ 
\end{itemize}

\item[212321] For this reduced expression there are three possible cases
\begin{itemize}
\item 2123-21 with embedding $\2\13213$: $v=(0,0,T(d),2T(c+1))$ 
\item 2123-21 with embedding $\2132\13$: \\$v=((b+d+2)(b+1)(d+1),2(a+1)(a+2)(d+1),0,2(a+1)(a+2)(b+1))$ 
\item 2123-21 with embedding $213\2\13$: $v=(T(b+1),2T(a+1),0,0)$ 
\end{itemize}
\end{description}

\medskip
\paragraph{\it The reduced expression uses 5 negative basis vectors}\label{sssec:5letters}

In this case, we give a unbounded polyhedron $P$ that contains the cone spanned by $v_1,\dots,v_5$ and does not intersect the negative orthant in~$\R^5$. The polyhedron is the sum of a 3-dimensional subspace and a cone spanned by three vectors. This is shown explicitly for the reduced expression 12321-2, the other cases are obtained similarly. There are two possibilities in this case: $\213213$ and $213\213$, where the letter $\2$ corresponds to either the first or the fourth row of the last 6 rows in $M_{213,m}^t$. The cone is spanned by the rows of the $5\times 5$ matrix~$M$~(when removing row 1 or 4):
\[
\small
	M=\left(
\begin{array}{ccccc}
  -T(a+1) & S(b+2)  & -T(c+1) & S(d+2) & -T(e+1) \\
  -T(a)+1 & 2T(b+1) & -T(c+1) & 2T(d+1) & -T(e)+1 \\
  -T(a+1) & 2T(b+1) & -T(c)+1 & 2T(d+1) & -T(e+1) \\
  T(a+1)  & -S(b+2)+1 & T(c+1) & -S(d+2)+1 & T(e+1) \\
  T(a) &  -2T(b+1) & T(c+1) & -2T(d+1) & T(e) \\
  T(a+1) & -2T(b+1)  & T(c) & -2T(d+1) & T(e+1)
\end{array}
\right),
\]
where $a>b\geq c>d\geq e\geq1$ correspond to the copy of $c$ (counted from right to left) where the 5 letters are taken. The two possible cones are contained in the polyhedron $P$ whose lineal subspace is generated by row $1$, $5$ and $6$ and whose cone is generated by the vectors $(1,0,0,0,1),(0,0,1,0,0)$ and $(0,1,0,1,0)$. If $a+c=2b$, we have

\[
	\left(
\begin{array}{ccc}
  0 & \frac{-1}{(a+1)}  & \frac{1}{(a+1)} \\
  \frac{-(d+1)}{(b+2)(b-d)} &  \frac{-(a+2)}{2(b+2)(b-d)} & \frac{a-2(d+1)}{2(b+2)(b-d)}   \\
  \frac{(b+1)}{(d+2)(b-d)} & \frac{(a+2)}{2(d+2)(b-d)} & \frac{2(b+1)-a}{2(d+2)(b-d)}
\end{array}
\right)\cdot M_{1,5,6}=\left(
\begin{array}{ccccc}
  1  & 0 & \frac{-(c+1)}{(a+1)} & 0 & \frac{(e+1)}{(a+1)} \\
  0  & 1 &  \frac{-2(c+1)}{a+c+4} & 0 & \frac{2(e+1)(a-e)}{(a+c+4)} \\
  0  & 0 & 0 & 1 & \frac{-(e+1)(a-e)}{(d+2)(a+c-2d)}
\end{array}
\right),
\]
and it is impossible to have a linear combination of the rows resulting in a vector with all entries negative. In any other case, we have
\[
	\left(
\begin{array}{ccc}
  \frac{2(b+1)}{(a+1)} & \frac{2(b+1)-c}{(a+1)}  & \frac{(c+2)}{(a+1)} \\
  \frac{(a+c+2)}{(b+2)} &  \frac{(a+2)}{(b+2)} & \frac{(c+2)}{(b+2)}   \\
  \frac{2(b+1)}{(c+1)} & \frac{(a+2)}{(c+1)} & \frac{2(b+1)-a}{(c+1)}
\end{array}
\right)\cdot \frac{M_{1,5,6}}{a+c-2b}=\left(
\begin{array}{ccccc}
  1  & 0 & 0 &  \frac{2(d+2)(b-d)}{(a+1)(a+c-2b)} & \frac{-(e+1)(2b-c-e)}{(a+1)(a+c-2b)} \\
  0  & 1 &  0 & \frac{(d+2)(a+c-2d)}{(b+2)(a+c-2b)} & \frac{-(e+1)(a-e)}{(b+2)(a+c-2b)} \\
  0  & 0 & 1 &  \frac{2(d+2)(b-d)}{(c+1)(a+c-2b)} & \frac{-(e+1)(a-e)}{(c+1)(a+c-2b)}
\end{array}
\right),
\]
and it is again impossible to have a linear combination of the rows resulting in a vector with all entries negative. Therefore, the subspace generated by the rows~$1$, $5$ and $6$ does not intersect the negative orthant, and so neither does the polyhedron $P$.

\medskip
\paragraph{\it The reduced expression uses 6 negative basis vectors}\label{sssec:6letters}

We need to show that the cone spanned by~$v_1,v_2,\dots,v_6$ does not intersect the negative orthant in~$\R^6$. Since all reduced expressions of~$\wo$ contain $2132$ as a subword (up to commutations), we can restrict the study to the coordinates of~$v_1,v_2,\dots,v_6$ corresponding to these four letters. The resulting vectors are the rows of the following matrix
\[
\small
	M=\left(
\begin{array}{cccc}
  S(a+2) & -T(b+1)  & -T(c+1) & S(d+2) \\
  2T(a+1) &  -T(b){\bf +1} & -T(c+1) & 2T(d+1)  \\
  2T(a+1) & -T(b+1)  & -T(c){\bf +1} & 2T(d+1)  \\    
  -S(a+2){\bf +1} & T(b+1)  & T(c+1) & -S(d+2){\bf +1} \\  
  -2T(a+1) &  T(b) & T(c+1) & -2T(d+1)  \\
  -2T(a+1) & T(b+1)  & T(c) & -2T(d+1)  
\end{array}
\right),
\]

\noindent
where the parameters $a\geq b,c>d$ denote the copy of $c$ in which the letters of the \emph{diamond shape}~$2132$ appear in $c^{m-2}$ (counted from right to left). In all possible cases we can use the diamond vector $v=\DiamondVector$ defined above. The inner product of $\DiamondVector$ with any of the six rows of $M$ is a non-negative number. More explicitly, the inner product is equal to zero for any row not containing a~``$+1$" (rows 1,5 and 6), and strictly greater than zero for any row containing a~``$+1$" (rows 2,3 and 4). The hyperplane orthogonal to $v$ separates the positive span of~$v_1,v_2,\dots,v_6$ and the negative orthant.

\section{Polytopality}\label{sec:polytopality}

This section concerns the question of polytopality of spherical subword complexes and multi-associahedra. Following~\cite{de_loera_triangulations_2010}, we say that a fan is \defn{regular} if it is the normal fan of a polytope. This terminology comes from the relation between normal fans of polytopes and regular triangulations of point configurations, see~\cite[Section~2.2 and Section~9.5.3]{de_loera_triangulations_2010}.

\subsection{Regularity of $\fanp{\primalMatrix}$}

Let~$\subwordComplex[Q]$ be a spherical subword complex of type~$A_n$ with $n\leq 3$, and $\varphi$ be an embedding of $Q$ into $c^m$. By Theorem~\ref{thm:complete_fan}, the fan~$\fanp{\primalMatrix}$ is a complete simplicial fan realizing~$\subwordComplex[Q]$. One natural question is whether this fan is the normal fan of a polytope. The answer to this question is negative in general. In fact, most of the fans~$\fanp{\primalMatrix}$ are not regular. 

We restrict the verification of regularity to the family of multi-associahedra~$\Delta_{\ell,k}$, whose facets are in bijection with $k$-triangulations of a convex $\ell$-gon. Every other spherical subword complex of type $A$ can be obtained as the link of a face in a multi-associahedron (this is known as the universality of multi-associahedra of type $A$~\cite[Proposition 5.6]{PilaudSantos-brickPolytope}, see~\cite[Theorem 2.15]{ceballos_subword_2013} for the analogous universality result for all finite types). The multi-associahedron $\Delta_{n+2k+1,k}$ can be obtained as a subword complex~$\subwordComplex[Q]$ for any word $Q$ equal to $c^k\wo(c)$ up to commutations, where~$\wo(c)$ denotes the lexicographically first reduced expression of $\wo$ in $c^\infty$. We refer to~\cite{ceballos_subword_2013} for more details in this connection.
We used the computer algebra system Sage~\cite{sage} to produce all the words~$Q$ satisfying this property and tested regularity of the fans corresponding to embeddings of~$Q$ into a longer word~$c^m$ for different values of $m$ in types $A_2$ and $A_3$. The results are summarized in Tables~\ref{tab:polytopality_A2} and~\ref{tab:polytopality_A3}.

\begin{small}
\begin{table}[!htbp]
\begin{center}
\begin{tabular}{|c|c|c|cc|}
\hline
&&&\\[-.13in]
$\Delta_{2k+3,k}$ & $A_2$ & $m$ & Regular fans & Non-regular fans \\[0.05in]
\hline
&&&\\[-.13in]
$\Delta_{5,1}$ & $k=1$ & $m\leq 11$ & 6006 & 0\\
$\Delta_{7,2}$ & $k=2$ & $m\leq 11$ & 12870 & 0 \\
$\Delta_{9,3}$ & $k=3$ & $m\leq 11$ & 16016 & 0 \\
$\Delta_{11,4}$ & $k=4$ & $m\leq 11$ & 12376 & 0 \\
$\Delta_{13,5}$ & $k=5$ & $m\leq 10$ & 1360 & 0\\
$\Delta_{15,6}$ & $k=6$ & $m\leq 10$ & 306 & 0\\
\hline
\end{tabular}
\vspace{.05in}
\end{center}
\caption{\label{tab:polytopality_A2} 
Results of regularity tests on complete simplicial fans of type $A_2$.}
\vspace{-10pt}
\end{table}
\end{small}

\begin{small}
\begin{table}[!htbp]
\begin{center}
\begin{tabular}{|c|c|c|cc|}
\hline
&&&\\[-.13in]
$\Delta_{2k+4,k}$ & $A_3$ & $m$ & Regular fans & Non-regular fans \\[0.05in]
\hline
&&&\\[-.13in]
$\Delta_{6,1}$ & $k=1$ & $m\leq 12$ & 1 144 293 & 136 \\
$\Delta_{8,2}$ & $k=2$ & $m\leq 11$ & 66 293 & 743 560 \\
$\Delta_{10,3}$ & $k=3$ & $m\leq 10$ & 0 & 144 939 \\
\hline
\end{tabular}
\vspace{.05in}
\end{center}
\caption{\label{tab:polytopality_A3} 
Results of regularity tests on complete simplicial fans of type $A_3$.}
\vspace{-10pt}
\end{table}
\end{small}

The number of polytopal realizations of $\Delta_{8,2}$ is a significant improvement to the known polytopal realizations in~\cite{bokowski_symmetric_2009,ceballos_associahedra_2012}, where the authors use computational methods to obtain the polytopes. It is also quite surprising to have so many different non-regular fans realizing~$\Delta_{10,3}$. 

\subsection{Regularity obstruction for $\fanp{\primalMatrix}$}

Although the regularity of the fan depends on the choice of embedding, our computational results reveal that more than a hundred thousand possible fans realizing the multi-associahedron $\Delta_{10,3}$ are not the normal fan of a polytope. We do not know if there are embeddings into longer words $c^m$ for which the fan is regular. 
After investigating the fans of $\Delta_{10,3}$, we obtained the following \emph{obstruction} to the regularity of $\fanp{\primalMatrix}$. The fans associated to the subword complex $\Obs:=\subwordComplex[1212321212]$ embedded in words $c^m$ with $m\leq 10$ and $c$ any Coxeter element are not regular. This obstruction is \emph{minimal}: removing any letter to $\Obs$ yields a regular complete fan.

The subword complex $\Obs$ appears to be a possible candidate to disprove the polytopality conjecture for subword complexes. The $f$-vector of $\Obs$ is $(1, 9,30,42,21)$. Notice that the vertex corresponding to the letter $3$ is not a vertex of $\Obs$ since every expression of $w_\circ$ in $1212321212$ uses this letter. In fact, this word contains two possible reduced expressions, namely $123212$ and $212321$ which have opposite signs. As it turns out, according to the enumeration of 3-dimensional manifolds on 9 vertices of Altshuler--Steinberg, the subword complex $\Obs$ is polytopal, see \cite[Table~3]{altshuler_enumeration_1976}. Here is an explicit construction of a 4-dimensional polytope with~9 vertices whose boundary complex is $\Obs$. This construction is due to Francisco Santos and was obtained during the conference FPSAC 2013 in Paris. We are grateful to him and his great insight.

\begin{example}\label{ex:Santos}
	Let $\{1,2,3,4,5,6,7,8,9\}$ be the ground set for a simplicial complex constructed as follows. Let $C_1=(1234)$ and $C_2=(56789)$ be two cycles and $J=C_1\ast C_2$ their join. The simplicial complex $J$ is the boundary complex of the polar dual of the product of a square and a pentagon. The simplicial complex $\Obs$ is obtained from $J$ by removing $\{1,4\}\ast (56789)$ and filing back the following 6 tetraedra: $\{1,5,6,9\},\{1,6,7,9\},\{1,7,8,9\},\{4,5,6,9\},\{4,6,7,9\}$ and $\{4,7,8,9\}$. The first three and the last three triangulate a pentagonal pyramid each, together triangulating a pentagonal bipyramid. The following coordinates give an explicit realization of this polytope.
	\[
		\begin{array}{cccc}
			v_1=(-1,1,0,0) & v_2=(0,1,0,0) & v_3=(1,0,0,0) \\
			v_4=(1,-1,0,0) & v_5=(0,0,1,0) & v_6=(0,0,0,1) \\
			v_7=(0,0,-1,0) & v_8=(0,0,0,1) & v_9=(-1/4,-1/4,1,1).
		\end{array}
	\]
	The bijection sends the letters of the word $1212321212$ from left to right to the vertices $v_i$ with $1\leq i\leq 9$, skipping the letter $3$.
\end{example}

%

%
%
\section{Fan realization space}\label{sec:realization_space}

This section presents a large connected component of the fan realization space of every spherical subword complex of type~$A_3$ and every multi-associahedron~$\Delta_{2k+4,k}$. 
It is not a full dimensional component, but it contains a lot of information about the space. For instance, it contains all realizations of~$\subwordComplex[c^m]$ from Theorem~\ref{thm:complete_fan}, corresponding to all possible embeddings of $Q=c^m$ into a longer word $c^{m'}$. Similarly as in the previous section, we tested regularity of the fans arising in this part of the realization space. Surprisingly, non of the fans we tested for long words~$Q$ in type~$A_3$ were regular. In particular, non of the fans we tested realizing the multi-associahedron~$\Delta_{10,3}$ were regular.  
Given a collection of real parameters~${\bf a}=\{a_i,b_i,c_i\}_{i=1,\dots,m}$ define the ${\bf a}$-counting matrix~$\dualCountingMatrixp{123}{\bf a}$ as the matrix

\begin{minipage}{.4\textwidth}
\footnotesize
\[
	\dualCountingMatrixp{123}{\bf a}=
	\left(\begin{array}{c}
	\begin{array}{c@{\hspace{0.25cm}}|@{\hspace{0.5cm}}c@{\hspace{0.5cm}}|@{\hspace{0.25cm}}c}
	&& \\[1ex]
	\widetilde D_m^{\bf a} & \cdots & \widetilde D_1^{\bf a}\\
	&& \\[1ex]
	\end{array}
	\end{array}\right)_{6\times 3m}
\]
\end{minipage}
\qquad
\begin{minipage}{.4\textwidth}
\footnotesize
\[
	\widetilde D_i^{\bf a}=\left(\begin{array}{ccc}
	1 & 0 & 0 \\
	a_i & m-b_i+1 & 0 \\
	\binom{a_i+1}{2} & (m-b_i+1)(b_i) & \binom{m-c_i+2}{2} \\
	0 & 1 & 0 \\
	0 & b_i & m-c_i+1 \\
	0 & 0 & 1
	\end{array}\right),
\]
\end{minipage}

\medskip
\noindent
and denote by $M_{123,\bf a}$ any Gale dual matrix of $\dualCountingMatrixp{123}{\bf a}$. We are particularly interested in parameters satisfying the \defn{$c$-signature inequalities} in Figure~\ref{fig:parameter_inequalities_123}, for $c=(s_1,s_2,s_3)$.
The case where $a_i=b_i=c_i=i$ gives rise to the counting matrix $\dualCountingMatrixp{123}{m}$ in Appendix~\ref{app:matrices}.

\begin{figure}[!htbp]
  \begin{center}
  \begin{tikzpicture}[scale=0.78]

  \node (s1-1) at (1,1) {\small$a_m$};
  \node (s2-1) at (2,2) {\small$b_m$} edge[<-,thick] (s1-1);
  \node (s3-1) at (3,3) {\small$c_m$} edge[<-,thick] (s2-1);

  \node (s1-2) at (3,1) {\small$a_{m-1}$} edge[<-,thick] (s2-1);
  \node (s2-2) at (4,2) {\small$b_{m-1}$} edge[<-,thick] (s1-2) edge[<-,thick] (s3-1);
  \node (s3-2) at (5,3) {\small$c_{m-1}$} edge[<-,thick] (s2-2);

  \node (s1-3) at (5,1) {} edge[<-,thick] (s2-2);
  \node (s2-3) at (6,2) {} edge[<-,thick] (s1-3) edge[<-,thick] (s3-2);
  \node (s3-3) at (7,3) {} edge[<-,thick] (s2-3);

  \node (s1-4) at (7,1) {\small$a_2$} edge[<-,thick] (s2-3);
  \node (s2-4) at (8,2) {\small$b_2$} edge[<-,thick] (s1-4) edge[<-,thick] (s3-3);
  \node (s3-4) at (9,3) {\small$c_2$} edge[<-,thick] (s2-4);

  \node (s1-5) at (9,1) {\small$a_1$} edge[<-,thick] (s2-4);
  \node (s2-5) at (10,2) {\small$b_1$} edge[<-,thick] (s1-5) edge[<-,thick] (s3-4);
  \node (s3-5) at (11,3) {\small$c_1$} edge[<-,thick] (s2-5);

  \node (eq) at (-3.6,2) {
  \small
  $
\begin{array}{rcc}
 a_i-a_{i-1} &  > & 0  \\
 b_i-b_{i-1} &  > & 0  \\
 c_i-c_{i-1} &  > & 0  \\ 
 2b_i-c_i-a_{i-1}+2b_j-c_j-a_{j-1} &  > & 0  \\ 
 c_{i+1}+a_i-2b_i+c_{j+1}+a_j-2b_j &  > & 0  \\ 
\end{array}
$
  };  
  
  \draw[dashed,red] (1.4,2)--(2.7,3.3)--(3.8,3.3)--(2.5,2)--(3.8,0.7)--(2.7,0.7) -- cycle;
  \draw[dashed,red] (7.4,2)--(8.7,3.3)--(9.8,3.3)--(8.5,2)--(9.8,0.7)--(8.7,0.7) -- cycle;
  
  \end{tikzpicture}
  \end{center}
  \caption{Signature inequalities for $c=(s_1,s_2,s_3)$ in type $A_3$. The dashed shapes on the right show an example of the parameters involved in the fourth equation.}
  \label{fig:parameter_inequalities_123}
\end{figure}

\begin{theorem}\label{thm:fan_real_parameters_123}
Let~$c=(s_1,s_2,s_3)$ be a Coxeter element of type $A_3$ and $Q=c^m$, with $m\geq 3$. For any choice of real parameters ${\bf a}=\{a_i,b_i,c_i\}_{i=1,\dots,m}$ satisfying the $c$-signature inequalities in Figure~\ref{fig:parameter_inequalities_123}, the fan~$\fanp{M_{123,\bf a}}$ is a complete fan realizing~$\subwordComplex[Q]$.
\end{theorem} 

\begin{proof}
If the parameters ${\bf a}=\{a_i,b_i,c_i\}_{i=1,\dots,m}$ satisfy the $c$-signature inequalities in Figure~\ref{fig:parameter_inequalities_123} then they also satisfy the polynomial inequalities in Table~\ref{table:determinantsA3_2} for all reduced expressions of $\wo$ in $Q$. Therefore, the matrix~$\dualCountingMatrixp{123}{\bf a}$ is a Coxeter signature matrix for the pair $(Q,\wo)$. By Lemma~\ref{lem:complete_fan2}, this implies that any two adjacent cones of the fan~$\fanp{M_{123,\bf a}}$ lie in opposite sides of their intersection. Hence, $\fanp{M_{123,\bf a}}$ is an $\ell$-covering of the sphere. We need to show that $\ell=1$ for any choice of~$\bf a$. The space of solutions to the signature inequalities is a path-connected set. In particular, any two points in this set can be connected by a continuous path contained in the set. For a particular choice of Gale dual matrices, this induces a continuous path of matrices $M_{123,\bf a}$ for $\bf a$ in the path. The value of $\ell$ for all associated fans remains constant. Thus, the value of $\ell$ is constant for all fans~$\fanp{M_{123,\bf a}}$ with $\bf a$ satisfying the $c$-signature inequalities. We have seen in Corollary~\ref{cor:fan_123} that $\ell=1$ for the particular case where $a_i=b_i=c_i=i$. This finishes the proof.
\end{proof}

\begin{remark}
One can also consider parameters ${\bf a}=\{a_i,b_i,c_i\}_{i=1,\dots,m}$ satisfying the polynomial inequalities in Table~\ref{table:determinantsA3_2}, for all reduced expressions of $\wo$ in $c^m$. 
It turns out that the set of solutions can be described in a very simple way:
for $m\geq 4$, the parameters ${\bf a}=\{a_i,b_i,c_i\}$ satisfy the polynomial inequalities in Table~\ref{table:determinantsA3_2} if and only if they satisfy the linear signature inequalities in Figure~\ref{fig:parameter_inequalities_123} or their converses (replacing all symbols $>$ by $<$). In the converse situation, the inequalities can be obtained by rotating Figure~\ref{fig:parameter_inequalities_123} $180^{\circ}$ degrees. This operation only affects the fan by relabelling of the rays. For this reason, we restrict our study to the simpler linear system of signature inequalities. Figure~\ref{fig:polynomial_not-linear} shows an example in the extremal case $m=3$ where the shown parameters satisfy the polynomial inequalities but not the signature inequalities or their converses.
\end{remark}

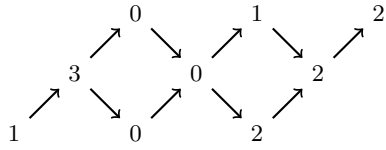
\begin{figure}[!htbp]
  \begin{center}
  \begin{tikzpicture}[scale=0.8]

  \node (s1-1) at (1,1) {\small$1$};
  \node (s2-1) at (2,2) {\small$3$} edge[<-,thick] (s1-1);
  \node (s3-1) at (3,3) {\small$0$} edge[<-,thick] (s2-1);

  \node (s1-2) at (3,1) {\small$0$} edge[<-,thick] (s2-1);
  \node (s2-2) at (4,2) {\small$0$} edge[<-,thick] (s1-2) edge[<-,thick] (s3-1);
  \node (s3-2) at (5,3) {\small$1$} edge[<-,thick] (s2-2);

  \node (s1-3) at (5,1) {\small$2$} edge[<-,thick] (s2-2);
  \node (s2-3) at (6,2) {\small$2$} edge[<-,thick] (s1-3) edge[<-,thick] (s3-2);
  \node (s3-3) at (7,3) {\small$2$} edge[<-,thick] (s2-3);

  \end{tikzpicture}
  \end{center}
  \caption{Example of parameters, for $m=3$, satisfying the polynomial inequalities in Table~\ref{table:determinantsA3_2} but not the signature inequalities (or their converses) in Figure~\ref{fig:parameter_inequalities_123}.}
  \label{fig:polynomial_not-linear}
\end{figure}

Similarly, we can also define an ${\bf a}$-counting matrix $\dualCountingMatrixp{213}{\bf a}$ for the bipartite Coxeter element $c=(s_2,s_1,s_3)$ as the matrix

\begin{minipage}{.4\textwidth}
\footnotesize
\[
	\dualCountingMatrixp{213}{\bf a}=\left(\begin{array}{c}
	\begin{array}{c@{\hspace{0.25cm}}|@{\hspace{0.5cm}}c@{\hspace{0.5cm}}|@{\hspace{0.25cm}}c}
	&& \\[1ex]
	\widehat D_m^{\bf a} & \cdots & \widehat D_1^{\bf a}\\
	&& \\[1ex]
	\end{array}
	\end{array}\right)_{6\times 3m}
\]
\end{minipage}
\qquad
\begin{minipage}{.4\textwidth}
\footnotesize
\[
	\widehat D_i^{\bf a}=\left(\begin{array}{ccc}
	1 & 0 & 0 \\
	a_i & m-b_i+1 & 0 \\
	a_i & 0 & m-c_i+1 \\
	a_i^2 & \binom{m+1}{2}-\binom{b_i}{2} & \binom{m+1}{2}-\binom{c_i}{2}\\
	0 & 0 & 1 \\ 
	0 & 1 & 0
	\end{array}\right)
\]
\end{minipage}

\medskip
\noindent
and denote by $M_{213,\bf a}$ any Gale dual matrix of $\dualCountingMatrixp{213}{\bf a}$. As before, we are interested in parameters satisfying the $c$-signature inequalities in Figure~\ref{fig:parameter_inequalities_213}, for $c=(s_2,s_1,s_3)$.
The case where $a_i=b_i=c_i=i$ gives rise to the counting matrix $\dualCountingMatrixp{213}{m}$ in Appendix~\ref{app:matrices}. 
We omit the proof of the following result since it is similar to the proof of Theorem~\ref{thm:fan_real_parameters_123} above.

\begin{theorem}\label{thm:fan_real_parameters_213}
Let~$c=(s_2,s_1,s_3)$ be a bipartite Coxeter element of type $A_3$ and $Q=c^m$, with $m\geq 3$. For any choice of real parameters ${\bf a}=\{a_i,b_i,c_i\}_{i=1,\dots,m}$ satisfying the $c$-signature inequalities in Figure~\ref{fig:parameter_inequalities_213}, the fan~$\fanp{M_{213,\bf a}}$ is a complete fan realizing~$\subwordComplex[Q]$. In particular, it is a complete fan realizing the multi-associahedron $\Delta_{2k+4,k}$ for $k=m-2$. 
\end{theorem}

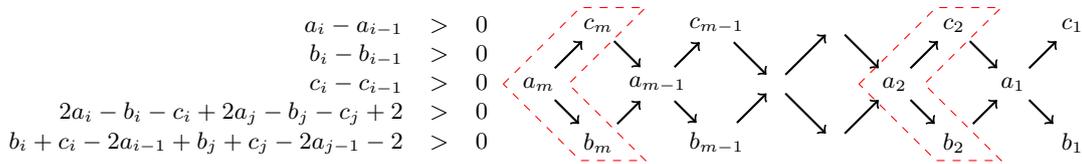
\begin{figure}[!htbp]
  \begin{center}
  \begin{tikzpicture}[scale=0.78]

  \node (s2-1) at (2,2) {\small$a_m$};
  \node (s3-1) at (3,3) {\small$c_m$} edge[<-,thick] (s2-1);
  \node (s1-2) at (3,1) {\small$b_m$} edge[<-,thick] (s2-1);
  
  \node (s2-2) at (4,2) {\small$a_{m-1}$} edge[<-,thick] (s1-2) edge[<-,thick] (s3-1);
  \node (s3-2) at (5,3) {\small$c_{m-1}$} edge[<-,thick] (s2-2);
  \node (s1-3) at (5,1) {\small$b_{m-1}$} edge[<-,thick] (s2-2);
  
  \node (s2-3) at (6,2) {} edge[<-,thick] (s1-3) edge[<-,thick] (s3-2);
  \node (s3-3) at (7,3) {} edge[<-,thick] (s2-3);
  \node (s1-4) at (7,1) {} edge[<-,thick] (s2-3);
  
  \node (s2-4) at (8,2) {\small$a_2$} edge[<-,thick] (s1-4) edge[<-,thick] (s3-3);
  \node (s3-4) at (9,3) {\small$c_2$} edge[<-,thick] (s2-4);
  \node (s1-5) at (9,1) {\small$b_2$} edge[<-,thick] (s2-4);
  
  \node (s2-5) at (10,2) {\small$a_1$} edge[<-,thick] (s1-5) edge[<-,thick] (s3-4);
  \node (s3-5) at (11,3) {\small$c_1$} edge[<-,thick] (s2-5);
  \node (s1-6) at (11,1) {\small$b_1$} edge[<-,thick] (s2-5);

  \node (eq) at (-2.9,2) {
  \small
  $
\begin{array}{rcc}
 a_i-a_{i-1} &  > & 0  \\
 b_i-b_{i-1} &  > & 0  \\
 c_i-c_{i-1} &  > & 0  \\ 
 2a_i-b_i-c_i+2a_j-b_j-c_j +2 &  > & 0  \\ 
 b_i+c_i-2a_{i-1}+b_j+c_j-2a_{j-1} -2 &  > & 0  \\ 
\end{array}
$
  };  
  
  \draw[dashed,red] (1.4,2)--(2.7,3.3)--(3.8,3.3)--(2.5,2)--(3.8,0.7)--(2.7,0.7) -- cycle;
  \draw[dashed,red] (7.4,2)--(8.7,3.3)--(9.8,3.3)--(8.5,2)--(9.8,0.7)--(8.7,0.7) -- cycle;
    
  \end{tikzpicture}
  \end{center}
  \caption{Signature inequalities for $c=(s_2,s_1,s_3)$ in type $A_3$. The dashed shapes on the right show an example of the parameters involved in the fourth equation.}
  \label{fig:parameter_inequalities_213}
\end{figure}

%

%
%

\section{Multi-associahedra of Type $A_4$}\label{sec:typeA4}

The construction of the fans presented in this paper only works for subword complexes of type~$A_n$ for~$n\leq 3$. However, doing a slight modification to the construction we were able to obtain fan realizations of the multi-associahedra~$\Delta_{9,2}$ and~$\Delta_{11,3}$, which are isomorphic to the subword complexes of type~$A_4$ corresponding to the words~$Q=c^2\wo(c)$ and~$Q=c^3\wo(c)$ respectively, for any Coxeter element~$c$. We worked with the bipartite Coxeter element $c=(2,4,1,3)$ and obtained the following results. In both cases, the cones of the fan are spanned by the row vectors corresponding to the faces of the associated subword complex. We refer to \cite[Section 2.4]{ceballos_subword_2013} for an explicit bijection between positions in the word $Q$ and $k$-relevant diagonals in the polygon.


\begin{example}[Multi-associahedron~$\Delta_{9,2}$]
The multi-associahedron~$\Delta_{9,2}$ can be realized as the complete simplicial fan whose rays are the row vectors of the matrix below and whose cones are spanned by the row vectors corresponding to faces of~$\Delta_{9,2}$. The pairs of numbers on the left of the matrix are the 2-relevant diagonals of the 9-gon associated to each of the rows of the matrix.
Remarkably, this fan is not the normal fan of a polytope. 
\end{example}

\[
\footnotesize
\begin{blockarray}{rrrrrrrrr}
\begin{block}{r(rrrrrrrr)}
1,6 & -1 & 0 & 0 & 0 & 0 & 0 & 0 & 0\ \\[0.5em]
2,5 & 0 & -1 & 0 & 0 & 0 & 0 & 0 & 0\ \\[0.5em]
1,7 & 0 & 0 & -1 & 0 & 0 & 0 & 0 & 0\ \\[0.5em]
2,6 & 0 & 0 & 0 & -1 & 0 & 0 & 0 & 0\ \\[0.5em]
2,7 & 0 & 0 & 0 & 0 & -1 & 0 & 0 & 0\ \\[0.5em]
3,6 & 0 & 0 & 0 & 0 & 0 & -1 & 0 & 0\ \\[0.5em]
2,8 & 0 & 0 & 0 & 0 & 0 & 0 & -1 & 0\ \\[0.5em]
3,7 & 0 & 0 & 0 & 0 & 0 & 0 & 0 & -1\ \\[0.5em]
3,8 & \frac{77}{5} & 4 & -4 & -9 & 6 & 1 & -1 & -2\ \\[0.5em]
4,7 & \frac{189}{20} & 10 & -1 & -12 & 3 & 4 & 0 & -3\ \\[0.5em]
3,9 & \frac{199}{10} & 8 & -2 & -15 & 6 & 2 & 1 & -3\ \\[0.5em]
4,8 & \frac{201}{10} & 8 & -5 & -12 & 6 & 2 & -1 & -1\ \\[0.5em]
4,9 & -\frac{141}{20} & 0 & 3 & 3 & -3 & 0 & 1 & 1\ \\[0.5em]
5,8 & \frac{17}{20} & -8 & -3 & 9 & 0 & -3 & -1 & 3\ \\[0.5em]
1,4 & -\frac{397}{20} & -8 & 3 & 15 & -6 & -2 & 0 & 3\ \\[0.5em]
5,9 & -20 & -8 & 5 & 13 & -6 & -2 & 1 & 2\ \\[0.5em]
1,5 & -\frac{73}{10} & -4 & 1 & 6 & -2 & -1 & 0 & 1\ \\[0.5em]
6,9 & -\frac{41}{4} & -1 & 4 & 3 & -3 & 0 & 1 & 0\ \\
\end{block}
\end{blockarray}
 \]


\begin{example}[Multi-associahedron~$\Delta_{11,3}$]
The multi-associahedron~$\Delta_{11,3}$ can be realized as the complete simplicial fan whose rays are the row vectors of the matrix below and whose cones are spanned by the row vectors corresponding to faces of~$\Delta_{11,3}$. The pairs of numbers on the left of the matrix are the 3-relevant diagonals of the 11-gon associated to each of the rows of the matrix.
Remarkably, this fan is not the normal fan of a polytope. 
\end{example}

\[
\scriptsize
\begin{blockarray}{lrrrrrrrrrrrr}
\begin{block}{l(rrrrrrrrrrrr)}
1,7 & -1 & 0 & 0 & 0 & 0 & 0 & 0 & 0 & 0 &0 & 0 & 0\hspace{0.1cm} \\[0.5em]
2,6 & 0 & -1 & 0 & 0 & 0 & 0 & 0 & 0 & 0 & 0 & 0 & 0\hspace{0.1cm} \\[0.5em]
1,8 & 0 & 0 & -1 & 0 & 0 & 0 & 0 & 0 & 0 & 0 & 0 & 0\hspace{0.1cm} \\[0.5em]
2,7 & 0 & 0 & 0 & -1 & 0 & 0 & 0 & 0 & 0 & 0 & 0 & 0\hspace{0.1cm} \\[0.5em]
2,8 & 0 & 0 & 0 & 0 & -1 & 0 & 0 & 0 & 0 & 0 & 0 & 0\hspace{0.1cm} \\[0.5em]
3,7 & 0 & 0 & 0 & 0 & 0 & -1 & 0 & 0 & 0 & 0 & 0 & 0\hspace{0.1cm} \\[0.5em]
2,9 & 0 & 0 & 0 & 0 & 0 & 0 & -1 & 0 & 0 & 0 & 0 & 0\hspace{0.1cm} \\[0.5em]
3,8 & 0 & 0 & 0 & 0 & 0 & 0 & 0 & -1 & 0 & 0 & 0 & 0\hspace{0.1cm} \\[0.5em]
3,9 & 0 & 0 & 0 & 0 & 0 & 0 & 0 & 0 & -1 & 0 & 0 & 0\hspace{0.1cm} \\[0.5em]
4,8 & 0 & 0 & 0 & 0 & 0 & 0 & 0 & 0 & 0 & -1 & 0 & 0\hspace{0.1cm} \\[0.5em]
3,10 & 0 & 0 & 0 & 0 & 0 & 0 & 0 & 0 & 0 & 0 & -1 & 0\hspace{0.1cm} \\[0.5em]
4,9 & 0 & 0 & 0 & 0 & 0 & 0 & 0 & 0 & 0 & 0 & 0 & -1\hspace{0.1cm} \\[0.5em]
4,10 & \frac{34142}{2005} & \frac{2600}{1203} & -\frac{6200}{1203} & -\frac{3880}{401} & \frac{3960}{401} & \frac{1040}{1203} & -\frac{2840}{1203} & -\frac{1780}{401} & \frac{1720}{401} & \frac{260}{1203} & -\frac{860}{1203} & -\frac{1720}{1203}\hspace{0.1cm} \\[0.5em]
5,9 & \frac{2196909}{200500} & \frac{5424}{401} & -\frac{133}{401} & -\frac{74871}{4010} & \frac{11946}{2005} & \frac{14858}{2005} & \frac{749}{4010} & -\frac{34131}{4010} & \frac{3852}{2005} & \frac{6722}{2005} & \frac{721}{4010} & -\frac{5294}{2005}\hspace{0.1cm} \\[0.5em]
4,11 & \frac{958753}{40100} & \frac{3090}{401} & -\frac{1045}{401} & -\frac{15915}{802} & \frac{4680}{401} & \frac{1236}{401} & \frac{375}{802} & -\frac{6431}{802} & \frac{1413}{401} & \frac{309}{401} & \frac{1133}{802} & -\frac{872}{401}\hspace{0.1cm} \\[0.5em]
5,10 & \frac{478059}{20050} & \frac{9205}{1203} & -\frac{7792}{1203} & -\frac{12513}{802} & \frac{4581}{401} & \frac{3682}{1203} & -\frac{5951}{2406} & -\frac{1968}{401} & \frac{1370}{401} & \frac{1841}{2406} & -\frac{685}{1203} & -\frac{167}{1203}\hspace{0.1cm} \\[0.5em]
5,11 & -\frac{1344229}{200500} & \frac{3398}{1203} & \frac{4945}{1203} & \frac{2501}{4010} & -\frac{9596}{2005} & \frac{6796}{6015} & \frac{28123}{12030} & \frac{4961}{4010} & -\frac{4492}{2005} & \frac{1699}{6015} & \frac{10507}{12030} & \frac{4492}{6015}\hspace{0.1cm} \\[0.5em]
6,10 & \frac{87853}{200500} & -\frac{5977}{401} & -\frac{2219}{401} & \frac{74643}{4010} & -\frac{1908}{2005} & -\frac{15964}{2005} & -\frac{11057}{4010} & \frac{37623}{4010} & -\frac{1011}{2005} & -\frac{5996}{2005} & -\frac{3673}{4010} & \frac{6352}{2005}\hspace{0.1cm} \\[0.5em]
1,5 & -\frac{2356367}{100250} & -\frac{3034}{401} & \frac{1405}{401} & \frac{39228}{2005} & -\frac{23046}{2005} & -\frac{6068}{2005} & \frac{993}{2005} & \frac{15888}{2005} & -\frac{6972}{2005} & -\frac{1517}{2005} & -\frac{843}{2005} & \frac{4329}{2005}\hspace{0.1cm} \\[0.5em]
6,11 & -\frac{47046}{2005} & -\frac{9050}{1203} & \frac{7700}{1203} & \frac{6565}{401} & -\frac{4530}{401} & -\frac{3620}{1203} & \frac{2945}{1203} & \frac{2340}{401} & -\frac{1360}{401} & -\frac{905}{1203} & \frac{680}{1203} & \frac{1360}{1203}\hspace{0.1cm} \\[0.5em]
1,6 & -\frac{400064}{50125} & -\frac{5468}{1203} & \frac{824}{1203} & \frac{16277}{2005} & -\frac{6939}{2005} & -\frac{10936}{6015} & -\frac{674}{6015} & \frac{5762}{2005} & -\frac{1738}{2005} & -\frac{2734}{6015} & -\frac{1136}{6015} & \frac{3743}{6015}\hspace{0.1cm} \\[0.5em]
7,11 & -\frac{2203731}{200500} & \frac{1010}{401} & \frac{2311}{401} & -\frac{891}{4010} & -\frac{9684}{2005} & \frac{3223}{2005} & \frac{10159}{4010} & -\frac{3871}{4010} & -\frac{2748}{2005} & \frac{1307}{2005} & \frac{2921}{4010} & -\frac{1089}{2005}\hspace{0.1cm} \\
\end{block}
\end{blockarray}
 \]

In the remaining of this section, we explain how we obtained these results. As in the construction in type $A_3$, the key point is to find a matrix having the Signature property and whose Gale dual matrix has the Injectivity property for the required subword complex. Such a Gale dual matrix will give rise to complete simplicial fan realizing the subword complex. The signature matrices for the subword complexes associated to~$\Delta_{9,2}$ and~$\Delta_{11,3}$ were found applying a slight modification to the counting matrices in Definition~\ref{def:dual_counting_matrix}. It remained to check the Injectivity property of a Gale dual matrix. The multi-associahedron $\Delta_{9,2}$ is 8-dimensional with 594 maximal cones, and $\Delta_{11,3}$ is 12-dimensional with 4719 maximal cones. It is worth mentioning that the completeness verification in these cases has a practical complexity reaching very fast a bottleneck. To avoid these complications, we verified the Injectivity condition using the computer algebra system Sage~\cite{sage} and finally used Theorem~\ref{thm:fan_reformulation} to conclude. In practice, the Injectivity condition needed in Theorem~\ref{thm:fan_reformulation} usually holds for signature matrices.

Let~$c=(s_2,s_4,s_1,s_3)$ be a bipartite Coxeter element of type~$A_4$ and~$Q=c^k\wo(c)$. For this particular Coxeter element~$\wo(c)= (2,4,1,3,2,4,1,3,2,4)$. The subword complex~$\subwordComplex[Q]$ is isomorphic to the multi-associahedron~$\Delta_{2k+5,k}$. Consider the counting matrix from Definition~\ref{def:dual_counting_matrix} which counts reduced expressions of $c_\alpha$'s in this particular word. We tested if this matrix is a signature matrix for different values of $k$ and obtained the following results. 

\begin{table}[!htdp]
\begin{tabular}{|c|c|c|c|c|}
\hline
&&&&\\[-.13in]
$k$ & good signs & bad signs & zero & total \\[0.05in]
\hline
1& 42 & 0 & 0 & 42 \\
2& 593 & 0 & 1 & 594 \\
3& 4702 & 0 & 17 & 4719 \\
4& 25905 & 6 & 115 & 26026 \\
\hline
\end{tabular}
\caption{Signs of the determinants of reduced expressions of~$\wo$ in $c^k\wo(c)$ for the counting matrix associated to multi-associahedra of type $A_4$.}
\label{tab:signatureA4}
\end{table}%

\vspace{-0.4cm}
In the cases $k=2$ and $k=3$ almost all determinants had the right signs with the exception of a few that were equal to zero. We modified some of the columns involved in the zero determinants in such a way that they become non-zero determinants with the right signs. Making this modification small enough makes sure that all other determinants still have good signs. Note that we are not able to obtain complete simplicial fan realizations for the case~$k=4$ because such a small modification would not correct the bad signs that appear. The signature matrices we obtained are given below. Their Gale dual matrices are the transposes of the matrices above.

\noindent
{\bf Signature matrix for $\Delta_{9,2}$.} 
\[
\footnotesize
\left(\begin{array}{rrrrrrrrrrrrrrrrrr}
\frac{21}{20} & 0 & 0 & 0 & 1 & 0 & 0 & 0
& 1 & 0 & 0 & 0 & 1 & 0 & 0 & 0 & 1
& 0 \\[0.5em]
\frac{1}{20} & 1 & 0 & 0 & 0 & 1 & 0 & 0
& 0 & 1 & 0 & 0 & 0 & 1 & 0 & 0 & 0
& 1 \\[0.5em]
\frac{81}{20} & 0 & 1 & 0 & 3 & 0 & 2 & 0
& 2 & 0 & 3 & 0 & 1 & 0 & 4 & 0 & 0
& 0 \\[0.5em]
\frac{201}{20} & 10 & 0 & 1 & 9 & 9 & 0 & 4
& 7 & 7 & 0 & 9 & 4 & 4 & 0 & 16 & 0
& 0 \\[0.5em]
\frac{801}{20} & 30 & 10 & 4 & 27 & 26 & 19
& 14 & 14 & 19 & 26 & 27 & 4 & 10 & 30
& 40 & 0 & 0 \\[0.5em]
\frac{81}{20} & 0 & 0 & 1 & 3 & 0 & 0 & 2
& 2 & 0 & 0 & 3 & 1 & 0 & 0 & 4 & 0
& 0 \\[0.5em]
\frac{1}{20} & 4 & 0 & 1 & 0 & 3 & 0 & 2
& 0 & 2 & 0 & 3 & 0 & 1 & 0 & 4 & 0
& 0 \\[0.5em]
\frac{321}{20} & 0 & 4 & 4 & 9 & 0 & 7 & 7
& 4 & 0 & 9 & 9 & 1 & 0 & 10 & 10 &
0 & 0 \\[0.5em]
\frac{1}{10} & 0 & 0 & 1 & 0 & 0 & 0 & 1
& 0 & 0 & 0 & 1 & 0 & 0 & 0 & 1 & 0
& 0 \\[0.5em]
\frac{1}{20} & 0 & 1 & 0 & 0 & 0 & 1 & 0
& 0 & 0 & 1 & 0 & 0 & 0 & 1 & 0 & 0
& 0
\end{array}\right)
\]

\bigskip
\noindent
{\bf Signature matrix for $\Delta_{11,3}$:} 
\[
\footnotesize
\left(\begin{array}{rrrrrrrrrrrrrrrrrrrrrr}
\frac{51}{50} & 0 & 0 & 0 & 1 & 0 & 0 & 0
& 1 & 0 & 0 & 0 & \frac{201}{200} & 0 & 0
& 0 & 1 & 0 & 0 & \frac{3}{50} & 1 & 0 \\[0.5em]
\frac{1}{50} & 1 & 0 & 0 & 0 & 1 & 0 & 0
& 0 & 1 & 0 & 0 & \frac{1}{200} & 1 & 0
& 0 & 0 & 1 & 0 & \frac{1}{50} & 0 & 1 \\[0.5em]
\frac{127}{25} & 0 & 1 & 0 & 4 & 0 & 2 & 0
& 3 & 0 & 3 & 0 & \frac{401}{200} & 0 & 4
& 0 & 1 & 0 & 5 & \frac{1}{50} & 0 & 0 \\[0.5em]
\frac{376}{25} & 15 & 0 & 1 & 14 & 14 & 0 &
4 & 12 & 12 & 0 & 9 & \frac{1801}{200} & 9 &
0 & 16 & 5 & 5 & 0 & \frac{1251}{50} & 0 & 0
\\[0.5em]
\frac{3751}{50} & 55 & 15 & 5 & 56 & 50 & 29
& 18 & 36 & 41 & 41 & 36 & \frac{3601}{200}
& 29 & 50 & 56 & 5 & 15 & 55 &
\frac{3751}{50} & 0 & 0 \\[0.5em]
\frac{251}{50} & 0 & 0 & 1 & 4 & 0 & 0 & 2
& 3 & 0 & 0 & 3 & \frac{401}{200} & 0 & 0
& 4 & 1 & 0 & 0 & \frac{251}{50} & 0 & 0 \\[0.5em]
\frac{1}{50} & 5 & 0 & 1 & 0 & 4 & 0 & 2
& 0 & 3 & 0 & 3 & \frac{1}{200} & 2 & 0
& 4 & 0 & 1 & 0 & \frac{251}{50} & 0 & 0 \\[0.5em]
\frac{1251}{50} & 0 & 5 & 5 & 16 & 0 & 9 & 9
& 9 & 0 & 12 & 12 & \frac{801}{200} & 0 & 14
& 14 & 1 & 0 & 15 & \frac{751}{50} & 0 & 0
\\[0.5em]
\frac{2}{25} & 0 & 0 & 1 & 0 & 0 & 0 & 1
& 0 & 0 & 0 & 1 & \frac{1}{100} & 0 & 0
& 1 & 0 & 0 & 0 & \frac{51}{50} & 0 & 0 \\[0.5em]
\frac{1}{50} & 0 & 1 & 0 & 0 & 0 & 1 & 0
& 0 & 0 & 1 & 0 & \frac{1}{200} & 0 & 1
& 0 & 0 & 0 & 1 & \frac{1}{50} & 0 & 0
\end{array}\right)
\]

\appendix
\section{Counting matrices and their Gale dual matrices: general formulas}\label{app:matrices}

In this appendix, we present general formulas for the counting matrices~$\dualCountingMatrix$ of type $A_n$ with~$n\leq 3$. We also present formulas for their Gale dual matrices, which are used to exhibit explicit coordinates for the complete simplicial fans realizing the subword complexes and multi-associahedra of type~$A_3$ in Corollaries~\ref{cor:fan_bipartite213} and~\ref{cor:fan_123}.

\subsection{Counting matrices}
The matrices~$\dualCountingMatrix$ of type $A_n$ with $n\leq 3$ are given by: 
\begin{description}
\item[Type $A_1$, $c=(s_1)$] 
\[
\footnotesize
\dualCountingMatrixp{c}{m} 
=
\left(
\begin{array}{cccccc}
  1 & 1  & \dots & 1  & 1 & 1     
\end{array}
\right)_{1\times m}
\]

\item[Type $A_2$, $c=(s_1,s_2)$] 
\[
\footnotesize
\dualCountingMatrixp{c}{m} 
=
\left(
\begin{array}{ccccccc}
  1& 0  & 1 & 0  &\dots& 1 & 0      \\
 m &  1 & m-1 & 2  &\dots& 1 & m     \\
  0 & 1  & 0 & 1  &\dots& 0 & 1     
\end{array}
\right)_{3\times 2m}
\]
rows {\small \{$\alpha_1$, $\alpha_1+\alpha_2$,  $\alpha_2$\}} in this order. 
\smallskip

\item[Type $A_3$, $c=(s_2,s_1,s_3)$] 
\hfill \\
\begin{minipage}{.4\textwidth}
\footnotesize
\[
	\dualCountingMatrixp{c}{m}=\left(\begin{array}{c}
	\begin{array}{c@{\hspace{0.25cm}}|@{\hspace{0.5cm}}c@{\hspace{0.5cm}}|@{\hspace{0.25cm}}c}
	&& \\[1ex]
	\widehat D_m & \cdots & \widehat D_1\\
	&& \\[1ex]
	\end{array}
	\end{array}\right)_{6\times 3m}
\]
\end{minipage}
\qquad
\begin{minipage}{.4\textwidth}
\footnotesize
\[
	\widehat D_i=\left(\begin{array}{ccc}
	1 & 0 & 0 \\
	i & m-i+1 & 0 \\
	i & 0 & m-i+1 \\
	i^2 & \binom{m+1}{2}-\binom{i}{2} & \binom{m+1}{2}-\binom{i}{2}\\
	0 & 0 & 1 \\ 
	0 & 1 & 0
	\end{array}\right)
\]
\end{minipage}

\noindent
rows \small{\{$\alpha_2$, $\alpha_1+\alpha_2$, $\alpha_2+\alpha_3$, $\alpha_1+\alpha_2+\alpha_3$,  $\alpha_3$, $\alpha_1$\}} in this order.
\smallskip

\item[Type $A_3$, $c=(s_1,s_2,s_3)$] 
\hfill \\
\begin{minipage}{.4\textwidth}
\footnotesize
\[
	\dualCountingMatrixp{c}{m}=\left(\begin{array}{c}
	\begin{array}{c@{\hspace{0.25cm}}|@{\hspace{0.5cm}}c@{\hspace{0.5cm}}|@{\hspace{0.25cm}}c}
	&& \\[1ex]
	\widetilde D_m & \cdots & \widetilde D_1\\
	&& \\[1ex]
	\end{array}
	\end{array}\right)_{6\times 3m}
\]
\end{minipage}
\qquad
\begin{minipage}{.4\textwidth}
\footnotesize
\[
	\widetilde D_i=\left(\begin{array}{ccc}
	1 & 0 & 0 \\
	i & m-i+1 & 0 \\
	\binom{i+1}{2} & (m-i+1)(i) & \binom{m-i+2}{2} \\
	0 & 1 & 0 \\
	0 & i & m-i+1 \\
	0 & 0 & 1
	\end{array}\right)
\]
\end{minipage}

\noindent
rows \small{\{$\alpha_1$, $\alpha_1+\alpha_2$, $\alpha_1+\alpha_2+\alpha_3$,  $\alpha_2$, $\alpha_2+\alpha_3$, $\alpha_3$\}} in this order.
\smallskip

\end{description}

\subsection{Gale dual matrices}\label{ap:Gale_dual}

Explicit choices for Gale duals matrices $M_{c,m}$ of~$\dualCountingMatrix$ are:
\begin{description}
\item[Type $A_1$, $c=(s_1)$] 
{\footnotesize
\[
	M_{1,m}=
	\left(\begin{array}{c} \hspace{-0.3cm}
	\begin{array}{c@{\hspace{0.25cm}}|@{\hspace{0.1cm}}r@{\hspace{-0.1cm}}}
	&1 \\[-1ex]
	-I_{m-1} & \vdots  \\
	&1 \\[2ex]
	\end{array}
	\end{array}\right)_{(m-1)\times m}
\]}
\item[Type $A_2$, $c=(s_1,s_2)$] 
\hfill \\
\begin{minipage}{.4\textwidth}
\footnotesize
\[
	M_{12,m}=
	\left(\begin{array}{c@{\hspace{0.25cm}}|@{\hspace{-0.01cm}}c@{\hspace{-0.05cm}}}
	-I_{2m-3}&
	\begin{array}{ccc}
	&E_1 & \\
	\hline \\[-0.3cm]
	&\vdots& \\[0.1cm]
	\hline \\[-0.1cm]
	&E_{m-2}& \\[0.1cm]
	\hline \\[-0.2cm]
	-1 & 1 & 1
	\end{array}
	\end{array}\right)_{(2m-3)\times (2m)}
\]
\end{minipage}
\qquad
\begin{minipage}{.4\textwidth}
\footnotesize
\[
	\footnotesize
	E_i:=\left(\begin{array}{ccc}
	-m+i & 1 & m-i \\
	m-i & 0 & -m+i+1 
	\end{array}\right)
\]
\end{minipage}

\bigskip
\item[Type $A_3$, $c=(s_2,s_1,s_3)$] Matrix $M_{213,m}$ in Corollary~\ref{cor:fan_bipartite213}
\item[Type $A_3$, $c=(s_1,s_2,s_3)$] Matrix $M_{123,m}$ in Corollary~\ref{cor:fan_123}.
\end{description}

\bibliographystyle{alpha}
\bibliography{SubwordComplexFans}
\label{sec:biblio}

\end{document}

%% file: figures/Associahedra/Asso1.tex
\begin{tikzpicture}%
	[x={(0.885124cm, -0.336154cm)},
	y={(0.284499cm, 0.938124cm)},
	z={(-0.368261cm, -0.083209cm)},
	scale=0.0675,
	back/.style={loosely dotted, thin},
	edge/.style={color=blue!95!black, thick},
	facet/.style={fill=blue!95!black,fill opacity=0.500000},
	vertex/.style={inner sep=0pt,circle,draw=green!25!black,fill=green!75!black,thick,anchor=base}]
%
%
\coordinate (37.1, 13.8, 19.7) at (37.1, 13.8, 19.7);
\coordinate (12.2, 8.80, 19.7) at (12.2, 8.80, 19.7);
\coordinate (32.0, 13.8, 19.7) at (32.0, 13.8, 19.7);
\coordinate (4.48, 8.80, 19.7) at (4.48, 8.80, 19.7);
\coordinate (52.5, 22.9, 33.3) at (52.5, 22.9, 33.3);
\coordinate (52.5, 19.9, 28.9) at (52.5, 19.9, 28.9);
\coordinate (52.5, 43.5, 99.6) at (52.5, 43.5, 99.6);
\coordinate (52.5, 45.6, 99.6) at (52.5, 45.6, 99.6);
\coordinate (6.45, 10.4, 19.7) at (6.45, 10.4, 19.7);
\coordinate (38.9, 40.8, 99.6) at (38.9, 40.8, 99.6);
\coordinate (39.4, 45.6, 85.7) at (39.4, 45.6, 85.7);
\coordinate (52.5, 45.6, 81.9) at (52.5, 45.6, 81.9);
\coordinate (37.1, 45.6, 99.6) at (37.1, 45.6, 99.6);
\coordinate (31.1, 40.8, 99.6) at (31.1, 40.8, 99.6);
\draw[edge,back] (37.1, 13.8, 19.7) -- (32.0, 13.8, 19.7);
\draw[edge,back] (32.0, 13.8, 19.7) -- (52.5, 22.9, 33.3);
\draw[edge,back] (32.0, 13.8, 19.7) -- (6.45, 10.4, 19.7);
\draw[edge,back] (4.48, 8.80, 19.7) -- (6.45, 10.4, 19.7);
\draw[edge,back] (6.45, 10.4, 19.7) -- (39.4, 45.6, 85.7);
\node[vertex] at (6.45, 10.4, 19.7)     {};
\node[vertex] at (32.0, 13.8, 19.7)     {};
\fill[facet] (31.1, 40.8, 99.6) -- (38.9, 40.8, 99.6) -- (52.5, 43.5, 99.6) -- (52.5, 45.6, 99.6) -- (37.1, 45.6, 99.6) -- cycle {};
\fill[facet] (31.1, 40.8, 99.6) -- (4.48, 8.80, 19.7) -- (12.2, 8.80, 19.7) -- (38.9, 40.8, 99.6) -- cycle {};
\fill[facet] (38.9, 40.8, 99.6) -- (12.2, 8.80, 19.7) -- (37.1, 13.8, 19.7) -- (52.5, 19.9, 28.9) -- (52.5, 43.5, 99.6) -- cycle {};
\fill[facet] (39.4, 45.6, 85.7) -- (37.1, 45.6, 99.6) -- (52.5, 45.6, 99.6) -- (52.5, 45.6, 81.9) -- cycle {};
\fill[facet] (52.5, 45.6, 81.9) -- (52.5, 22.9, 33.3) -- (52.5, 19.9, 28.9) -- (52.5, 43.5, 99.6) -- (52.5, 45.6, 99.6) -- cycle {};
\draw[edge] (37.1, 13.8, 19.7) -- (12.2, 8.80, 19.7);
\draw[edge] (37.1, 13.8, 19.7) -- (52.5, 19.9, 28.9);
\draw[edge] (12.2, 8.80, 19.7) -- (4.48, 8.80, 19.7);
\draw[edge] (12.2, 8.80, 19.7) -- (38.9, 40.8, 99.6);
\draw[edge] (4.48, 8.80, 19.7) -- (31.1, 40.8, 99.6);
\draw[edge] (52.5, 22.9, 33.3) -- (52.5, 19.9, 28.9);
\draw[edge] (52.5, 22.9, 33.3) -- (52.5, 45.6, 81.9);
\draw[edge] (52.5, 19.9, 28.9) -- (52.5, 43.5, 99.6);
\draw[edge] (52.5, 43.5, 99.6) -- (52.5, 45.6, 99.6);
\draw[edge] (52.5, 43.5, 99.6) -- (38.9, 40.8, 99.6);
\draw[edge] (52.5, 45.6, 99.6) -- (52.5, 45.6, 81.9);
\draw[edge] (52.5, 45.6, 99.6) -- (37.1, 45.6, 99.6);
\draw[edge] (38.9, 40.8, 99.6) -- (31.1, 40.8, 99.6);
\draw[edge] (39.4, 45.6, 85.7) -- (52.5, 45.6, 81.9);
\draw[edge] (39.4, 45.6, 85.7) -- (37.1, 45.6, 99.6);
\draw[edge] (37.1, 45.6, 99.6) -- (31.1, 40.8, 99.6);
\node[vertex] at (37.1, 13.8, 19.7)     {};
\node[vertex] at (12.2, 8.80, 19.7)     {};
\node[vertex] at (4.48, 8.80, 19.7)     {};
\node[vertex] at (52.5, 22.9, 33.3)     {};
\node[vertex] at (52.5, 19.9, 28.9)     {};
\node[vertex] at (52.5, 43.5, 99.6)     {};
\node[vertex] at (52.5, 45.6, 99.6)     {};
\node[vertex] at (38.9, 40.8, 99.6)     {};
\node[vertex] at (39.4, 45.6, 85.7)     {};
\node[vertex] at (52.5, 45.6, 81.9)     {};
\node[vertex] at (37.1, 45.6, 99.6)     {};
\node[vertex] at (31.1, 40.8, 99.6)     {};
\end{tikzpicture}

%% file: figures/Associahedra/Asso2.tex
\begin{tikzpicture}%
	[x={(0.869613cm, -0.382202cm)},
	y={(0.289799cm, 0.907657cm)},
	z={(-0.399737cm, -0.173439cm)},
	scale=0.027000,
	back/.style={loosely dotted, thin},
	edge/.style={color=blue!95!black, thick},
	facet/.style={fill=orange!95!black,fill opacity=0.500000},
	vertex/.style={inner sep=0pt,circle,draw=green!25!black,fill=green!75!black,thick,anchor=base}]
%
%
\coordinate (-30.3, -16.1, -1.00) at (-30.3, -16.1, -1.00);
\coordinate (-25.2, -8.69, -1.00) at (-25.2, -8.69, -1.00);
\coordinate (1.30, 0.250, -1.00) at (1.30, 0.250, -1.00);
\coordinate (64.9, 0.250, -1.00) at (64.9, 0.250, -1.00);
\coordinate (31.3, -16.1, -1.00) at (31.3, -16.1, -1.00);
\coordinate (14.0, 58.0, 147.) at (14.0, 58.0, 147.);
\coordinate (82.3, 13.3, 9.44) at (82.3, 13.3, 9.44);
\coordinate (82.3, 61.0, 47.6) at (82.3, 61.0, 47.6);
\coordinate (50.3, 58.0, 147.) at (50.3, 58.0, 147.);
\coordinate (39.6, 93.4, 116.) at (39.6, 93.4, 116.);
\coordinate (38.3, 93.4, 147.) at (38.3, 93.4, 147.);
\coordinate (82.3, 73.6, 147.) at (82.3, 73.6, 147.);
\coordinate (82.3, 93.4, 147.) at (82.3, 93.4, 147.);
\coordinate (82.3, 93.4, 94.7) at (82.3, 93.4, 94.7);
\draw[edge,back] (-30.3, -16.1, -1.00) -- (-25.2, -8.69, -1.00);
\draw[edge,back] (-25.2, -8.69, -1.00) -- (1.30, 0.250, -1.00);
\draw[edge,back] (-25.2, -8.69, -1.00) -- (39.6, 93.4, 116.);
\draw[edge,back] (1.30, 0.250, -1.00) -- (64.9, 0.250, -1.00);
\draw[edge,back] (1.30, 0.250, -1.00) -- (82.3, 61.0, 47.6);
\node[vertex] at (-25.2, -8.69, -1.00)     {};
\node[vertex] at (1.30, 0.250, -1.00)     {};
\fill[facet] (82.3, 93.4, 94.7) -- (82.3, 61.0, 47.6) -- (82.3, 13.3, 9.44) -- (82.3, 73.6, 147.) -- (82.3, 93.4, 147.) -- cycle {};
\fill[facet] (82.3, 93.4, 94.7) -- (39.6, 93.4, 116.) -- (38.3, 93.4, 147.) -- (82.3, 93.4, 147.) -- cycle {};
\fill[facet] (50.3, 58.0, 147.) -- (31.3, -16.1, -1.00) -- (-30.3, -16.1, -1.00) -- (14.0, 58.0, 147.) -- cycle {};
\fill[facet] (82.3, 93.4, 147.) -- (38.3, 93.4, 147.) -- (14.0, 58.0, 147.) -- (50.3, 58.0, 147.) -- (82.3, 73.6, 147.) -- cycle {};
\fill[facet] (82.3, 73.6, 147.) -- (82.3, 13.3, 9.44) -- (64.9, 0.250, -1.00) -- (31.3, -16.1, -1.00) -- (50.3, 58.0, 147.) -- cycle {};
\draw[edge] (-30.3, -16.1, -1.00) -- (31.3, -16.1, -1.00);
\draw[edge] (-30.3, -16.1, -1.00) -- (14.0, 58.0, 147.);
\draw[edge] (64.9, 0.250, -1.00) -- (31.3, -16.1, -1.00);
\draw[edge] (64.9, 0.250, -1.00) -- (82.3, 13.3, 9.44);
\draw[edge] (31.3, -16.1, -1.00) -- (50.3, 58.0, 147.);
\draw[edge] (14.0, 58.0, 147.) -- (50.3, 58.0, 147.);
\draw[edge] (14.0, 58.0, 147.) -- (38.3, 93.4, 147.);
\draw[edge] (82.3, 13.3, 9.44) -- (82.3, 61.0, 47.6);
\draw[edge] (82.3, 13.3, 9.44) -- (82.3, 73.6, 147.);
\draw[edge] (82.3, 61.0, 47.6) -- (82.3, 93.4, 94.7);
\draw[edge] (50.3, 58.0, 147.) -- (82.3, 73.6, 147.);
\draw[edge] (39.6, 93.4, 116.) -- (38.3, 93.4, 147.);
\draw[edge] (39.6, 93.4, 116.) -- (82.3, 93.4, 94.7);
\draw[edge] (38.3, 93.4, 147.) -- (82.3, 93.4, 147.);
\draw[edge] (82.3, 73.6, 147.) -- (82.3, 93.4, 147.);
\draw[edge] (82.3, 93.4, 147.) -- (82.3, 93.4, 94.7);
\node[vertex] at (-30.3, -16.1, -1.00)     {};
\node[vertex] at (64.9, 0.250, -1.00)     {};
\node[vertex] at (31.3, -16.1, -1.00)     {};
\node[vertex] at (14.0, 58.0, 147.)     {};
\node[vertex] at (82.3, 13.3, 9.44)     {};
\node[vertex] at (82.3, 61.0, 47.6)     {};
\node[vertex] at (50.3, 58.0, 147.)     {};
\node[vertex] at (39.6, 93.4, 116.)     {};
\node[vertex] at (38.3, 93.4, 147.)     {};
\node[vertex] at (82.3, 73.6, 147.)     {};
\node[vertex] at (82.3, 93.4, 147.)     {};
\node[vertex] at (82.3, 93.4, 94.7)     {};
\end{tikzpicture}

%% file: figures/Associahedra/Asso5.tex
\begin{tikzpicture}%
	[x={(0.911169cm, -0.319239cm)},
	y={(0.032155cm, 0.685386cm)},
	z={(-0.410777cm, -0.654471cm)},
	scale=0.405000,
	back/.style={loosely dotted, thin},
	edge/.style={color=blue!95!black, thick},
	facet/.style={fill=yellow!95!black,fill opacity=0.500000},
	vertex/.style={inner sep=0pt,circle,draw=green!25!black,fill=green!75!black,thick,anchor=base}]
%
%
\coordinate (-49.4, -37.1, -16.8) at (-49.4, -37.1, -16.8);
\coordinate (-49.0, -37.1, -16.8) at (-49.0, -37.1, -16.8);
\coordinate (-54.1, -44.1, -28.5) at (-54.1, -44.1, -28.5);
\coordinate (-51.9, -44.1, -28.5) at (-51.9, -44.1, -28.5);
\coordinate (-47.4, -34.0, -17.6) at (-47.4, -34.0, -17.6);
\coordinate (-48.6, -41.2, -28.5) at (-48.6, -41.2, -28.5);
\coordinate (-47.4, -34.0, -16.8) at (-47.4, -34.0, -16.8);
\coordinate (-53.9, -43.8, -28.5) at (-53.9, -43.8, -28.5);
\coordinate (-46.0, -35.3, -16.8) at (-46.0, -35.3, -16.8);
\coordinate (-46.0, -34.0, -16.8) at (-46.0, -34.0, -16.8);
\coordinate (-47.0, -41.2, -28.5) at (-47.0, -41.2, -28.5);
\coordinate (-46.0, -34.0, -18.7) at (-46.0, -34.0, -18.7);
\coordinate (-46.0, -39.0, -27.0) at (-46.0, -39.0, -27.0);
\coordinate (-46.0, -40.3, -27.8) at (-46.0, -40.3, -27.8);
\draw[edge,back] (-49.0, -37.1, -16.8) -- (-51.9, -44.1, -28.5);
\draw[edge,back] (-54.1, -44.1, -28.5) -- (-51.9, -44.1, -28.5);
\draw[edge,back] (-51.9, -44.1, -28.5) -- (-47.0, -41.2, -28.5);
\draw[edge,back] (-48.6, -41.2, -28.5) -- (-47.0, -41.2, -28.5);
\draw[edge,back] (-47.0, -41.2, -28.5) -- (-46.0, -40.3, -27.8);
\node[vertex] at (-51.9, -44.1, -28.5)     {};
\node[vertex] at (-47.0, -41.2, -28.5)     {};
\fill[facet] (-46.0, -40.3, -27.8) -- (-46.0, -35.3, -16.8) -- (-46.0, -34.0, -16.8) -- (-46.0, -34.0, -18.7) -- (-46.0, -39.0, -27.0) -- cycle {};
\fill[facet] (-47.4, -34.0, -17.6) -- (-53.9, -43.8, -28.5) -- (-54.1, -44.1, -28.5) -- (-49.4, -37.1, -16.8) -- (-47.4, -34.0, -16.8) -- cycle {};
\fill[facet] (-46.0, -39.0, -27.0) -- (-48.6, -41.2, -28.5) -- (-53.9, -43.8, -28.5) -- (-47.4, -34.0, -17.6) -- (-46.0, -34.0, -18.7) -- cycle {};
\fill[facet] (-46.0, -34.0, -16.8) -- (-47.4, -34.0, -16.8) -- (-49.4, -37.1, -16.8) -- (-49.0, -37.1, -16.8) -- (-46.0, -35.3, -16.8) -- cycle {};
\fill[facet] (-46.0, -34.0, -18.7) -- (-47.4, -34.0, -17.6) -- (-47.4, -34.0, -16.8) -- (-46.0, -34.0, -16.8) -- cycle {};
\draw[edge] (-49.4, -37.1, -16.8) -- (-49.0, -37.1, -16.8);
\draw[edge] (-49.4, -37.1, -16.8) -- (-54.1, -44.1, -28.5);
\draw[edge] (-49.4, -37.1, -16.8) -- (-47.4, -34.0, -16.8);
\draw[edge] (-49.0, -37.1, -16.8) -- (-46.0, -35.3, -16.8);
\draw[edge] (-54.1, -44.1, -28.5) -- (-53.9, -43.8, -28.5);
\draw[edge] (-47.4, -34.0, -17.6) -- (-47.4, -34.0, -16.8);
\draw[edge] (-47.4, -34.0, -17.6) -- (-53.9, -43.8, -28.5);
\draw[edge] (-47.4, -34.0, -17.6) -- (-46.0, -34.0, -18.7);
\draw[edge] (-48.6, -41.2, -28.5) -- (-53.9, -43.8, -28.5);
\draw[edge] (-48.6, -41.2, -28.5) -- (-46.0, -39.0, -27.0);
\draw[edge] (-47.4, -34.0, -16.8) -- (-46.0, -34.0, -16.8);
\draw[edge] (-46.0, -35.3, -16.8) -- (-46.0, -34.0, -16.8);
\draw[edge] (-46.0, -35.3, -16.8) -- (-46.0, -40.3, -27.8);
\draw[edge] (-46.0, -34.0, -16.8) -- (-46.0, -34.0, -18.7);
\draw[edge] (-46.0, -34.0, -18.7) -- (-46.0, -39.0, -27.0);
\draw[edge] (-46.0, -39.0, -27.0) -- (-46.0, -40.3, -27.8);
\node[vertex] at (-49.4, -37.1, -16.8)     {};
\node[vertex] at (-49.0, -37.1, -16.8)     {};
\node[vertex] at (-54.1, -44.1, -28.5)     {};
\node[vertex] at (-47.4, -34.0, -17.6)     {};
\node[vertex] at (-48.6, -41.2, -28.5)     {};
\node[vertex] at (-47.4, -34.0, -16.8)     {};
\node[vertex] at (-53.9, -43.8, -28.5)     {};
\node[vertex] at (-46.0, -35.3, -16.8)     {};
\node[vertex] at (-46.0, -34.0, -16.8)     {};
\node[vertex] at (-46.0, -34.0, -18.7)     {};
\node[vertex] at (-46.0, -39.0, -27.0)     {};
\node[vertex] at (-46.0, -40.3, -27.8)     {};
\end{tikzpicture}

%% file: figures/Associahedra/Asso6.tex
\begin{tikzpicture}%
	[x={(-0.672434cm, -0.630739cm)},
	y={(-0.098279cm, -0.442545cm)},
	z={(-0.733603cm, 0.637434cm)},
	scale=0.018000,
	back/.style={loosely dotted, thin},
	edge/.style={color=blue!95!black, thick},
	facet/.style={fill=green!95!black,fill opacity=0.500000},
	vertex/.style={inner sep=0pt,circle,draw=green!25!black,fill=green!75!black,thick,anchor=base}]
%
%
\coordinate (-30.8, 4.94, 109.) at (-30.8, 4.94, 109.);
\coordinate (-89.3, -68.2, -1.00) at (-89.3, -68.2, -1.00);
\coordinate (-69.7, -68.2, -1.00) at (-69.7, -68.2, -1.00);
\coordinate (-3.93, 4.94, 109.) at (-3.93, 4.94, 109.);
\coordinate (-58.0, -29.0, -1.00) at (-58.0, -29.0, -1.00);
\coordinate (-0.0778, 43.4, 85.9) at (-0.0778, 43.4, 85.9);
\coordinate (-0.0778, 43.4, 109.) at (-0.0778, 43.4, 109.);
\coordinate (-1.00, -0.500, -1.00) at (-1.00, -0.500, -1.00);
\coordinate (32.0, 32.5, 32.0) at (32.0, 32.5, 32.0);
\coordinate (11.5, -0.500, -1.00) at (11.5, -0.500, -1.00);
\coordinate (32.0, 20.0, 19.5) at (32.0, 20.0, 19.5);
\coordinate (32.0, 43.4, 53.8) at (32.0, 43.4, 53.8);
\coordinate (32.0, 34.9, 109.) at (32.0, 34.9, 109.);
\coordinate (32.0, 43.4, 109.) at (32.0, 43.4, 109.);
\draw[edge,back] (-89.3, -68.2, -1.00) -- (-69.7, -68.2, -1.00);
\draw[edge,back] (-69.7, -68.2, -1.00) -- (-3.93, 4.94, 109.);
\draw[edge,back] (-69.7, -68.2, -1.00) -- (11.5, -0.500, -1.00);
\draw[edge,back] (32.0, 20.0, 19.5) -- (32.0, 34.9, 109.);
\node[vertex] at (-69.7, -68.2, -1.00)     {};
\fill[facet] (32.0, 43.4, 109.) -- (-0.0778, 43.4, 109.) -- (-30.8, 4.94, 109.) -- (-3.93, 4.94, 109.) -- (32.0, 34.9, 109.) -- cycle {};
\fill[facet] (32.0, 43.4, 109.) -- (-0.0778, 43.4, 109.) -- (-0.0778, 43.4, 85.9) -- (32.0, 43.4, 53.8) -- cycle {};
\fill[facet] (32.0, 43.4, 53.8) -- (-0.0778, 43.4, 85.9) -- (-58.0, -29.0, -1.00) -- (-1.00, -0.500, -1.00) -- (32.0, 32.5, 32.0) -- cycle {};
\fill[facet] (32.0, 20.0, 19.5) -- (32.0, 32.5, 32.0) -- (-1.00, -0.500, -1.00) -- (11.5, -0.500, -1.00) -- cycle {};
\fill[facet] (-0.0778, 43.4, 109.) -- (-30.8, 4.94, 109.) -- (-89.3, -68.2, -1.00) -- (-58.0, -29.0, -1.00) -- (-0.0778, 43.4, 85.9) -- cycle {};
\draw[edge] (-30.8, 4.94, 109.) -- (-89.3, -68.2, -1.00);
\draw[edge] (-30.8, 4.94, 109.) -- (-3.93, 4.94, 109.);
\draw[edge] (-30.8, 4.94, 109.) -- (-0.0778, 43.4, 109.);
\draw[edge] (-89.3, -68.2, -1.00) -- (-58.0, -29.0, -1.00);
\draw[edge] (-3.93, 4.94, 109.) -- (32.0, 34.9, 109.);
\draw[edge] (-58.0, -29.0, -1.00) -- (-0.0778, 43.4, 85.9);
\draw[edge] (-58.0, -29.0, -1.00) -- (-1.00, -0.500, -1.00);
\draw[edge] (-0.0778, 43.4, 85.9) -- (-0.0778, 43.4, 109.);
\draw[edge] (-0.0778, 43.4, 85.9) -- (32.0, 43.4, 53.8);
\draw[edge] (-0.0778, 43.4, 109.) -- (32.0, 43.4, 109.);
\draw[edge] (-1.00, -0.500, -1.00) -- (32.0, 32.5, 32.0);
\draw[edge] (-1.00, -0.500, -1.00) -- (11.5, -0.500, -1.00);
\draw[edge] (32.0, 32.5, 32.0) -- (32.0, 20.0, 19.5);
\draw[edge] (32.0, 32.5, 32.0) -- (32.0, 43.4, 53.8);
\draw[edge] (11.5, -0.500, -1.00) -- (32.0, 20.0, 19.5);
\draw[edge] (32.0, 43.4, 53.8) -- (32.0, 43.4, 109.);
\draw[edge] (32.0, 34.9, 109.) -- (32.0, 43.4, 109.);
\node[vertex] at (-30.8, 4.94, 109.)     {};
\node[vertex] at (-89.3, -68.2, -1.00)     {};
\node[vertex] at (-3.93, 4.94, 109.)     {};
\node[vertex] at (-58.0, -29.0, -1.00)     {};
\node[vertex] at (-0.0778, 43.4, 85.9)     {};
\node[vertex] at (-0.0778, 43.4, 109.)     {};
\node[vertex] at (-1.00, -0.500, -1.00)     {};
\node[vertex] at (32.0, 32.5, 32.0)     {};
\node[vertex] at (11.5, -0.500, -1.00)     {};
\node[vertex] at (32.0, 20.0, 19.5)     {};
\node[vertex] at (32.0, 43.4, 53.8)     {};
\node[vertex] at (32.0, 34.9, 109.)     {};
\node[vertex] at (32.0, 43.4, 109.)     {};
\end{tikzpicture}

%% file: figures/Associahedra/Asso3.tex
\begin{tikzpicture}%
	[x={(-0.452930cm, -0.621780cm)},
	y={(0.675566cm, 0.228300cm)},
	z={(-0.581777cm, 0.749179cm)},
	scale=0.0112500,
	back/.style={loosely dotted, thin},
	edge/.style={color=blue!95!black, thick},
	facet/.style={fill=red!95!black,fill opacity=0.500000},
	vertex/.style={inner sep=0pt,circle,draw=green!25!black,fill=green!75!black,thick,anchor=base}]
%
%
\coordinate (-186., -290., -74.7) at (-186., -290., -74.7);
\coordinate (-252., -234., -78.5) at (-252., -234., -78.5);
\coordinate (-91.5, -144., 101.) at (-91.5, -144., 101.);
\coordinate (-252., -312., -137.) at (-252., -312., -137.);
\coordinate (-138., -101., 101.) at (-138., -101., 101.);
\coordinate (-82.5, 156., 101.) at (-82.5, 156., 101.);
\coordinate (-108., 156., 46.8) at (-108., 156., 46.8);
\coordinate (-134., -90.1, -108.) at (-134., -90.1, -108.);
\coordinate (-54.9, 156., -15.6) at (-54.9, 156., -15.6);
\coordinate (0.750, -0.167, -0.100) at (0.750, -0.167, -0.100);
\coordinate (40.6, 61.9, 37.1) at (40.6, 61.9, 37.1);
\coordinate (40.6, 61.9, 101.) at (40.6, 61.9, 101.);
\coordinate (40.6, 156., 51.3) at (40.6, 156., 51.3);
\coordinate (40.6, 156., 101.) at (40.6, 156., 101.);
\draw[edge,back] (-252., -234., -78.5) -- (-252., -312., -137.);
\draw[edge,back] (-252., -234., -78.5) -- (-138., -101., 101.);
\draw[edge,back] (-252., -234., -78.5) -- (-108., 156., 46.8);
\node[vertex] at (-252., -234., -78.5)     {};
\fill[facet] (40.6, 156., 101.) -- (-82.5, 156., 101.) -- (-108., 156., 46.8) -- (-54.9, 156., -15.6) -- (40.6, 156., 51.3) -- cycle {};
\fill[facet] (40.6, 156., 101.) -- (-82.5, 156., 101.) -- (-138., -101., 101.) -- (-91.5, -144., 101.) -- (40.6, 61.9, 101.) -- cycle {};
\fill[facet] (40.6, 156., 101.) -- (40.6, 61.9, 101.) -- (40.6, 61.9, 37.1) -- (40.6, 156., 51.3) -- cycle {};
\fill[facet] (40.6, 156., 51.3) -- (-54.9, 156., -15.6) -- (-134., -90.1, -108.) -- (0.750, -0.167, -0.100) -- (40.6, 61.9, 37.1) -- cycle {};
\fill[facet] (0.750, -0.167, -0.100) -- (-186., -290., -74.7) -- (-252., -312., -137.) -- (-134., -90.1, -108.) -- cycle {};
\fill[facet] (40.6, 61.9, 101.) -- (-91.5, -144., 101.) -- (-186., -290., -74.7) -- (0.750, -0.167, -0.100) -- (40.6, 61.9, 37.1) -- cycle {};
\draw[edge] (-186., -290., -74.7) -- (-91.5, -144., 101.);
\draw[edge] (-186., -290., -74.7) -- (-252., -312., -137.);
\draw[edge] (-186., -290., -74.7) -- (0.750, -0.167, -0.100);
\draw[edge] (-91.5, -144., 101.) -- (-138., -101., 101.);
\draw[edge] (-91.5, -144., 101.) -- (40.6, 61.9, 101.);
\draw[edge] (-252., -312., -137.) -- (-134., -90.1, -108.);
\draw[edge] (-138., -101., 101.) -- (-82.5, 156., 101.);
\draw[edge] (-82.5, 156., 101.) -- (-108., 156., 46.8);
\draw[edge] (-82.5, 156., 101.) -- (40.6, 156., 101.);
\draw[edge] (-108., 156., 46.8) -- (-54.9, 156., -15.6);
\draw[edge] (-134., -90.1, -108.) -- (-54.9, 156., -15.6);
\draw[edge] (-134., -90.1, -108.) -- (0.750, -0.167, -0.100);
\draw[edge] (-54.9, 156., -15.6) -- (40.6, 156., 51.3);
\draw[edge] (0.750, -0.167, -0.100) -- (40.6, 61.9, 37.1);
\draw[edge] (40.6, 61.9, 37.1) -- (40.6, 61.9, 101.);
\draw[edge] (40.6, 61.9, 37.1) -- (40.6, 156., 51.3);
\draw[edge] (40.6, 61.9, 101.) -- (40.6, 156., 101.);
\draw[edge] (40.6, 156., 51.3) -- (40.6, 156., 101.);
\node[vertex] at (-186., -290., -74.7)     {};
\node[vertex] at (-91.5, -144., 101.)     {};
\node[vertex] at (-252., -312., -137.)     {};
\node[vertex] at (-138., -101., 101.)     {};
\node[vertex] at (-82.5, 156., 101.)     {};
\node[vertex] at (-108., 156., 46.8)     {};
\node[vertex] at (-134., -90.1, -108.)     {};
\node[vertex] at (-54.9, 156., -15.6)     {};
\node[vertex] at (0.750, -0.167, -0.100)     {};
\node[vertex] at (40.6, 61.9, 37.1)     {};
\node[vertex] at (40.6, 61.9, 101.)     {};
\node[vertex] at (40.6, 156., 51.3)     {};
\node[vertex] at (40.6, 156., 101.)     {};
\end{tikzpicture}

%% file: figures/Associahedra/Asso4.tex
\begin{tikzpicture}%
	[x={(0.068016cm, 0.964094cm)},
	y={(0.664353cm, 0.148190cm)},
	z={(-0.744318cm, 0.220369cm)},
	scale=0.010,
	back/.style={loosely dotted, thin},
	edge/.style={color=blue!95!black, thick},
	facet/.style={fill=purple!95!black,fill opacity=0.500000},
	vertex/.style={inner sep=0pt,circle,draw=green!25!black,fill=green!75!black,thick,anchor=base}]
%
%
\coordinate (-133., 248., 117.) at (-133., 248., 117.);
\coordinate (-185., 88.2, 80.5) at (-185., 88.2, 80.5);
\coordinate (-24.7, 248., 50.9) at (-24.7, 248., 50.9);
\coordinate (-157., 109., 202.) at (-157., 109., 202.);
\coordinate (-18.5, 9.94, 202.) at (-18.5, 9.94, 202.);
\coordinate (-117., 248., 202.) at (-117., 248., 202.);
\coordinate (-35.5, -62.7, -51.7) at (-35.5, -62.7, -51.7);
\coordinate (-41.7, -65.4, -73.1) at (-41.7, -65.4, -73.1);
\coordinate (-24.7, -24.0, -65.8) at (-24.7, -24.0, -65.8);
\coordinate (3.53, 4.20, 6.80) at (3.53, 4.20, 6.80);
\coordinate (15.9, 28.9, 43.8) at (15.9, 28.9, 43.8);
\coordinate (15.9, 248., 202.) at (15.9, 248., 202.);
\coordinate (15.9, 59.1, 202.) at (15.9, 59.1, 202.);
\coordinate (15.9, 248., 138.) at (15.9, 248., 138.);
\draw[edge,back] (-133., 248., 117.) -- (-117., 248., 202.);
\draw[edge,back] (-157., 109., 202.) -- (-117., 248., 202.);
\draw[edge,back] (-117., 248., 202.) -- (15.9, 248., 202.);
\node[vertex] at (-117., 248., 202.)     {};
\fill[facet] (15.9, 248., 138.) -- (-24.7, 248., 50.9) -- (-24.7, -24.0, -65.8) -- (3.53, 4.20, 6.80) -- (15.9, 28.9, 43.8) -- cycle {};
\fill[facet] (15.9, 248., 202.) -- (15.9, 248., 138.) -- (15.9, 28.9, 43.8) -- (15.9, 59.1, 202.) -- cycle {};
\fill[facet] (15.9, 59.1, 202.) -- (-18.5, 9.94, 202.) -- (-35.5, -62.7, -51.7) -- (3.53, 4.20, 6.80) -- (15.9, 28.9, 43.8) -- cycle {};
\fill[facet] (-41.7, -65.4, -73.1) -- (-185., 88.2, 80.5) -- (-157., 109., 202.) -- (-18.5, 9.94, 202.) -- (-35.5, -62.7, -51.7) -- cycle {};
\fill[facet] (-24.7, -24.0, -65.8) -- (-24.7, 248., 50.9) -- (-133., 248., 117.) -- (-185., 88.2, 80.5) -- (-41.7, -65.4, -73.1) -- cycle {};
\fill[facet] (3.53, 4.20, 6.80) -- (-35.5, -62.7, -51.7) -- (-41.7, -65.4, -73.1) -- (-24.7, -24.0, -65.8) -- cycle {};
\draw[edge] (-133., 248., 117.) -- (-185., 88.2, 80.5);
\draw[edge] (-133., 248., 117.) -- (-24.7, 248., 50.9);
\draw[edge] (-185., 88.2, 80.5) -- (-157., 109., 202.);
\draw[edge] (-185., 88.2, 80.5) -- (-41.7, -65.4, -73.1);
\draw[edge] (-24.7, 248., 50.9) -- (-24.7, -24.0, -65.8);
\draw[edge] (-24.7, 248., 50.9) -- (15.9, 248., 138.);
\draw[edge] (-157., 109., 202.) -- (-18.5, 9.94, 202.);
\draw[edge] (-18.5, 9.94, 202.) -- (-35.5, -62.7, -51.7);
\draw[edge] (-18.5, 9.94, 202.) -- (15.9, 59.1, 202.);
\draw[edge] (-35.5, -62.7, -51.7) -- (-41.7, -65.4, -73.1);
\draw[edge] (-35.5, -62.7, -51.7) -- (3.53, 4.20, 6.80);
\draw[edge] (-41.7, -65.4, -73.1) -- (-24.7, -24.0, -65.8);
\draw[edge] (-24.7, -24.0, -65.8) -- (3.53, 4.20, 6.80);
\draw[edge] (3.53, 4.20, 6.80) -- (15.9, 28.9, 43.8);
\draw[edge] (15.9, 28.9, 43.8) -- (15.9, 59.1, 202.);
\draw[edge] (15.9, 28.9, 43.8) -- (15.9, 248., 138.);
\draw[edge] (15.9, 248., 202.) -- (15.9, 59.1, 202.);
\draw[edge] (15.9, 248., 202.) -- (15.9, 248., 138.);
\node[vertex] at (-133., 248., 117.)     {};
\node[vertex] at (-185., 88.2, 80.5)     {};
\node[vertex] at (-24.7, 248., 50.9)     {};
\node[vertex] at (-157., 109., 202.)     {};
\node[vertex] at (-18.5, 9.94, 202.)     {};
\node[vertex] at (-35.5, -62.7, -51.7)     {};
\node[vertex] at (-41.7, -65.4, -73.1)     {};
\node[vertex] at (-24.7, -24.0, -65.8)     {};
\node[vertex] at (3.53, 4.20, 6.80)     {};
\node[vertex] at (15.9, 28.9, 43.8)     {};
\node[vertex] at (15.9, 248., 202.)     {};
\node[vertex] at (15.9, 59.1, 202.)     {};
\node[vertex] at (15.9, 248., 138.)     {};
\end{tikzpicture}

%% file: figures/Associahedra/Asso7.tex
\begin{tikzpicture}%
	[x={(-0.493225cm, -0.848744cm)},
	y={(0.320685cm, -0.381172cm)},
	z={(-0.808635cm, 0.366526cm)},
	scale=0.01125000,
	back/.style={loosely dotted, thin},
	edge/.style={color=blue!95!black, thick},
	facet/.style={fill=pink!95!black,fill opacity=0.500000},
	vertex/.style={inner sep=0pt,circle,draw=green!25!black,fill=green!75!black,thick,anchor=base}]
%
%
\coordinate (-81.0, -64.4, -52.6) at (-81.0, -64.4, -52.6);
\coordinate (-9.00, 175., 19.3) at (-9.00, 175., 19.3);
\coordinate (-71.0, -167., -71.0) at (-71.0, -167., -71.0);
\coordinate (-104., -189., -146.) at (-104., -189., -146.);
\coordinate (-2.71, 175., -5.80) at (-2.71, 175., -5.80);
\coordinate (-62.3, -62.8, -125.) at (-62.3, -62.8, -125.);
\coordinate (0.375, -0.125, 0.375) at (0.375, -0.125, 0.375);
\coordinate (48.1, 95.3, 95.8) at (48.1, 95.3, 95.8);
\coordinate (48.1, 175., 95.8) at (48.1, 175., 95.8);
\coordinate (48.1, 70.7, 170.) at (48.1, 70.7, 170.);
\coordinate (19.2, -6.28, 170.) at (19.2, -6.28, 170.);
\coordinate (48.1, 175., 170.) at (48.1, 175., 170.);
\coordinate (7.89, 54.1, 170.) at (7.89, 54.1, 170.);
\coordinate (33.9, 175., 170.) at (33.9, 175., 170.);
\draw[edge,back] (-71.0, -167., -71.0) -- (0.375, -0.125, 0.375);
\draw[edge,back] (-62.3, -62.8, -125.) -- (0.375, -0.125, 0.375);
\draw[edge,back] (0.375, -0.125, 0.375) -- (48.1, 95.3, 95.8);
\draw[edge,back] (48.1, 95.3, 95.8) -- (48.1, 175., 95.8);
\draw[edge,back] (48.1, 95.3, 95.8) -- (48.1, 70.7, 170.);
\node[vertex] at (0.375, -0.125, 0.375)     {};
\node[vertex] at (48.1, 95.3, 95.8)     {};
\fill[facet] (33.9, 175., 170.) -- (48.1, 175., 170.) -- (48.1, 70.7, 170.) -- (19.2, -6.28, 170.) -- (7.89, 54.1, 170.) -- cycle {};
\fill[facet] (33.9, 175., 170.) -- (-9.00, 175., 19.3) -- (-2.71, 175., -5.80) -- (48.1, 175., 95.8) -- (48.1, 175., 170.) -- cycle {};
\fill[facet] (33.9, 175., 170.) -- (-9.00, 175., 19.3) -- (-81.0, -64.4, -52.6) -- (7.89, 54.1, 170.) -- cycle {};
\fill[facet] (7.89, 54.1, 170.) -- (-81.0, -64.4, -52.6) -- (-104., -189., -146.) -- (-71.0, -167., -71.0) -- (19.2, -6.28, 170.) -- cycle {};
\fill[facet] (-62.3, -62.8, -125.) -- (-104., -189., -146.) -- (-81.0, -64.4, -52.6) -- (-9.00, 175., 19.3) -- (-2.71, 175., -5.80) -- cycle {};
\draw[edge] (-81.0, -64.4, -52.6) -- (-9.00, 175., 19.3);
\draw[edge] (-81.0, -64.4, -52.6) -- (-104., -189., -146.);
\draw[edge] (-81.0, -64.4, -52.6) -- (7.89, 54.1, 170.);
\draw[edge] (-9.00, 175., 19.3) -- (-2.71, 175., -5.80);
\draw[edge] (-9.00, 175., 19.3) -- (33.9, 175., 170.);
\draw[edge] (-71.0, -167., -71.0) -- (-104., -189., -146.);
\draw[edge] (-71.0, -167., -71.0) -- (19.2, -6.28, 170.);
\draw[edge] (-104., -189., -146.) -- (-62.3, -62.8, -125.);
\draw[edge] (-2.71, 175., -5.80) -- (-62.3, -62.8, -125.);
\draw[edge] (-2.71, 175., -5.80) -- (48.1, 175., 95.8);
\draw[edge] (48.1, 175., 95.8) -- (48.1, 175., 170.);
\draw[edge] (48.1, 70.7, 170.) -- (19.2, -6.28, 170.);
\draw[edge] (48.1, 70.7, 170.) -- (48.1, 175., 170.);
\draw[edge] (19.2, -6.28, 170.) -- (7.89, 54.1, 170.);
\draw[edge] (48.1, 175., 170.) -- (33.9, 175., 170.);
\draw[edge] (7.89, 54.1, 170.) -- (33.9, 175., 170.);
\node[vertex] at (-81.0, -64.4, -52.6)     {};
\node[vertex] at (-9.00, 175., 19.3)     {};
\node[vertex] at (-71.0, -167., -71.0)     {};
\node[vertex] at (-104., -189., -146.)     {};
\node[vertex] at (-2.71, 175., -5.80)     {};
\node[vertex] at (-62.3, -62.8, -125.)     {};
\node[vertex] at (48.1, 175., 95.8)     {};
\node[vertex] at (48.1, 70.7, 170.)     {};
\node[vertex] at (19.2, -6.28, 170.)     {};
\node[vertex] at (48.1, 175., 170.)     {};
\node[vertex] at (7.89, 54.1, 170.)     {};
\node[vertex] at (33.9, 175., 170.)     {};
\end{tikzpicture}

%% file: figures/Associahedra/Asso8.tex
\begin{tikzpicture}%
	[x={(-0.367374cm, 0.892104cm)},
	y={(-0.884812cm, -0.422378cm)},
	z={(0.286607cm, -0.160463cm)},
	scale=0.0247500,
	back/.style={loosely dotted, thin},
	edge/.style={color=blue!95!black, thick},
	facet/.style={fill=cyan!95!black,fill opacity=0.500000},
	vertex/.style={inner sep=0pt,circle,draw=green!25!black,fill=green!75!black,thick,anchor=base}]
%
%
\coordinate (-24.4, 4.74, -119.) at (-24.4, 4.74, -119.);
\coordinate (-29.2, -29.2, -264.) at (-29.2, -29.2, -264.);
\coordinate (-55.1, -87.4, -303.) at (-55.1, -87.4, -303.);
\coordinate (-36.7, -73.6, -110.) at (-36.7, -73.6, -110.);
\coordinate (-4.05, -8.26, 151.) at (-4.05, -8.26, 151.);
\coordinate (0.167, 0.167, 0.500) at (0.167, 0.167, 0.500);
\coordinate (20.0, 39.8, 120.) at (20.0, 39.8, 120.);
\coordinate (20.0, 68.1, 34.7) at (20.0, 68.1, 34.7);
\coordinate (20.0, 39.8, 151.) at (20.0, 39.8, 151.);
\coordinate (20.0, 68.1, 151.) at (20.0, 68.1, 151.);
\coordinate (7.27, 68.1, -118.) at (7.27, 68.1, -118.);
\coordinate (2.62, 31.7, 151.) at (2.62, 31.7, 151.);
\coordinate (2.79, 68.1, -64.2) at (2.79, 68.1, -64.2);
\coordinate (16.3, 68.1, 151.) at (16.3, 68.1, 151.);
\draw[edge,back] (-24.4, 4.74, -119.) -- (-55.1, -87.4, -303.);
\draw[edge,back] (-24.4, 4.74, -119.) -- (2.62, 31.7, 151.);
\draw[edge,back] (-24.4, 4.74, -119.) -- (2.79, 68.1, -64.2);
\node[vertex] at (-24.4, 4.74, -119.)     {};
\fill[facet] (16.3, 68.1, 151.) -- (20.0, 68.1, 151.) -- (20.0, 68.1, 34.7) -- (7.27, 68.1, -118.) -- (2.79, 68.1, -64.2) -- cycle {};
\fill[facet] (16.3, 68.1, 151.) -- (20.0, 68.1, 151.) -- (20.0, 39.8, 151.) -- (-4.05, -8.26, 151.) -- (2.62, 31.7, 151.) -- cycle {};
\fill[facet] (0.167, 0.167, 0.500) -- (-29.2, -29.2, -264.) -- (-55.1, -87.4, -303.) -- (-36.7, -73.6, -110.) -- cycle {};
\fill[facet] (20.0, 39.8, 151.) -- (-4.05, -8.26, 151.) -- (-36.7, -73.6, -110.) -- (0.167, 0.167, 0.500) -- (20.0, 39.8, 120.) -- cycle {};
\fill[facet] (7.27, 68.1, -118.) -- (-29.2, -29.2, -264.) -- (0.167, 0.167, 0.500) -- (20.0, 39.8, 120.) -- (20.0, 68.1, 34.7) -- cycle {};
\fill[facet] (20.0, 68.1, 151.) -- (20.0, 68.1, 34.7) -- (20.0, 39.8, 120.) -- (20.0, 39.8, 151.) -- cycle {};
\draw[edge] (-29.2, -29.2, -264.) -- (-55.1, -87.4, -303.);
\draw[edge] (-29.2, -29.2, -264.) -- (0.167, 0.167, 0.500);
\draw[edge] (-29.2, -29.2, -264.) -- (7.27, 68.1, -118.);
\draw[edge] (-55.1, -87.4, -303.) -- (-36.7, -73.6, -110.);
\draw[edge] (-36.7, -73.6, -110.) -- (-4.05, -8.26, 151.);
\draw[edge] (-36.7, -73.6, -110.) -- (0.167, 0.167, 0.500);
\draw[edge] (-4.05, -8.26, 151.) -- (20.0, 39.8, 151.);
\draw[edge] (-4.05, -8.26, 151.) -- (2.62, 31.7, 151.);
\draw[edge] (0.167, 0.167, 0.500) -- (20.0, 39.8, 120.);
\draw[edge] (20.0, 39.8, 120.) -- (20.0, 68.1, 34.7);
\draw[edge] (20.0, 39.8, 120.) -- (20.0, 39.8, 151.);
\draw[edge] (20.0, 68.1, 34.7) -- (20.0, 68.1, 151.);
\draw[edge] (20.0, 68.1, 34.7) -- (7.27, 68.1, -118.);
\draw[edge] (20.0, 39.8, 151.) -- (20.0, 68.1, 151.);
\draw[edge] (20.0, 68.1, 151.) -- (16.3, 68.1, 151.);
\draw[edge] (7.27, 68.1, -118.) -- (2.79, 68.1, -64.2);
\draw[edge] (2.62, 31.7, 151.) -- (16.3, 68.1, 151.);
\draw[edge] (2.79, 68.1, -64.2) -- (16.3, 68.1, 151.);
\node[vertex] at (-29.2, -29.2, -264.)     {};
\node[vertex] at (-55.1, -87.4, -303.)     {};
\node[vertex] at (-36.7, -73.6, -110.)     {};
\node[vertex] at (-4.05, -8.26, 151.)     {};
\node[vertex] at (0.167, 0.167, 0.500)     {};
\node[vertex] at (20.0, 39.8, 120.)     {};
\node[vertex] at (20.0, 68.1, 34.7)     {};
\node[vertex] at (20.0, 39.8, 151.)     {};
\node[vertex] at (20.0, 68.1, 151.)     {};
\node[vertex] at (7.27, 68.1, -118.)     {};
\node[vertex] at (2.62, 31.7, 151.)     {};
\node[vertex] at (2.79, 68.1, -64.2)     {};
\node[vertex] at (16.3, 68.1, 151.)     {};
\end{tikzpicture}

%% file: SubwordComplexFans.bbl
\newcommand{\etalchar}[1]{$^{#1}$}
\begin{thebibliography}{AHBC{\etalchar{+}}14}

\bibitem[AHBC{\etalchar{+}}14]{postnikov_sign_physics_2012}
Nima Arkani-Hamed, Jacob~L. Bourjaily, Freddy Cachazo, Alexander~B. Goncharov,
  Alexander Postnikov, and Jaroslav Trnka.
\newblock Scattering amplitudes and the positive {G}rassmannian.
\newblock {\em \tt arXiv:1212.5605v2}, preprint:158 pages, (March) 2014.

\bibitem[AS76]{altshuler_enumeration_1976}
Amos Altshuler and Leon Steinberg.
\newblock An enumeration of combinatorial 3-manifolds with nine vertices.
\newblock {\em Discrete Math.}, 16(2):91--108, 1976.

\bibitem[BB05]{bjorner_combinatorics_2005}
Anders Bj\"orner and Francesco Brenti.
\newblock {\em Combinatorics of {C}oxeter groups}, volume 231 of {\em GTM}.
\newblock Springer, 2005.

\bibitem[BP09]{bokowski_symmetric_2009}
J\"urgen Bokowski and Vincent Pilaud.
\newblock On symmetric realizations of the simplicial complex of
  3-crossing-free sets of diagonals of the octagon.
\newblock In {\em Proceedings of the 21st Canadian Conference on Computational
  Geometry {(CCCG2009)}}, pages 41--44, 2009.

\bibitem[Ceb12]{ceballos_associahedra_2012}
Cesar Ceballos.
\newblock {\em On associahedra and related topics}.
\newblock PhD thesis, Freie Universit\"at Berlin, Berlin, 2012.
\newblock Available at
  \url{http://www.diss.fu-berlin.de/diss/receive/FUDISS_thesis_000000039026}.

\bibitem[CLS14]{ceballos_subword_2013}
Cesar Ceballos, Jean-Philippe Labb{\'e}, and Christian Stump.
\newblock Subword complexes, cluster complexes, and generalized
  multi-associahedra.
\newblock {\em J. Algebraic Combin.}, 39(1):17--51, 2014.

\bibitem[CSZ14]{ceballos_associahedra}
Cesar Ceballos, Francisco Santos, and G\"unter~M. Ziegler.
\newblock Many non-equivalent realizations of the associahedron.
\newblock {\em Combinatorica}, 2014.
\newblock To appear.

\bibitem[CZ12]{ceballos_realizing_2012}
Cesar Ceballos and G\"unter~M. Ziegler.
\newblock Realizing the associahedron: Mysteries and questions.
\newblock In Folkert M\"uller-Hoissen, Jean~Marcel Pallo, and Jim Stasheff,
  editors, {\em Associahedra, Tamari Lattices and Related Structures}, number
  299 in Progress in Mathematics, pages 119--127. Springer Basel, January 2012.

\bibitem[DLRS10]{de_loera_triangulations_2010}
Jes{\'u}s~A. De~Loera, J{\"o}rg Rambau, and Francisco Santos.
\newblock {\em Triangulations}, volume~25 of {\em Algorithms and Computation in
  Mathematics}.
\newblock Springer-Verlag, Berlin, 2010.

\bibitem[{Eln}97]{elnitsky_rhombic_1997}
Serge {Elnitsky}.
\newblock {Rhombic tilings of polygons and classes of reduced words in Coxeter
  groups.}
\newblock {\em {J. Comb. Theory, Ser. A}}, 77(2):193--221, 1997.

\bibitem[FW00]{felsner_theorem_2000}
S.~{Felsner} and H.~{Weil}.
\newblock {A theorem on higher Bruhat orders.}
\newblock {\em {Discrete Comput. Geom.}}, 23(1):121--127, 2000.

\bibitem[HL07]{HohlwegLange}
Christophe Hohlweg and Carsten~E.~M.~C. Lange.
\newblock Realizations of the associahedron and cyclohedron.
\newblock {\em Discrete Comput.~Geom.}, 37(4):517--543, 2007.

\bibitem[Hum92]{humphreys_reflection_1992}
James~E. Humphreys.
\newblock {\em Reflection groups and {C}oxeter groups}.
\newblock Cambridge Studies in Advanced Mathematics. 29 (Cambridge University
  Press), 1992.

\bibitem[Jon05]{jonsson_generalized_2005}
Jakob Jonsson.
\newblock Generalized triangulations and diagonal-free subsets of stack
  polyominoes.
\newblock {\em J. Comb. Theory, Ser.~A}, 112(1):117--142, 2005.

\bibitem[KM04]{KnutsonMiller-subwordComplex}
Allen Knutson and Ezra Miller.
\newblock Subword complexes in {C}oxeter groups.
\newblock {\em Adv.~Math.}, 184(1):161--176, 2004.

\bibitem[KM05]{knutson_grobner_2005}
Allen Knutson and Ezra Miller.
\newblock {Gr\"obner geometry of Schubert polynomials}.
\newblock {\em Ann. Math. (2)}, 161(3):1245--1318, 2005.

\bibitem[MHPS12]{mueller_associahedra_2012}
Folkert M{\"u}ller-Hoissen, Jean~Marcel Pallo, and Jim Stasheff, editors.
\newblock {\em Associahedra, {T}amari lattices and related structures}, volume
  299 of {\em Progress in Mathematics}.
\newblock Birkh{\"a}user Boston Inc., Boston, MA, 2012.

\bibitem[MS89]{manin_arrangement_1989}
Yu.I. {Manin} and V.V. {Shekhtman}.
\newblock {Arrangements of hyperplanes, higher braid groups and higher Bruhat
  orders.}
\newblock In {\em {Explicit universal deformations of Galois representations}},
  volume~17 of {\em Adv. Stud. Pure Math.}, pages 289--308. Academic Press,
  Boston, 1989.

\bibitem[PP12]{PilaudPocchiola}
Vincent Pilaud and Michel Pocchiola.
\newblock Multitriangulations, pseudotriangulations and primitive sorting
  networks.
\newblock {\em Discrete Comput. Geom.}, 48(1):142--191, 2012.

\bibitem[PS09]{PilaudSantos-multitriangulations}
Vincent Pilaud and Francisco Santos.
\newblock Multitriangulations as complexes of star polygons.
\newblock {\em Discrete~Comput.~Geom.}, 41(2):284--317, 2009.

\bibitem[PS11]{pilaud_brick_2011}
Vincent Pilaud and Christian Stump.
\newblock Brick polytopes of spherical subword complexes: A new approach to
  generalized associahedra.
\newblock {\em {\tt arXiv:1111.3349v3}}, preprint:52 pages, 2011.

\bibitem[PS12]{PilaudSantos-brickPolytope}
Vincent Pilaud and Francisco Santos.
\newblock The brick polytope of a sorting network.
\newblock {\em European~J.~Combin.}, 33(4):632--662, 2012.

\bibitem[RG96]{richter-gebert_realization_1996}
J\"urgen Richter-Gebert.
\newblock {\em Realization spaces of polytopes}.
\newblock Springer, Berlin; New York, 1996.

\bibitem[RR13]{reiner_diameter_2013}
Victor Reiner and Yuval Roichman.
\newblock Diameter of graphs of reduced words and galleries.
\newblock {\em Transactions of the American Mathematical Society},
  365(5):2779--2802, 2013.

\bibitem[RSS03]{RoteSantosStreinu-pseudotriangulationPolytope}
G{\"u}nter Rote, Francisco Santos, and Ileana Streinu.
\newblock Expansive motions and the polytope of pointed pseudo-triangulations.
\newblock In {\em Discrete and computational geometry}, volume~25 of {\em
  Algorithms Combin.}, pages 699--736. Springer, Berlin, 2003.

\bibitem[S{\etalchar{+}}14]{sage}
William~A. Stein et~al.
\newblock {\em {S}age {M}athematics {S}oftware ({V}ersion 6.1.1)}.
\newblock The Sage Development Team, 2014.
\newblock {\tt http://www.sagemath.org}.

\bibitem[SS12]{serrano_maximal_2012}
Luis Serrano and Christian Stump.
\newblock Maximal fillings of moon polyominoes, simplicial complexes, and
  {S}chubert polynomials.
\newblock {\em Electron. J. Combin.}, 19(1):P16, January 2012.

\bibitem[SSV97]{shapiro_connected_1997}
B.~{Shapiro}, M.~{Shapiro}, and A.~{Vainshtein}.
\newblock {Connected components in the intersection of two open opposite
  Schubert cells in $SL_n(\mathbb R)/B$.}
\newblock {\em {Int. Math. Res. Not.}}, 1997(10):469--493, 1997.

\bibitem[Stu11]{Stump}
Christian Stump.
\newblock A new perspective on $k$-triangulations.
\newblock {\em J. Combin. Theory Ser. A}, 118(6):1794--1800, 2011.

\bibitem[Tit69]{tits_probleme_1969}
Jacques Tits.
\newblock Le probl\`eme des mots dans les groupes de {C}oxeter.
\newblock In {\em Symposia {M}athematica ({INDAM}, {R}ome, 1967/68), {V}ol. 1},
  pages 175--185. Academic Press, London, 1969.

\bibitem[{Zie}93]{ziegler_higher_1993}
G\"unter~M. {Ziegler}.
\newblock {Higher Bruhat orders and cyclic hyperplane arrangements.}
\newblock {\em {Topology}}, 32(2):259--279, 1993.

\end{thebibliography}
